\documentclass[11pt]{amsart}

\usepackage{verbatim, amssymb,hyperref}

\usepackage{color}
\usepackage{bbm}
\usepackage{graphicx}
\usepackage{subcaption}
\usepackage{todonotes}
\usepackage[english]{cleveref}

\newcommand{\eps}{\varepsilon}

\newcommand{\cD}{\mathcal D}

\newcommand{\cO}{\mathcal O}

\renewcommand{\P}{{\mathbb P}}

\newcommand{\N}{{\mathbb N}}

\newcommand{\E}{{\mathbb E}}

\newcommand{\R}{{\mathbb R}}

\theoremstyle{plain}

\newtheorem{theorem}{Theorem}[section]

\newtheorem{lemma}[theorem]{Lemma}
\newtheorem{corollary}[theorem]{Corollary}
\newtheorem{definition}[theorem]{Definition}

\theoremstyle{definition}
\newtheorem{remark}[theorem]{Remark}

\newtheorem{assumption}[theorem]{Assumption}

\definecolor{sk}{rgb}{0.8,0.1,0.1}

\definecolor{am}{rgb}{0.1,0.1,0.8}

\title[Strong Averaging Principle and Long-Time Dynamics for Fast-Slow SDEs]%
{Strong Averaging Principle and Long-Time Dynamics for Fast-Slow SDEs with Increasing Time-Scale Separation and Degenerate Noise}
\author[]%[Kassing]
{Sebastian Kassing}
\address{Sebastian Kassing\\
	Department of Mathematics \& Informatics,  
  University of Wuppertal,
    42119 Wuppertal, Germany}
\email{kassing@uni-wuppertal.de}

\author[]%[Miwa]
{Asuto Miwa}
\address{Asuto Miwa\\
	Department of Mathematics \& Informatics,  
  University of Wuppertal,
    42119 Wuppertal, Germany}
\email{amiwa@uni-wuppertal.de}


\begin{document}

\begin{abstract}
	We establish a strong averaging principle for fast-slow stochastic differential equations with a time-dependent scale-separation parameter $(\varepsilon_t)_{t \geq 0}$ satisfying $\eps_t \to 0$ as $t \to \infty$. 
    In contrast to approaches based on noise-induced smoothing or elliptic regularity, our approach relies on dissipativity of the frozen fast dynamics and therefore permits degenerate diffusion coefficients. We prove a maximal $L^p$-estimate between the slow variable and the averaged ODE at late times, with the classical strong convergence rate of order $1/2$. Under an additional decay condition on $(\eps_t)_{t \ge 0}$, this estimate implies that the slow variable is almost surely an asymptotic pseudo-trajectory of the averaged ODE. As a consequence, we obtain criteria for the identification of possible limit points and for convergence toward asymptotically stable equilibria for the slow variable by analyzing the dynamical behavior of the averaged equation. 
\end{abstract}

\maketitle

\section{Introduction}

Various real-world phenomena in finance, climate science, optimization, and natural sciences can be modeled by systems whose components evolve on different time-scales. 
Mathematically, these phenomena are often modeled using so-called \emph{fast-slow systems}. In the stochastic setting, these processes are typically of the form
\begin{align*}
    \begin{cases}
    d X_t^\varepsilon = f(X_t^\varepsilon,Y_t^\varepsilon) dt+ \sigma_1(X_t^\varepsilon,Y_t^\varepsilon)  dW_t^{(1)}, &X_0^\varepsilon=x_0 \in \R^{d_1},\\
    d Y_t^\varepsilon = \frac{1}{\varepsilon}g(X_t^\varepsilon,Y_t^\varepsilon)dt+ \frac{1}{\sqrt{\varepsilon}} \sigma_2(X_t^\varepsilon,Y_t^\varepsilon)  d W_t^{(2)}, &Y_0^\varepsilon=y_0 \in \R^{d_2},
    \end{cases}
\end{align*}
where $f:\R^{d_1+d_2} \to \R^{d_1}$, $g: \R^{d_1+d_2} \to \R^{d_2}$, $\sigma_1: \R^{d_1+d_2} \to \R^{d_1 \times d_1} $, $\sigma_2: \R^{d_1+d_2} \to \R^{d_2 \times d_2}$, $(W_t^{(1)})_{t \ge 0}$ and $(W_t^{(2)})_{t \ge 0}$ are independent Brownian motions, and $\eps>0$ is a small parameter that controls the time-scale separation between the slow variable $(X_t^{\eps})_{t \ge 0}$ and the fast variable $(Y_t^\eps)_{t \ge 0}$. Decreasing $\varepsilon$ corresponds to a larger time-scale separation.
See, e.g. \cite{FreidlinMarkI2012Rpod, berglund2006noise, pavliotis2008multiscale} for a detailed introduction to stochastic fast-slow systems.

As $\eps \to 0$, the dynamics of the fast variable can asymptotically be described by the so-called \emph{frozen SDE}
\begin{align} \label{eq:frozenintro}
    d \hat{Y}_t^x = g(x,\hat{Y}^x_t)dt+ \sigma_2(x,\hat{Y}^x_t) d\hat W_t^{(2)}
\end{align}
fixed at the current position $x \in \R^{d_1}$ of the slow variable, where $(\hat W_t^{(2)})_{t \ge 0}$ denotes a Brownian motion on the fast time-scale. Assuming ergodicity of the frozen SDE for all $x \in \R^{d_1}$, the dynamics of the slow variable can, in turn, asymptotically be described by the averaged SDE
\begin{align*}
    d\bar X_t = \bar{f}(\bar X_t)dt+ \bar a^{1/2}(\bar X_t) dW_t^{(1)},
\end{align*}
where 
$$
    \bar f(x)= \int f(x,y) \, d\mu_x(y) \quad \text{ and } \quad  \bar a(x)= \int \sigma_1(x,y)\sigma_1(x,y)^\top \, d\mu_x(y)
$$
denote the coefficients $f$ and $\sigma_1$ integrated against the invariant measure $\mu_x$ of the frozen SDE \eqref{eq:frozenintro}.
The first result on this averaging principle goes back to Khasminskii \cite{Khasminskij1968OnTP}, where weak convergence was established. Convergence in probability is shown in \cite[Chapter 7.9]{FreidlinMarkI2012Rpod} under the assumption that the diffusion coefficient $\sigma_1$ depends only on the slow variable $X^\varepsilon$, while \cite[Chapter 7.2]{FreidlinMarkI2012Rpod} proves almost sure convergence to the averaged system under a different set of assumptions, including the condition $\sigma_1=0$.  
See also \cite{cerrai2009averaging, liu2010strong, BrehierCharles-Edouard2020Ooci,
FuHongbo2015Scia, de_Feo_2023, sun2024averaging} and the references mentioned therein for results deriving an approximation error between the slow variable $(X_t^\eps)_{t\in[0,T]}$ and the
solution $(\bar X_t)_{t\in[0,T]}$ of the averaged equation on compact time intervals.
For example, in \cite{de_Feo_2023}, $L^2$-convergence is shown to hold with a convergence rate of $\mathcal{O}(\sqrt{\varepsilon})$ for a fast-slow SDE defined on an infinite-dimensional Hilbert space. The same rate of convergence  is also obtained, for example, in  \cite{BrehierCharles-Edouard2020Ooci} or \cite{FuHongbo2015Scia}.
Moreover, central limit theorems are derived, e.g., in \cite{CerraiSandra2009Ndft, KellyDavid2017Fith}, and large deviation principles are proved in \cite{BouchetFreddy2016LDiF}. See also \cite{akyildiz2024multiscale} for a uniform-in-time weak averaging principle under strong contraction property of the averaged drift. 

Recently, the theory of multi-timescale dynamical systems has been applied to stochastic optimization algorithms with multiple hyperparameters; see e.g. \cite{borkar1997stochastic, tadic2004almost, borkar2008stochastic, borkar2025stochastic} for a general framework in optimization and \cite{dereich2024convergence, dereich2025ode} for an analysis of the Adam algorithm in a fast-slow scaling regime. The core idea is to provide a transparent and theoretically grounded basis for hyperparameter selection by analyzing an asymptotic regime in which one set of hyperparameters is chosen significantly smaller compared to others. 
For example, recent works suggest applying the Adam algorithm with the friction parameters being fixed while the learning rate follows a schedule, typically decaying over time; see e.g. \cite{wilson2017marginal, loshchilov2018decoupled,dereichetal2024nonconvergence}. 

Another application is the optimization of a cost functional depending on the invariant distribution of a parameterized family of Markov processes, e.g. in energy-based models \cite{nijkamp2020anatomy}. 
In this setting, one approach is to gradually increase the observation time of the Markov process to decrease the bias and guarantee convergence. In Section~\ref{sec:applications}, we will discuss these applications in more detail.

Motivated by the above problem classes, in this work, we derive an averaging principle for a fast-slow SDE with a time-dependent parameter $(\varepsilon_t)_{t \geq 0}$ that satisfies $\eps_t \to 0$ as $t \to \infty$. Hence, a convergence of the slow variable to the averaged trajectory does not happen for a family of SDEs parametrized by $\varepsilon \in (0,1]$ over a fixed time horizon, but rather for a solution to a single SDE after shifting the comparison interval to infinity. 

We derive a strong $L^p$-approximation error and prove almost sure convergence for the slow variable to the averaged ODE on compact intervals started at late times. Afterwards, we prove convergence statements for the slow variable as $t \to \infty$ by analyzing the asymptotic behavior of the averaged ODE. Our approach follows the ODE method and uses the notion of asymptotic pseudo-trajectories, which is a common technique in optimization for the asymptotic analysis of discrete-time algorithms; see e.g. \cite{ljung1977analysis, benveniste1990adaptive, benaim1996dynamical, benaim2006dynamics}.

Let us fix the mathematical framework. Let $d_1,d_2 \in \N$, let $(\Omega, \mathcal F, (\mathcal F_t)_{t \ge 0}, \P)$ be a filtered probability space satisfying the usual conditions, and let $(W_t)_{t \ge 0}$ be a $d_2$-dimensional $(\mathcal F_t)$-Brownian motion.
We consider the \textit{fast-slow SDE}
\begin{align}\label{eq:fastslow0}
\begin{cases}
dX_t = f(X_t,Y_t) dt, \\ 
dY_t =  \frac{1}{\varepsilon_t}g(X_t,Y_t)dt + \frac{1}{\sqrt{\varepsilon_t}}\sigma(X_t,Y_t)  dW_t, 
\end{cases}
\end{align}
where $f \colon \R^{d_1 + d_2} \to \R^{d_1}$, $g \colon \R^{d_1 + d_2} \to \R^{d_2}$, $\sigma: \R^{d_1+d_2} \to \R^{d_2\times d_2}$, and $(\eps_t)_{t \ge 0}$ is a strictly positive function.

Next, we present our assumptions on the coefficients $f,g$ and $\sigma$, as well as on the time-separation parameter $(\eps_t)_{t \ge 0}$. 
Subsequently, we denote by $|v|$ the Euclidean norm of a vector $v$, by $|A|$ and $|A|_F$ the induced operator norm and the Frobenius norm of a matrix $A$, respectively, and by $|B|$ the bilinear operator norm of a bilinear function $B$. The vector space corresponding to the norm will be clear from the context.

A standard approach for fast-slow SDEs is to rely on the smoothing property of elliptic SDEs, see e.g. \cite{Veretennikov_1991, pardoux2001poisson, pavliotis2008multiscale, roeckner2021diffusion}. However, in practice the noise is often degenerate due to invariances in the parametrization, e.g. for neural networks, or due to multiple equations being driven by the same noise, e.g. for the Adam algorithm \cite{kingma2014adam}. In the absence of ellipticity assumptions, such as a validation of the Hörmander condition, a key step to obtain an averaging principle is to prove regularity results for the Markov semigroup of the frozen fast SDE. This motivates the following set of assumptions. 

\begin{assumption} \label{assu:regularity}
    The drift and diffusion coefficients $f \colon \R^{d_1+d_2}\to\R^{d_1}$, $g \colon \R^{d_1+d_2}\to\R^{d_2}$, and $\sigma = (\sigma_1,\dots,\sigma_{d_2}) \colon \R^{d_1+d_2}\to\R^{d_2 \times d_2}$ satisfy the following conditions:\footnote{Although (i) follows from the uniform boundedness of the derivatives imposed in (iii), we state these conditions separately to distinguish the assumptions required for well-posedness, moment bounds, and ergodicity from the smoothness assumptions needed for the regularity of the Poisson corrector.}
    \begin{enumerate}
        \item[(i)] $f$, $g$, $\sigma$ are globally Lipschitz, and in particular, show linear growth in $y$.
        \item[(ii)] $f$, $g$, $\sigma$ are uniformly bounded at $y=0$, i.e. 
        \begin{align*}
            \sup_{x \in \R^{d_1}} (|f(x,0)| + |g(x,0)| + |\sigma(x,0)|) <\infty.
        \end{align*}
        \item[(iii)]
        $f$, $g$, $\sigma_1, \dots, \sigma_{d_2}$ are twice continuously differentiable in $y$ and once continuously differentiable in $x$, all the derivatives below being jointly continuous on $\R^{d_1+d_2}$, and 
        \begin{align*} 
         \sup_{(x,y) \in \R^{d_1+d_2}} \Big(&|D_x f(x,y)| 
            + |D_y f(x,y)| % für D_y Phi
            + |D_{y,x}^2 f(x,y)| 
            + |D_y^2 f(x,y)|  \\
            &+|D_x g(x,y)| % für D_x Phi 
            + |D_y g(x,y)| % für D_y Phi 
            + |D_{y,x}^2g(x,y)| % für D_x Phi 
            + |D_y^2g(x,y)| % für D_x Phi 
              \\
              &+ \sup_{k=1, \dots, d_2} \Big( |D_x \sigma_k(x,y)|
            + |D_y \sigma_k(x,y)|+ |D_{y,x}^2 \sigma_k(x,y)| % für D^2_y Phi
            + |D_y^2 \sigma_k(x,y)| \Big) \Big)
            <\infty.
        \end{align*}
    \end{enumerate}
\end{assumption}

\begin{assumption} \label{assu:epsilon}
    $(\varepsilon_t)_{t \geq 0}$ is a non-increasing, continuous function satisfying $\varepsilon_t \to 0$ for $t \to \infty$.
\end{assumption}

The following dissipativity assumption guarantees ergodicity of the frozen fast SDE, as well as a contraction property of the corresponding semigroup.

\begin{definition} \label{def:fastdissi}
    Let $p \geq 2$ and $\beta > 0$. We say that the fast variable $(Y_t)_{t \ge 0}$ is \emph{$p$-dissipative (with constant $\beta$)}, if for all $x\in \R^{d_1}$ and $y_1,y_2 \in \R^{d_2}$
\begin{align*}
    \langle y_1 - y_2, g(x,y_1)-g(x,y_2) \rangle + \frac{p-1}{2} |\sigma(x,y_1)-\sigma(x,y_2)|_F^2 \leq -\beta|y_1-y_2|^2.
\end{align*}
\end{definition}

\begin{remark} \label{rem:dissibigger}
    Under Assumption~\ref{assu:regularity}, $p$-dissipativity implies $p+\eps$-dissipativity for an $\eps >0$. In fact, assume that $\sigma$ is Lipschitz continuous with Lipschitz constant $L_\sigma>0$. If $Y$ is $p$-dissipative with constant $\beta$, then for every $q\ge p$,
\begin{align*}
\langle y_1-y_2,g(x,y_1)-g(x,y_2)\rangle
+\frac{q-1}{2}|\sigma(x,y_1)-\sigma(x,y_2)|_F^2 \le
-\Big(\beta-\frac{q-p}{2}L_\sigma^2\Big)|y_1-y_2|^2.
\end{align*}
Hence, whenever
$
q \in (p, p+\frac{2\beta}{L_\sigma^2}),
$
the process is $q$-dissipative, with constant
$
\beta_q=\beta-\frac{q-p}{2}L_\sigma^2>0.
$
\end{remark}

On another filtered probability space $(\hat{\Omega}, \hat{\mathcal{F}}, (\hat{\mathcal{F}}_t)_{t \geq 0}, \hat{\P})$ satisfying the usual conditions with $d_2$-dimensional Brownian motion $(\hat{W}_t)_{t \ge 0}$, for $x \in \R^{d_1}$ and $y \in \R^{d_2}$, we can define the \textit{frozen fast SDE}
\begin{align}\label{eq:frozen0}
d\hat{Y}^{x,y}_t = g(x,\hat{Y}^{x,y}_t)dt + \sigma(x, \hat{Y}_t^{x,y})  d\hat{W}_t
\end{align}
with initial condition $\hat{Y}^{x,y}_0=y$.
This describes the dynamics of the fast variable in the large time-separation limit, where the slow variable is replaced by a fixed frozen value $x$.

Under Assumption~\ref{assu:regularity}, for every $x \in\R^{d_1}$ there exists a unique strong solution to the SDE \eqref{eq:frozen0}. 
Moreover, if the fast variable $(Y_t)_{t \ge 0}$ is $p$-dissipative for a $p \ge 2$, the solution to \eqref{eq:frozen0} is an ergodic Markov process with a unique invariant distribution $\mu_x$ for each $x \in \R^{d_1}$; see Lemma~\ref{lem:invariantmoment}. This allows us to define the \textit{averaged function}
\begin{align}\label{eq:averagedF}
\bar{f}\colon \R^{d_1}\to \R^{d_1}\quad ,\quad \bar{f}(x) = \int f(x,y)\mu_x(dy)
\end{align}
and the \textit{averaged ODE}
\begin{align}\label{eq:averaged0}
\begin{cases}
\dot{\mathbf x}_t(x) = \bar{f}(\mathbf{x}_t(x)),\\
\mathbf{x}_0(x) = x
\end{cases}   
\end{align}
for any initial condition $x \in \R^{d_1}$. By the Lipschitz continuity of $\bar f$, see Lemma~\ref{lem:barfLipschitz}, there exists a unique global solution to the ODE \eqref{eq:averaged0}.

We are now in the position to state the main result of this article. 

\begin{theorem}\label{th:convergenceLp}
Let $p \geq 2$, $x \in \R^{d_1}$ and $y \in \R^{d_2}$. Assume that Assumptions~\ref{assu:regularity} and~\ref{assu:epsilon} hold and let $(X_t,Y_t)_{t\ge 0}$ be the unique solution of \eqref{eq:fastslow0} with initial condition $(X_0,Y_0)=(x,y)$. If the fast variable is $p$-dissipative then, for every $T > 0$, there exists a constant $C \ge 0$ such that for all $T_0 \ge 0$
\[
\E \bigg[ \sup_{t \in [T_0,T_0+T]}|X_t -\mathbf{x}_{t-T_0}(X_{T_0})|^p \bigg] \le C \eps_{T_0}^{p/2}
\]
where, for $\tilde x \in \R^{d_1}$, $(\mathbf x_t(\tilde x))_{t \ge 0}$ denotes the solution of the averaged ODE \eqref{eq:averaged0}.
\end{theorem}

\section{Asymptotic Pseudo-Trajectories and Convergence of the Slow Variable}

Our main result, Theorem~\ref{th:convergenceLp}, has strong consequences for the asymptotic behavior of the slow variable $(X_t)_{t \ge 0}$ as $t \to \infty$. 
In fact, one can deduce convergence results for $(X_t)_{t \ge 0}$ by analyzing the limiting semiflow \eqref{eq:averaged0}. In the optimization literature, this technique is called \emph{ODE method} and builds on the notion of asymptotic pseudo-trajectories.

\begin{definition}
	Let $(y_t)_{t \in [0,\infty)}$ be a curve in $\R^{d_1}$. We say that $(y_t)_{t \in [0,\infty)}$ is an \emph{asymptotic pseudo-trajectory} of the averaged ODE \eqref{eq:averaged0} if and only if it holds 
	for all $ T \in (0,\infty) $ that
	\begin{align} \label{eq:349h93hge}
	\lim_{T_0 \to \infty} \sup_{t \in [T_0,T_0+T]} |y_t-\mathbf x_{t-T_0}(y_{T_0})|=0.
	\end{align}
\end{definition}

Combining Theorem~\ref{th:convergenceLp} with a Borel-Cantelli argument, one can show that, if $(\eps_t)_{t \ge 0}$ vanishes sufficiently fast, the slow variable $(X_t)_{t \ge 0}$ is almost surely an asymptotic pseudo-trajectory of the averaged ODE \eqref{eq:averaged0}.

\begin{corollary} \label{cor:APT}
    Let $p \geq 2$, $x \in \R^{d_1}$ and $y \in \R^{d_2}$. Assume that Assumptions~\ref{assu:regularity} and~\ref{assu:epsilon} hold and let $(X_t,Y_t)_{t\ge 0}$ be the unique solution of \eqref{eq:fastslow0} with initial condition $(X_0,Y_0)=(x,y)$. If the fast variable is $p$-dissipative and $\eps_t = O(t^{-q})$ for a $q \in (2/p, \infty)$ then $(X_t)_{t \ge 0}$ is almost surely an asymptotic pseudo-trajectory of the averaged ODE \eqref{eq:averaged0}.
\end{corollary}

\begin{proof}
    Let $T > 0$ and $\kappa >0 $.
    By the Markov inequality and Theorem~\ref{th:convergenceLp}, there exists a constant $C \ge 0$ such that
    \begin{align*}
    &\sum_{N \in \N_0} \P\Big(\sup_{t \in [NT,(N+2)T]} |X_t- \mathbf{x}_{t-NT}(X_{NT})| > \kappa\Big) \\
    &\leq \sum_{N \in \N_0} \kappa^{-p} \E \Big[\sup_{t \in [NT,(N+2)T]} |X_t- \mathbf{x}_{t-NT}(X_{NT})|^p\Big] \\
    & \le C \kappa^{-p}  \sum_{N \in \N_0} \eps_{NT}^{p/2} <\infty.
    \end{align*}
    Thus, we get by the Borel-Cantelli lemma that
\begin{align*}
    \P\left( \limsup_{N \to \infty}\sup_{t \in [NT,(N+2)T]} |X_t- \mathbf{x}_{t-NT}(X_{NT})| > \kappa \right) = 0.
\end{align*}
Since $\kappa$ can be chosen arbitrarily small, it follows that
\begin{align} \label{eq:2349856293212}
    \limsup_{N\to\infty}\sup_{t \in [NT,(N+2)T]} |X_t- \mathbf{x}_{t-NT}(X_{NT})| =0, \quad \text{ almost surely.}
\end{align}

Now, let $T_0>0$ and set $N = \lfloor T_0/T \rfloor$. Then $[T_0,T_0+T] \subset [NT, (N+2)T]$. By the Lipschitz continuity of $\bar f$ (see Lemma~\ref{lem:barfLipschitz}), one has for every $t \in [T_0,T_0+T]$
\begin{align*}
    |X_t-\mathbf x_{t-T_0}(X_{T_0})| &\le |X_t-\mathbf x_{t-NT}(X_{NT})| + |\mathbf x_{t-T_0}(\mathbf x_{T_0-NT}(X_{NT}))-\mathbf x_{t-T_0}(X_{T_0})| \\
    & \le (1+\exp(L_{\bar f} T)) \sup_{s \in [NT,(N+2)T]} |X_s- \mathbf{x}_{s-NT}(X_{NT})|,
\end{align*}
so that there exists an almost sure event $\Omega_T$ such that \eqref{eq:349h93hge} holds. Hence, $(X_t(\omega))_{t \ge 0}$ is an asymptotic pseudo-trajectory of the averaged ODE \eqref{eq:averaged0} for all $\omega$ in the almost sure event $\bigcap_{T \in \N} \Omega_T$.
\end{proof}

Let us mention a few immediate implications of Theorem~\ref{th:convergenceLp} and Corollary~\ref{cor:APT}.
First, we show that, if $(X_t)_{t \ge 0}$ converges, its limit point is a zero of the averaged vector field $\bar f$. 

\begin{corollary}
\label{cor:criticalpoint}
 Let $x \in \R^{d_1}$ and $y \in \R^{d_2}$. Assume that Assumptions~\ref{assu:regularity} and~\ref{assu:epsilon} hold and let $(X_t,Y_t)_{t\ge 0}$ be the unique solution of \eqref{eq:fastslow0} with initial condition $(X_0,Y_0)=(x,y)$. Assume that the fast variable is $2$-dissipative.
Let $X_\infty $ be an $ \R^{d_1} $-valued random variable 
and let $\Omega_0$ be an event with $\lim_{t\to\infty} X_t=X_\infty$, in probability, on $\Omega_0$, i.e., 
assume for all $ \delta \in (0,\infty) $ that 
$$
\textstyle 
  \lim_{t\to\infty} \P\bigl(\{|X_t-X_\infty|\ge \delta\}\cap \Omega_0\bigr)=0.
$$
Then 
$$
  \P\bigl(\{\bar f(X_\infty)\not=0\}\cap\Omega_0\bigr) = 0 .
$$
\end{corollary}

\begin{proof}
	By Theorem~\ref{th:convergenceLp}, one has for all $T>0$ that
	\begin{align*}
	\lim_{T_0 \to \infty} \sup_{t \in [T_0,T_0+T]}|X_t -\mathbf{x}_{t-T_0}(X_{T_0})| =0,\text{ \ in probability}.
	\end{align*}
	Moreover, for all $T >0$ the mapping $x \mapsto \mathbf x_T(x)$ is continuous, since $\bar f$ is Lipschitz; see Lemma~\ref{lem:barfLipschitz}. Thus,
	$$
	\P(\{\mathbf x_T (X_\infty)\not= X_\infty \}\cap \Omega_0)=0.
	$$
	Since $\mathbb Q\cap(0,\infty)$ is countable, the statement is also true when taking the union of all these events over $ T \in \mathbb Q\cap(0,\infty)$. By continuity of $\mathbf x_{ ( \cdot ) }(x)$ for every fixed $x \in \R^{d_1}$, we can conclude that
	$$
	\P\Bigl(
	  \Bigl( 
  	  \bigcup_{T>0} \{\mathbf x_T(X_\infty)\not= X_\infty\}
  	  \Bigr) \cap \Omega_0\Bigr)=0.
	$$
	This shows that, on the event $\Omega_0$, when started in~$X_\infty$ the flow is almost surely constant. Thus, $\bar f(X_\infty)=0$ almost surely on $\Omega_0$.
\end{proof}

Next, we establish convergence to an asymptotically stable point of the averaged vector field, 
assuming that  $(X_t)_{t \ge 0}$ enters its basin of attraction at late times.

\begin{corollary}
    Let $p \geq 2$, $x \in \R^{d_1}$ and $y \in \R^{d_2}$. Assume that Assumptions~\ref{assu:regularity} and~\ref{assu:epsilon} hold and let $(X_t,Y_t)_{t\ge 0}$ be the unique solution of \eqref{eq:fastslow0} with initial condition $(X_0,Y_0)=(x,y)$. Assume that the fast variable is $p$-dissipative and $\eps_t = O(t^{-q})$ for a $q \in (2/p, \infty)$.
	Let $\mathcal D \subset \R^{d_1}$ be a closed set, let $x^* \in \operatorname{int}(\mathcal D)$ satisfy 
	$$
	\lim_{t \to \infty} \mathbf x_t(\theta)=x^*, \quad \text{ uniformly in } \theta \in \mathcal D.
	$$
	Then, one has almost surely on the event
	$
	\bigcap_{T \ge 0} \bigcup_{t \ge T}\{X_t \in \mathcal D \} 
	$
	that
	$$
	\lim_{t \to \infty} X_t =x^*.
	$$
\end{corollary}

\begin{proof}
	Let $\delta\in(0,\infty)$ be such that $\overline {B_{2\delta}(x^\ast)}\subset \cD$. The statement follows once we showed that almost surely, on the event
	$\bigcap_{T \ge 0} \bigcup_{t \ge T}\{X_t \in \mathcal D \}$ one has that $(X_t)_{t \ge 0}$ takes values in $\overline{B_{\delta}(x^\ast)}$ from a random time $\tau$ onward. Applying this argument iteratively for $\delta$ replaced by $(\delta 2^{-n})_{n \in \N}$ then proves convergence.
	
	For this, choose $T>0$ such that for all $\theta\in\cD$ and $t\in[T,\infty)$ one has $\mathbf x_t(\theta)\in B_{\delta/2}(x^\ast)$. By Corollary~\ref{cor:APT}, for almost all $\omega \in \Omega$ there exists $\tau(\omega) \in [0,\infty)$ such that for all $t \ge \tau(\omega)$
	$$
	\sup_{s \in [t,t+2T]}|X_s(\omega) -\mathbf{x}_{s-t}(X_{t}(\omega))|\le \frac{\delta}{2}.
	$$
	Now, whenever $X_t(\omega)\in\cD$ for a $t \ge \tau(\omega)$, one has
    $X_s(\omega) \in \overline{B_{\delta}(x^\ast)}$ for all $s \ge t + T$. In fact, if $s \in [t+T, t+2T]$ then
    $$
        |X_s(\omega)-x^*|\le |X_s(\omega) - \mathbf x_{s-t}(X_t(\omega))| + |\mathbf x_{s-t}(X_t(\omega))-x^*| \le  \delta
    $$
    and if $s > t +2T$ one gets by induction that $X_{s-T}(\omega) \in \overline {B_{\delta}(x^*)}$ and, thus,
    $$
        |X_s(\omega) - x^*| \le |X_s(\omega)- \mathbf x_{T}(X_{s-T})| + |\mathbf x_{T}(X_{s-T}(\omega))-x^*| \le  \delta.
    $$
\end{proof}

\section{Applications} 
\label{sec:applications}

In this section, we give two important applications for our main results: the theoretical analysis of optimization methods with multiple hyperparameters and energy-based models.

\subsection*{Heavy ball method}
A line of research studies the optimal hyperparameter selection for optimization algorithms via analyzing asymptotic regimes, where one set of hyperparameters is chosen much smaller compared to others.

We aim to minimize a differentiable loss function $L:\R^{d} \to \R$ using the heavy ball method given by
\begin{align*}
\begin{cases}
dX_t = Y_t dt, \\ 
dY_t =  -\frac{1}{\varepsilon_t}(\nabla L(X_t)+\beta Y_t)dt + \frac{1}{\sqrt{\varepsilon_t}}\sigma(X_t)  dW_t, 
\end{cases}
\end{align*}
where $\beta >0$, $\sigma: \R^{d} \to \R^{d \times d}$  and $(W_t)_{t \ge 0}$ denotes a $d$-dimensional Brownian motion.
For fixed $x \in \R^{d}$, the fast equation is an Ornstein-Uhlenbeck process with invariant distribution
\[
\mu_x
=
\mathcal N \big(-\beta^{-1}\nabla L(x),\,(2\beta)^{-1}\sigma(x)\sigma(x)^\top\big).
\]
Therefore, the averaged ODE is precisely the gradient flow
\[
\dot{\mathbf x}_t
=
\int y\,\mu_{\mathbf x_t}(dy)
=
-\frac{1}{\beta} \nabla L(\mathbf x_t).
\]

\subsection*{Training an energy-based model}
Given a parametrized family of probability distributions \((\mu_x)_{x \in \R^{d_1}}\) and a data distribution $\nu$ on $\R^{d_2}$, we aim to find a parameter $x$ such that $\mu_x$ is a good approximation for the data distribution.
We consider an energy-based model, i.e.,
\[
\mu_x(dz)
=
\frac{1}{Z(x)}
\exp(-E_x(z))\, dz,
\]
where $E_x:\mathbb R^{d_2}\to\mathbb R$ is called the energy function and  
\[
Z(x)
=
\int
\exp(-E_x(z))\, dz
\]
denotes a finite normalizing constant.

A natural objective for finding a good model parameter $x$ is to maximize the log-likelihood function $\ell(x;z)= - E_x(z) - \log Z(x)$ averaged over the data distribution, i.e.
\[
\ell(x)
=
 - \int E_x(z) \, \nu(dz) - \log Z(x).
\]
Differentiating the log-likelihood function gives (under suitable assumptions)
\[
\nabla \ell(x)
=
- \int \nabla_x E_x(z) \, \nu(dz) - \nabla \log Z(x)
\]
with 
$$
  \nabla \log Z(x) = -\int  \nabla_x E_x(z) \, \mu_x(dz).
$$

Since $Z(x)$ and, hence $\mu_x$, are often intractable, it is difficult to compute $\nabla \log Z(x)$. 
Instead, $\mu_x$ is often approximated using a Markov process. For fixed
$x$, consider the overdamped Langevin diffusion
$$
dY_s^x
=
-\nabla_y E_x(Y_s^x)\,ds
+
\sqrt{2}\,dW_s,
$$
where $(W_t)_{t \ge 0}$ denotes a $d_2$-dimensional Brownian motion.
If there exists a $\beta >0$ such that for all $x \in \R^{d_1}$ and $y_1, y_2 \in \R^{d_2}$
$$
    \langle y_1-y_2 , \nabla_y E_x(y_1)- \nabla_y E_x(y_2) \rangle \ge \beta |y_1-y_2|^2
$$
the Langevin diffusion is $p$-dissipative for all $p \ge 2$ and the unique invariant distribution is given by $\mu_x$. Hence, one can decrease the approximation error for $\mu_x$ by running the corresponding Langevin diffusion for a longer time. This naturally leads to the fast-slow system
\[
\begin{aligned}
dX_t
&=
\Big(\nabla_x E_{X_t}(Y_t)- \int \nabla_x E_{X_t}(z) \, \nu(dz)  \Big) \, dt,
\\
dY_t
&=
-\frac{1}{\varepsilon_t}
\nabla_y E_{X_t}(Y_t)\,dt
+
\sqrt{\frac{2}{\varepsilon_t}}\,dW_t,
\end{aligned}
\]
where the fast variable can be viewed as the sampling process and the slow variable is the parameter one aims to optimize. 
Hence, under additional regularity assumptions (see Assumption~\ref{assu:regularity}) and for decaying $\eps_t= \cO(t^{-q})$ for a $q>0$, Corollary~\ref{cor:APT} gives that the parameter process $(X_t)_{t \ge 0}$ is almost surely an asymptotic pseudo-trajectory of the gradient flow
\[
\dot{\bar X}_t = 
\nabla \ell(\bar X_t).
\]
See \cite{akyildiz2024multiscale} for a uniform-in-time 
averaging principle for energy-based models under the additional assumption that $-\ell$ is strongly convex.

\section{Proof of the main result}
We will prove the main result, Theorem~\ref{th:convergenceLp}, using the Poisson equation.

\begin{proof}[Proof of Theorem~\ref{th:convergenceLp}]
We note that for all $\tau \in [0,T]$
\begin{align*}
    &\sup_{t \in [T_0,T_0+\tau]} |X_t-\mathbf{x}_{t-T_0}(X_{T_0})| 
    \leq \sup_{t \in [T_0,T_0+\tau]} 
    \left|\int_{T_0}^t 
    (f(X_s,Y_s)-\bar{f}(\mathbf{x}_{s-T_0}(X_{T_0}))) \, ds\right| \\
    &\leq \sup_{t \in [T_0,T_0+\tau]} 
    \left|\int_{T_0}^t 
    (f(X_s,Y_s)-\bar{f}(X_s))ds\right|
    +  \int_{T_0}^{T_0+\tau} 
    |\bar{f}(X_s)-\bar{f}(\mathbf{x}_{s-T_0}(X_{T_0}))|\, ds \\
    &\leq \sup_{t \in [T_0,T_0+\tau]} 
    \left|\int_{T_0}^t 
    (f(X_s,Y_s)-\bar{f}(X_s))ds\right|  + L_{\bar{f}} \int_{0}^{\tau} 
    \sup_{r \in [T_0,T_0+s]} 
    |X_r-\mathbf{x}_{r-T_0}(X_{T_0})| \, ds,
\end{align*}
where $L_{\bar f}\ge 0$ denotes the Lipschitz constant of $\bar f$; see Lemma~\ref{lem:barfLipschitz}.
By Grönwall's lemma this implies
\begin{align} \label{eq:Gronwall1}
    \bigg\|\sup_{t \in [T_0,T_0+T]} 
    |X_t-\mathbf{x}_{t-T_0}(X_{T_0})| \bigg\|_{L^p}
    \leq e^{L_{\bar{f}}T} \, 
    \bigg\| \sup_{t \in [T_0,T_0+T]} 
    \left|\int_{T_0}^t 
    (f(X_s,Y_s)-\bar{f}(X_s))ds\right| 
    \bigg\|_{L^p}
\end{align}
so that it suffices to bound the term on the right-hand side of \eqref{eq:Gronwall1}.

First note that, for a function $\varphi \colon \mathbb R^{d_1+d_2}\to \mathbb R$ that is once continuously differentiable in $x$ and twice continuously differentiable in $y$, Itô's formula gives
\begin{align}
    \begin{split} \label{eq:Itôgeneral}
    d \varphi (X_t,Y_t)
    =&  D_x \varphi (X_t,Y_t) \, f(X_t,Y_t)  \,  dt \\
    &+ \frac{1}{\eps_t} 
    \Big(
     D_y \varphi (X_t,Y_t) \, g(X_t,Y_t) 
    + \frac{1}{2} 
    \mathrm{tr} \Big(
    \sigma(X_t,Y_t)\sigma(X_t,Y_t)^\top D^2_y \varphi (X_t,Y_t)
    \Big)
    \Big) dt \\
    &+ \frac{1}{\sqrt{\eps_t}} 
     D_y \varphi (X_t,Y_t) \, \sigma(X_t,Y_t) \, d W_t \\
    =& \mathcal{L}^0 \varphi (X_t,Y_t) \,  dt 
    + \frac{1}{\eps_t}\mathcal{L}^1 \varphi (X_t,Y_t) \, dt 
    + \frac{1}{\sqrt{\eps_t}} 
     D_y \varphi (X_t,Y_t) \, \sigma(X_t,Y_t) \, d W_t  ,
    \end{split}
\end{align}
where $\mathcal{L}^0$ and $\mathcal{L}^1$ are defined by
\begin{align*}
    \mathcal{L}^0 \varphi (x,y) 
    &=  D_x \varphi(x,y) \,  f(x,y) , \\
    \mathcal{L}^1 \varphi(x,y) 
    &= D_y \varphi(x,y) \, g(x,y)  
    + \frac{1}{2}
    \mathrm{tr}\left(\sigma(x,y)\sigma(x,y)^\top D_y^2 \varphi(x,y)\right).
\end{align*}
In particular, for fixed $x \in \mathbb R^{d_1}$, the map 
$y \mapsto \mathcal{L}^1 \varphi(x,y)$ is the generator of the Markov process 
$\hat{Y}^x$ obtained by freezing the slow variable, namely the solution to the SDE
\begin{align*}
    d\hat{Y}_t^x
    = g(x,\hat{Y}_t^x)dt
    + \sigma(x,\hat{Y}_t^x)d \hat{W}_t .
\end{align*}
We denote its corresponding semigroup by $P^x_t$, $t \geq 0$, and consider the compensator
\begin{align*}
    \Phi \colon \mathbb R^{d_1+d_2} \to \mathbb R^{d_1}\quad ,\quad
    \Phi(x,y) 
    = \int_0^\infty 
    \bar{f}(x)-P_t^x f(x,y) \,  dt.
\end{align*}
Now, Lemma~\ref{lem:barfLipschitz} shows that $\Phi$ is well-defined and Lemma~\ref{lem:Poisson} shows that, for each fixed $x$, $\Phi(x,\cdot)$ solves the Poisson equation
\begin{align*}
    \mathcal{L}^1\Phi(x,\cdot) = f(x,\cdot)-\bar{f}(x).
\end{align*}

Moreover, using Assumption~\ref{assu:regularity} and the $2$-dissipativity of the fast variable, Lemmas~\ref{___lem_DyPhi},~\ref{___lem_D2yPhi}, and~\ref{lem:DxPhi} show that $\Phi \in C^{1,2}$.
Thus, we can apply \eqref{eq:Itôgeneral} componentwise to $\Phi$ and get
\begin{align*}
    d \Phi(X_t,Y_t) 
    = \mathcal{L}^0\Phi(X_t,Y_t)dt
    +\frac{1}{\eps_t} \mathcal{L}^1\Phi(X_t,Y_t)dt 
    + \frac{1}{\sqrt{\eps_t}} 
     D_y\Phi(X_t,Y_t)\, \sigma(X_t,Y_t)dW_t.
\end{align*}
By the Poisson equation,
\begin{align*}
    &(f(X_t,Y_t)-\bar{f}(X_t))\, dt
    = \mathcal{L}^1\Phi(X_t,Y_t)\, dt \\
    &= \eps_t\, d\Phi(X_t,Y_t)
    -\eps_t\mathcal{L}^0\Phi(X_t,Y_t)\, dt
    -\sqrt{\eps_t}  D_y \Phi(X_t,Y_t)\, \sigma(X_t,Y_t)\, dW_t \\
    &= \eps_t\, d\Phi(X_t,Y_t)
    -\eps_t  D_x\Phi(X_t,Y_t)\, f(X_t,Y_t) \, dt
    -\sqrt{\eps_t} D_y \Phi(X_t,Y_t)\, \sigma(X_t,Y_t) \, dW_t  .
\end{align*}
Thus, we can decompose the term on the right-hand side of \eqref{eq:Gronwall1} into
\begin{align*}
    \bigg\| \sup_{t \in [T_0,T_0+T]} 
    \bigg|\int_{T_0}^t 
    (f(X_s,Y_s)-\bar{f}(X_s))&ds\bigg| 
    \bigg\|_{L^p} \leq 
    \bigg\| \sup_{t \in [T_0,T_0+T]} 
    \left|\int_{T_0}^t\eps_s\, d\Phi(X_s,Y_s)\right| 
    \bigg\|_{L^p} \\
    &\quad+
    \bigg\| \sup_{t \in [T_0,T_0+T]} 
    \left|\int_{T_0}^t
    \eps_s D_x\Phi(X_s,Y_s)\, f(X_s,Y_s)  ds\right| 
    \bigg\|_{L^p} \\
    &\quad+
    \bigg\| \sup_{t \in [T_0,T_0+T]} 
    \left|\int_{T_0}^t
    \sqrt{\eps_s}  D_y\Phi(X_s,Y_s)
    \,\sigma(X_s,Y_s) d W_s  \right| 
    \bigg\|_{L^p} .
\end{align*}

We individually bound each of the three terms on the right-hand side of the latter inequality.

\underline{Third term:}
By Assumption~\ref{assu:regularity}, there exists a $C_1 > 0$ with
\begin{align*}
    |\sigma(x,y)| \leq C_1 (1+|y|^p)^{1/p}.
\end{align*}
By the BDG inequality and Jensen's inequality, there exists $C_2 >0$ such that
\begin{align} \begin{split}
    \label{eq:w9regh34983333}
    &\bigg\| \sup_{t \in [T_0,T_0+T]} 
    \bigg|\int_{T_0}^t
    \sqrt{\eps_s}  D_y\Phi(X_s,Y_s)\, 
    \sigma(X_s,Y_s)d W_s  \bigg| 
    \bigg\|_{L^p} \\
    &\leq  C_2 
    {\E} 
    \bigg[\bigg(\int_{T_0}^{T_0+T} 
    \eps_s 
    | D_y\Phi(X_s,Y_s) \, \sigma(X_s,Y_s)|_F^2 \,  
    ds\bigg)^{p/2} \bigg]^{1/p} \\
    &\leq  C_2 \bigg(T^{(p-2)/2}  \int_{T_0}^{T_0+T} \eps_s^{p/2}
    {\E} [|D_y\Phi(X_s,Y_s) \, \sigma(X_s,Y_s)|_F^p] 
    \, ds\bigg)^{1/p} \\
    &\leq C_2 C_1 \sqrt{T}
    \sup_{x,y} |D_y \Phi(x,y)|_F  \Big(1+\sup_{t \geq 0}{\E}[|Y_t|^p]\Big)^{1/p}
    \sqrt{\eps_{T_0}} \\
    &=\mathcal{O}(\sqrt{\eps_{T_0}}),
    \end{split}
\end{align}
where we have used that $\sup_{t \ge 0}{\E}[|Y_t|^p]<\infty$ due to Lemma~\ref{lem:momentfast} and $\sup_{x,y} |D_y \Phi(x,y)|_F <\infty$ due to Lemma~\ref{___lem_DyPhi}.

\underline{Second term:}
Choose $\alpha \in (0,1)$ sufficiently small such that the fast variable is $p+p\alpha$-dissipative with constant $\beta/2$; see Remark~\ref{rem:dissibigger}. Then, by Lemma~\ref{lem:DxPhi}, there exists a constant $C_3\ge 0$ with 
\begin{align*}
    |D_x \Phi(x,y)| \leq C_3 (1+|y|^\alpha),
\end{align*}
so that
\begin{align*}
    |D_x \Phi(x,y) f(x,y)| \leq C_3 (1+|y|^{\alpha})(|f(x,0)| + L_f|y|) \leq C_4(1+|y|^{1+\alpha}),
\end{align*}
where $L_f \ge 0$ denotes the Lipschitz constant of $f$ and $C_4 := 2C_3 \max(L_f, |f(\cdot, 0)|_\infty)$.
We conclude, using Lemma~\ref{lem:momentfast}, that
\begin{align*}
    &\bigg\| \sup_{t \in [T_0,T_0+T]} 
    \left|\int_{T_0}^t
    \eps_s  D_x\Phi(X_s,Y_s)\, f(X_s,Y_s)  ds\right| 
    \bigg\|_{L^p} \\
    &\leq 
    \int_{T_0}^{T_0+T} 
    \eps_s \| D_x\Phi(X_s,Y_s)\, f(X_s,Y_s) \|_{L^p}ds \\
    &\leq 
    C_4T\eps_{T_0}
    (1+\sup_{t \geq 0}\|Y_t\|_{L^{p+p\alpha }}^{1+\alpha }) \\
    &=\mathcal{O}(\eps_{T_0}).
\end{align*}

\underline{First term:}
As a consequence of Itô's product rule and the finite variation of $(\eps_t)_{t \ge 0}$, we have
\begin{align*}
    \eps_t\, d\Phi(X_t,Y_t) 
    = d(\eps_t\Phi(X_t,Y_t))-\Phi(X_t,Y_t) \, d\eps_t .
\end{align*}
Therefore,
\begin{align}
\begin{split}
    \label{eq:ej9gh349h4}
    &\bigg\| \sup_{t \in [T_0,T_0+T]} 
    \bigg|\int_{T_0}^t\eps_s\, d\Phi(X_s,Y_s)\bigg| 
    \bigg\|_{L^p}\\
    &\leq 2 \bigg\| \sup_{t \in [T_0,T_0+T]} 
    \bigg| \eps_t \Phi(X_t,Y_t)\bigg| 
    \bigg\|_{L^p} 
    %+ \eps_{T_0}
    %\bigg\| \Phi(X_{T_0},Y_{T_0})\bigg\|_{L^p} +
    + \bigg\| \int_{T_0}^{T_0+T} 
    |\Phi(X_s,Y_s)|d(-\eps_s) 
    \bigg\|_{L^p} .
    \end{split}
\end{align}
By Lemma~\ref{lem:barfLipschitz}, $\Phi$ satisfies uniform linear growth in $y$, i.e. there exists a constant $C_5$ such that
\begin{align} \label{eq:Phi348z39gz93}
    |\Phi(x,y)| \leq C_5 (1+|y|), \quad \text{ for all }
     x \in \mathbb R^{d_1},\ y\in\mathbb R^{d_2},
\end{align}
and, by Lemma~\ref{lem:momentfast}, there exists a constant $C_6\ge 0$ such that
\begin{align*}
    \| Y_t \|_{L^p} \leq C_6, 
    \quad \text{ for all } t\geq 0.
\end{align*}

Now, for the last term on the right-hand side of \eqref{eq:ej9gh349h4}, Hölder's inequality gives
\begin{align*}
    &\left\| \int_{T_0}^{T_0+T} 
    |\Phi(X_s,Y_s)|d(-\eps_s) 
    \right\|_{L^p} 
    \\
    &\leq 
    \bigg(
    \mathbb E \left[
    \left( \eps_{T_0}-\eps_{T_0+T} \right)^{p-1}
    \int_{T_0}^{T_0+T} 
    |\Phi(X_s,Y_s)|^p d(-\eps_s) 
    \right]\bigg)^{1/p} \\
    &\leq 
    \Big(
    \left( \eps_{T_0}-\eps_{T_0+T} \right)^p
    \sup_{t \in[T_0,T_0+T]}
    \mathbb E[|\Phi(X_t,Y_t)|^p]
    \Big)^{1/p} \\
    &\leq 
    \eps_{T_0} C_5
    \sup_{t \in[T_0,T_0+T]} 
    \left(1+\| Y_t \|_{L^p}\right) \\
    &\leq 
    \eps_{T_0}C_5
    (1+C_6)
    =\mathcal{O}(\eps_{T_0}).
\end{align*}

For the first term on the right-hand side of \eqref{eq:ej9gh349h4}, note that by \eqref{eq:Phi348z39gz93}
\begin{align*}
    \sup_{t\in[T_0,T_0+T]}\eps_t|\Phi(X_t,Y_t)|
    \leq C_5 \eps_{T_0}
    + C_5 
    \sup_{t\in[T_0,T_0+T]}\eps_t |Y_t|.
\end{align*}
Thus, it suffices to control 
$\|\sup_{t\in[T_0,T_0+T]}\eps_t |Y_t| \, \|_{L^p}$.
We define $(Z_t)_{t \ge T_0} := (\eps_t Y_t)_{t \ge T_0}$. Then,
\begin{align*}
    d Z_t 
    = Y_t d\eps_t +\eps_tdY_t 
    = Y_t d \eps_t 
    + g(X_t,Y_t) dt 
    + \sqrt{\eps_t}\sigma(X_t,Y_t)dW_t,
    \quad t\geq T_0.
\end{align*}
For $p\geq 2$, Itô's formula gives
\begin{align*}
    &d|Z_t|^p
    = 
    p |Z_t|^{p-2} \langle Z_t,Y_t \rangle d \eps_t 
    + p |Z_t|^{p-2} 
    \langle Z_t , g(X_t,Y_t) \rangle dt + p\sqrt{\eps_t}|Z_t|^{p-2} 
    \langle Z_t,\sigma(X_t,Y_t)dW_t \rangle \\
    &\qquad \qquad + \frac{p}{2} |Z_t|^{p-2} \eps_t 
    \mathrm{tr}\big(
    \sigma(X_t,Y_t)\sigma(X_t,Y_t)^\top
    \big) dt + \frac{p(p-2)}{2} |Z_t|^{p-4} \eps_t 
    |\sigma(X_t,Y_t)^\top Z_t|^2 dt \\
    &=
    -p \eps_t^{p-1} |Y_t|^{p} d (-\eps_t) 
    + p \eps_t^{p-1} |Y_t|^{p-2} 
    \langle Y_t , g(X_t,Y_t) \rangle dt + \frac{p}{2} \eps_t^{p-1} |Y_t|^{p-2}
    \mathrm{tr}\big(
    \sigma(X_t,Y_t)\sigma(X_t,Y_t)^\top
    \big) dt \\
    &\qquad + \frac{p(p-2)}{2} \eps_t^{p-1} |Y_t|^{p-4}
    |\sigma(X_t,Y_t)^\top Y_t|^2 dt+ p \eps_t^{p-1/2} |Y_t|^{p-2} 
    \langle Y_t,\sigma(X_t,Y_t)dW_t \rangle
    .
\end{align*}
Note that the first summand on the right-hand side of the latter equality is non-positive.
By Assumption~\ref{assu:regularity}, $g$ and $\sigma$ have at most linear growth in $y$ whilst being uniformly bounded in $x$. Thus, there exists a constant $C_7 >0$ such that
\begin{align}
    \begin{split}
        \label{eq:wjo4hhhh}
    {\E}\Big[\sup_{t \in [T_0,T_0+T]}|\eps_t Y_t|^p\Big]
    &\leq
    \eps_{T_0}^{p} {\E}[|Y_{T_0}|^p ]
    +  {\E}\bigg[ \int_{T_0}^{T_0+T} C_7 \eps_s^{p-1} (|Y_s|^p + |Y_s|^{p-1} + |Y_s|^{p-2}) ds \bigg] \\
    &\qquad + p {\E}\bigg[\sup_{t \in [T_0,T_0+T]} \Big| \int_{T_0}^t  \eps_s^{p-1/2} |Y_s|^{p-2} 
    \langle Y_s,\sigma(X_s,Y_s)dW_s \rangle  \Big| \bigg] \\
    &\leq \eps_{T_0}^{p} {\E}[|Y_{T_0}|^p ] + C_7 T \eps_{T_0}^{p-1} \sup_{t \geq 0} \big({\E}[|Y_t|^{p}]+{\E}[|Y_t|^{p-1}]+{\E}[|Y_t|^{p-2}]\big)\\
     &\qquad + p {\E}\bigg[\sup_{t \in [T_0,T_0+T]} \Big| \int_{T_0}^t  \eps_s^{p-1/2} |Y_s|^{p-2} 
    \langle Y_s,\sigma(X_s,Y_s)dW_s \rangle \Big|\bigg].
    \end{split}
\end{align}
For the martingale part, we fix a constant $C_8 > 0$ with
\begin{align*}
    |\sigma(x,y)|^2 \leq C_8(1+|y|^2)
\end{align*}
and note that, by BDG's inequality and Young's inequality, one has
\begin{align}
\begin{split}
    \label{eq:jjjeri34s}
     &
      \E \left[ \sup_{t\in [T_0,T_0+T]}  \left| \int_{T_0}^t  \eps_s^{p-1/2} |Y_s|^{p-2} \langle Y_s,\sigma(X_s,Y_s)dW_s \rangle \right| \right] \\
     &\qquad \leq
     C_9 \E \bigg[ \bigg( \int_{T_0}^{T_0+T}  \eps_s^{2p-1} |Y_s|^{2p-2}  |\sigma(X_s,Y_s)|^2 ds  \bigg)^{1/2} \bigg] \\
     &\qquad \leq
     C_9 \E \bigg[ \bigg(\sup_{t \in [T_0,T_0+T]} |\eps_t Y_t|^p \bigg)^{1/2} \bigg( \int_{T_0}^{T_0+T}  \eps_s^{p-1} |Y_s|^{p-2}  |\sigma(X_s,Y_s)|^2 ds  \bigg)^{1/2} \bigg] \\
     &\qquad \leq
     \delta \frac{C_9}{2} {\E} \Big[ \sup_{t \in [T_0,T_0+T]} |\eps_t Y_t|^p \Big] + \frac{C_8C_9}{2\delta} \int_{T_0}^{T_0+T} \eps_s^{p-1} ({\E}[|Y_s|^{p-2}] +{\E}[|Y_s|^{p}]) ds \\
     &\qquad \leq
     \delta \frac{C_9}{2} {\E} \Big[ \sup_{t \in [T_0,T_0+T]} |\eps_t Y_t|^p \Big] + \frac{C_8C_9T}{2\delta} \eps_{T_0}^{p-1} \sup_{t \geq 0}({\E}[|Y_t|^{p-2}] +{\E}[|Y_t|^{p}])  
     \end{split}
\end{align}
for a constant $C_9 > 0$ and arbitrary $\delta > 0$.
Using the uniform Lipschitz continuity of the coefficients, Assumption~\ref{assu:regularity}, we can apply a standard result from stochastic analysis, see e.g. Theorem 1.6.16 in \cite{yong1999stochastic}, to get
\begin{align*}
\E \Big[ \sup_{t \in [T_0,T_0+T]} |\eps_t Y_t|^p \Big] \leq \eps_{T_0}^p \E \Big[ \sup_{t \in [T_0,T_0+T]} |Y_t|^p \Big] < \infty.
\end{align*} 
Therefore, we can choose $\delta \in (0,\frac{2}{p \, C_9})$ and combine \eqref{eq:wjo4hhhh} and \eqref{eq:jjjeri34s} to arrive at
\begin{align*}
    0 \leq \Big(1-p \delta\frac{C_9}{2}\Big){\E} \Big[ \sup_{t \in [T_0,T_0+T]} |\eps_t Y_t|^p \Big] \leq \eps_{T_0}^{p-1} C_{10} \, \sup_{t \geq 0}({\E}[|Y_{t}|^{p}]+{\E}[|Y_{t}|^{p-1}]+{\E}[|Y_{t}|^{p-2}]),
\end{align*}
for a constant $C_{10}> 0$ and sufficiently large $T_0$.
As the moments of $Y$ are uniformly bounded over $t$ up to order $p$, see Lemma~\ref{lem:momentfast}, we have 
\begin{align*}
    \bigg\| \sup_{t \in [T_0,T_0+T]} 
    \eps_t |Y_t| \bigg\|_{L^p} 
    = \mathcal{O}\Big(\eps_{T_0}^{\frac{p-1}{p}}\Big).
\end{align*}
Consequently,
\begin{align*}
    \bigg\| \sup_{t \in [T_0,T_0+T]} 
    \eps_t |\Phi(X_t,Y_t)| \bigg\|_{L^p}
    =
    \mathcal{O}\Big(\eps_{T_0}^{\frac{p-1}{p}}\Big),
\end{align*}
and therefore
\begin{align*}
    \bigg\| \sup_{t \in [T_0,T_0+T]} 
    \left|\int_{T_0}^t\eps_s\, d\Phi(X_s,Y_s)\right| 
    \bigg\|_{L^p}
    =
    \mathcal{O}\Big(\eps_{T_0}^{\frac{p-1}{p}}\Big).
\end{align*}
The statement follows since $\eps_{T_0}^{(p-1)/p}\le \sqrt{\eps_{T_0}}$ for all $p \ge 2$ and sufficiently large $T_0$.
\end{proof}

\begin{remark}
    An alternative approach is to partition the interval $[T_0,T_0+T]$ into small subintervals and freeze the slow variable at the beginning of each subinterval. However, estimating the global error by applying the triangle inequality to the resulting local approximation errors generally leads to a suboptimal approximation error. Indeed, if one controls the resulting martingale contributions by applying the BDG inequality separately on the individual subintervals rather than to a single martingale over the full interval, one fails to exploit their cancellations and averaging properties optimally. This is especially important since the size of the martingale is the dominating term in the bound; see \eqref{eq:w9regh34983333}. Thus, recovering the optimal rate within this framework would require controlling the accumulated martingale errors jointly over the entire interval. In our Poisson corrector approach, the martingale is not split into parts.
\end{remark}

\begin{remark}
Many results on the averaging principle in the literature provide estimates with the supremum over time outside the expectation; see, for example,
\cite{liu2010strong,BrehierCharles-Edouard2020Ooci,FuHongbo2015Scia,sun2024averaging}.
These bounds, therefore, control the marginal errors uniformly over the compact time-interval, whereas Theorem~\ref{th:convergenceLp} gives the stronger pathwise estimate bounding
\[
\E\bigg[
\sup_{t\in[T_0,T_0+T]}
|X_t-\mathbf{x}_{t-T_0}(X_{T_0})|^p
\bigg].
\]
The additional technical difficulties become apparent, e.g., in \eqref{eq:jjjeri34s}, where, instead of 
$$
 \sup_{t \in [T_0,T_0+T]}{\E} [ |\eps_t Y_t|^p ] \le 
\eps_{T_0 }^p \sup_{t \in [T_0,T_0+T]}{\E}[  | Y_t|^p ],
$$
one has to control ${\E} [ \sup_{t \in [T_0,T_0+T]} |\eps_t Y_t|^p ]$. The former can immediately be bounded using the $p$-dissipativity, see Lemma~\ref{lem:momentfast}, while the latter requires a maximal estimate for the solution of an SDE. This distinction is important because the interval $[T_0,T_0+T]$ corresponds to a fast-time interval of length $\int_{T_0}^{T_0+T} \frac{1}{\eps_s} ds$ which diverges as $T_0 \to \infty$. Consequently, one needs to bound the fluctuations of the fast process over an increasing amount of time in its corresponding timescale. 
\end{remark}

\section{Auxiliary Results}
In this section, we provide auxiliary results used in the proof of Theorem~\ref{th:convergenceLp}.
First, we discuss moment bounds and ergodicity of the frozen fast equation \eqref{eq:frozen0} under the dissipativity and regularity assumptions. Then, we provide a moment bound for the fast variable $(Y_t)_{t \ge 0}$. Next, we prove well-definedness of the Poisson equation and show that the Poisson corrector $\Phi$ solves the Poisson equation. Lastly, we provide regularity results for the Poisson equation.

\subsection{Moment bounds and ergodicity of the frozen fast equation}
We consider the frozen fast SDE on a filtered probability space $(\hat \Omega, \hat {\mathcal F}, (\hat {\mathcal F}_t)_{t \ge 0}, \hat \P)$ satisfying the usual conditions, which is obtained by freezing the slow variable, i.e. the equation
\begin{align*}
    d\hat{Y}_t^{x,y}
    = g(x,\hat{Y}_t^{x,y})dt
    + \sigma(x,\hat{Y}_t^{x,y})d \hat{W}_t
\end{align*}
with initial condition $\hat Y_0^{x,y}=y$. 

We repeat the assumption on the coefficients.
\begin{assumption} \label{assu:frozenLipschitz}
    There exists $L_g,L_\sigma, C_g, C_\sigma \ge 0$ such that for all $x_1,x_2 \in \R^{d_1}$ and $y_1,y_2 \in \R^{d_2}$
    $$
        |g(x_1,y_1)-g(x_2,y_2)| \le L_g |(x_1,y_1)-(x_2,y_2)|,
    $$
    $$
        |\sigma(x_1,y_1)-\sigma(x_2,y_2)|_F \le L_\sigma |(x_1,y_1)-(x_2,y_2)|,
    $$
    as well as
    $$
        |\sigma(x_1,0)|_F \le C_\sigma \quad \text{ and } \quad |g(x_1,0)| \le C_g.
    $$
\end{assumption}

\begin{definition} \label{def:frozendissi}
    Let $p \geq 2$ and $\beta > 0$. We say that the frozen variable $(\hat Y_t^x)_{t \ge 0}$ is (uniformly) $p$-dissipative (with constant $\beta$), if for all $x\in \R^{d_1}$ and $y_1,y_2 \in \R^{d_2}$
\begin{align*}
    \langle y_1 - y_2, g(x,y_1)-g(x,y_2) \rangle + \frac{p-1}{2} |\sigma(x,y_1)-\sigma(x,y_2)|_F^2 \leq -\beta|y_1-y_2|^2.
\end{align*}
\end{definition}

\begin{lemma} \label{___lem_Y}
    Assume that Assumption~\ref{assu:frozenLipschitz} is satisfied. Then for any fixed time horizon $T>0$ and $p \geq 2$, the process $(\hat Y^{x,y}_t)_{t \in [0,T]}$ is locally $L^p$-bounded and continuous in $(x,y) \in \R^{d_1+d_2}$ with respect to $L^p$-convergence. More precisely, for any bounded set $U \subseteq \R^{d_2}$ there exists a constant $C_{p,T,U}>0$ such that
    \begin{align*}
        \sup_{(x,y) \in \R^{d_1}\times U} \hat \E \Big[ \sup_{0 \leq t \leq T} |\hat Y^{x,y}_t|^p \Big] \leq C_{p,T,U},
    \end{align*}
    and for any $(x,y) \in \R^{d_1+d_2}$ we have
    \begin{align*}
        \lim_{(x',y')\to (x,y)} \hat \E \Big[ \sup_{0 \leq t \leq T} |\hat Y^{x',y'}_t-\hat Y^{x,y}_t|^p \Big] =0.
    \end{align*}
\end{lemma}
\begin{proof}
    By the BDG-inequality, there exists a constant $C^{(1)}_p>0$ such that for any continuous $d_2$-dimensional local martingale $(M_t)_{t \geq 0}$ and any stopping time $\tau$, one has
    \begin{align}\label{___Y_BDG_p}
        \hat \E \Big[ \sup_{0\leq t \leq \tau} |M_t|^p \Big] \leq C^{(1)}_p \hat \E [[M]_{\tau}^{p/2}],
    \end{align}
    where $[\cdot]$ denotes the trace of the quadratic covariation of $M$, i.e. 
    \begin{align}\label{___Y_trace_covariation}
        [M]_t := \mathrm{tr} ([M^{(j)},M^{(k)}]_t)_{1 \leq j,k \leq d_2} = \sum_{k=1}^{d_2} [M^{(k)}]_t, \quad t \geq 0.
    \end{align}
    
    Next, for $(x,y) \in \R^{d_1+d_2}$, we define the localizing sequence
    \begin{align*} 
        \tau_N := \tau_N^{x,y} := \inf \{t \geq 0 : |\hat Y^{x,y}_t| \geq N \}, \quad N \in \N.
    \end{align*}
    Then, by the BDG inequality and the boundedness of $g(\cdot,0)$ and $\sigma(\cdot,0)$, there exists a constant $C^{(2)}_{p,T} >0$ such that
    \begin{align*}
        &\hat \E \Big[\sup_{0 \leq s \leq t \wedge \tau_N} |\hat Y^{x,y}_s|^p \Big]
        \\
        &\leq 3^{p-1} \bigg( |y|^p +
        \hat \E \bigg[ \sup_{0 \leq s \leq t \wedge \tau_N} \bigg|\int_0^s g(x,\hat Y^{x,y}_r) dr\bigg|^p \bigg] + \hat \E \bigg[ \sup_{0 \leq s \leq t\wedge \tau_N} \bigg| \int_0^s \sigma(x,\hat Y^{x,y}_r) d \hat W_r \bigg|^p \bigg] \bigg)
        \\
        &\leq 3^{p-1} \bigg( |y|^p +
        \int_0^t t^{p-1} \hat \E[|g(x,\hat Y^{x,y}_{s \wedge \tau_N})|^p] ds + C^{(1)}_p\hat \E \bigg[\bigg( \int_0^{t \wedge \tau_N} |\sigma(x,\hat Y^{x,y}_s)|_F^2 ds \bigg)^{p/2} \bigg] \bigg) 
        \\
        &\leq 3^{p-1}|y|^p + 6^{p-1}\int_0^t \hat \E\Big[ t^{p-1} (|g(x,0)|^p + L_g^p|\hat Y^{x,y}_{s \wedge \tau_N}|^p) + C^{(1)}_p t^{p/2-1}(|\sigma(x,0)|_F^p + L_\sigma^p|\hat Y^{x,y}_{s \wedge \tau_N}|^p ) \Big]ds 
        \\
        &\leq C^{(2)}_{p,T} (|y|^p+1) + C^{(2)}_{p,T} \int_0^t \hat \E \Big[\sup_{0 \leq r \leq s \wedge \tau_N} |\hat Y^{x,y}_r|^p \Big] ds
    \end{align*}
    for all $0 \leq t \leq T$ and $(x,y) \in \R^{d_1+d_2}$.
    As the left-hand side is finite due to localization, Grönwall's and Fatou's lemma yield
    \begin{align*}
        \hat \E \Big[\sup_{0 \leq t \leq T}|\hat Y^{x,y}_t|^p \Big]  \leq \liminf_{N\to \infty }\hat \E \Big[\sup_{0 \leq t \leq T \wedge \tau_N} |\hat Y^{x,y}_t|^p \Big] \leq C^{(2)}_{p,T} e^{C^{(2)}_{p,T} T} (|y|^p +1)
    \end{align*}
    for all $(x,y) \in \R^{d_1+d_2}$ and thus, the first claim.
    
    Next, for any $(x,y) \in \R^{d_1+d_2}$, one can find a constant $C^{(3)}_{T} >0 $ with
    \begin{align*}
        \hat \E \Big[\sup_{0 \leq s \leq t} |\hat Y^{x',y'}_s - \hat Y^{x,y}_s|^2 \Big]
        &\leq 3 \bigg( |y'-y|^2 + \hat \E \bigg[ \bigg( \int_0^t |g(x',\hat Y^{x',y'}_s)-g(x, \hat Y^{x,y}_s)| ds\bigg)^2 \bigg] \\
        &\qquad + C^{(1)}_2 \hat\E \bigg[\int_0^t |\sigma(x',\hat Y^{x',y'}_s)-\sigma(x, \hat Y^{x,y}_s)|_F^2 ds \bigg] \bigg)
        \\
        & \leq C^{(3)}_{T} (|y'-y|^2 + |x'-x|^2) + C^{(3)}_{T} \int_0^t \hat \E \Big[\sup_{0 \leq r \leq s} |\hat Y^{x',y'}_r - \hat Y^{x,y}_r|^2 \Big] ds
    \end{align*}
    for all $0\leq t \leq T$ and $(x',y')\in\R^{d_1+d_2}$.
    Hence, the continuity follows by Grönwall's inequality
    \begin{align*}
        \hat \E \Big[\sup_{0 \leq t \leq T} |\hat Y^{x',y'}_t - \hat Y^{x,y}_t|^2 \Big] \leq C^{(3)}_{T}e^{C^{(3)}_{T}T} (|y'-y|^2+|x'-x|^2 ) \longrightarrow 0,
    \end{align*}
    as $(x',y')\to (x,y)$.
    Since the differences $(\sup_{0 \leq t \leq T}|\hat Y^{x',y'}_t - \hat Y^{x,y}_t|)_{(x',y')\in U}$ are also uniformly $L^{2p}$-bounded in any bounded neighborhood $U$ of $(x,y) \in \R^{d_1+d_2}$ for any $p \geq 2$ (by the previous statement), we get, using $a^p=(a^2)^{p/(2p-2)}(a^{2p})^{(p-2)/(2p-2)}$ for $a \ge 0$ and Hölder inequality with exponents $(2p-2)/p$ and $(2p-2)/(p-2)$, that
    \begin{align*}
        &\Big\|\sup_{0 \leq t \leq T} |\hat Y^{x',y'}_t - \hat Y^{x,y}_t| \Big\|_{L^p}  \\
        &\le \Big\|\sup_{0 \leq t \leq T} |\hat Y^{x',y'}_t - \hat Y^{x,y}_t| \Big\|_{L^2}^{1/(p-1)} \, \Big\|\sup_{0 \leq t \leq T} |\hat Y^{x',y'}_t - \hat Y^{x,y}_t| \Big\|_{L^{2p}}^{(p-2)/(p-1)}    \longrightarrow 0,
    \end{align*}
    for all $p \ge 2$ as $(x',y')\to (x,y)$.
\end{proof}

Next, we show a contraction property and moment bounds for the frozen fast process.

\begin{lemma} \label{lem:frozenmoment} 
    Let $p \geq 2$ and $\beta >0$. Assume that Assumption~\ref{assu:frozenLipschitz} is satisfied and the frozen variable is $p$-dissipative with constant $\beta$. Then, the following holds:
\begin{enumerate}
    \item[(i)] For all $x\in \R^{d_1}$, $y_1,y_2 \in \R^{d_2}$, and $t\geq 0$ we have
        \begin{align*}\|\hat{Y}^{x,y_1}_t-\hat{Y}_t^{x,y_2}\|_{L^p(\hat{\P})} \leq e^{-\beta t}|y_1-y_2|.\end{align*}
    \item[(ii)] There exists a $C_{\beta,p} > 0$ such that for all $x_1,x_2 \in \R^{d_1}$, $y \in \R^{d_2}$, and $t\geq 0$ we have
    \begin{align*}
        \|\hat{Y}^{x_1,y}_t-\hat{Y}^{x_2,y}_t\|_{L^p(\hat{\P})} \leq C_{\beta,p} |x_1-x_2| .
    \end{align*}
    \item[(iii)] There exist $C_{\beta,p}^{(1)}, C_{\beta,p}^{(2)} > 0$ such that for all $x \in \R^{d_1}$, $y \in \R^{d_2}$, and $t\geq 0$ we have       
    \begin{align*}
    \| \hat{Y}^{x,y}_t\|_{L^p(\hat{\P})} \leq \left\{ e^{-p(\beta  /2) t}|y|^p + C_{\beta,p}^{(1)} \right\}^{1/p} \leq e^{-(\beta  /2) t}|y| + C_{\beta,p}^{(2)}. 
    \end{align*}
\end{enumerate}
\end{lemma}

\begin{proof}
    (i): Fix $x \in \R^{d_1}$, $y_1,y_2 \in \R^{d_2}$, and, for $t\ge 0$, set
            \begin{align*}
            Z_t 
            &:= \hat{Y}^{x,y_1}_t - \hat{Y}^{x,y_2}_t \\
            &= y_1-y_2 + \int_0^t ( g(x,\hat{Y}^{x,y_1}_s)-g(x,\hat{Y}^{x,y_2}_s) ) ds + \int_0^t ( \sigma(x,\hat{Y}^{x,y_1}_s)-\sigma(x,\hat{Y}^{x,y_2}_s) ) d \hat W_s.
            \end{align*}
            Applying Itô's formula yields
           \begin{align*}
            |Z_t|^p
            = &|y_1-y_2|^p \\
        &+ \int_0^t p |Z_s|^{p-2}
        \langle Z_s,
        g(x,\hat{Y}^{x,y_1}_s)-g(x,\hat{Y}^{x,y_2}_s)
        \rangle ds \\
        & + \int_0^t p |Z_s|^{p-2}
        \langle Z_s,
        (\sigma(x,\hat{Y}^{x,y_1}_s)-\sigma(x,\hat{Y}^{x,y_2}_s))d \hat W_s
        \rangle \\
        & + \int_0^t \frac{p}{2} |Z_s|^{p-2}
        |
        \sigma(x,\hat{Y}^{x,y_1}_s)-\sigma(x,\hat{Y}^{x,y_2}_s)
        |_{F}^2 ds \\
        & + \int_0^t \frac{p(p-2)}{2} |Z_s|^{p-4}
        |
        (\sigma(x,\hat{Y}^{x,y_1}_s)-\sigma(x,\hat{Y}^{x,y_2}_s))^\top Z_s
        |^2 ds,
        \end{align*}
        where
        $$
            M_t := \int_0^t p |Z_s|^{p-2}
        \langle Z_s,
        \big(\sigma(x,\hat{Y}^{x,y_1}_s)-\sigma(x,\hat{Y}^{x,y_2}_s)\big)d \hat W_s
        \rangle
        $$
        is an $L^2$-martingale, since Lemma~\ref{___lem_Y} implies that
        \[
        \hat \E \Big[ \int_0^t |Z_s|^{2p-4}|(\sigma(x,\hat Y_s^{x,y_1})-\sigma(x,\hat Y_s^{x,y_2}))^\top Z_s|^2 \, ds \Big]
        \le L_\sigma^2\int_0^T\mathbb{E} [ |Z_s|^{2p} ]\,ds <\infty , \quad 0\leq t \leq T
        \]
        for any finite $t \in [0,T]$ by Lemma~\ref{___lem_Y}. Thus, taking expectation yields
        \begin{align} \begin{split} \label{eq:2307272397}
            \hat{\E}[|Z_t|^p] = |y_1-y_2|^p + \int_0^t \hat{\E}\Big[&p|Z_s|^{p-2} \langle Z_s, g(x,\hat{Y}^{x,y_1}_s)-g(x,\hat{Y}^{x,y_2}_s) \rangle \\
            &+ \frac{p}{2} |Z_s|^{p-2} \, |\sigma(x,\hat{Y}^{x,y_1}_s)-\sigma(x,\hat{Y}^{x,y_2}_s)|_F^2 \\
            &+ \frac{p(p-2)}{2} |Z_s|^{p-4} \, |(\sigma(x,\hat{Y}^{x,y_1}_s)-\sigma(x,\hat{Y}^{x,y_2}_s))^\top Z_s| ^2 \Big] ds.
            \end{split}
        \end{align}
        Differentiating \eqref{eq:2307272397} and using $|A^\top v| \le |A| \, |v| \le |A|_F \, |v|$ for all $A \in \R^{d_2\times d_2} $ and $v \in \R^{d_2}$, as well as the uniform $p$-dissipativity, we get for all $t\geq 0$
        \begin{align*}
            \frac{d}{dt} \hat{\E}[|Z_t|^p] &\leq  p\hat{\E}\Big[|Z_t|^{p-2} \Big( \langle Z_t, g(x,\hat{Y}^{x,y_1}_t)-g(x,\hat{Y}^{x,y_2}_t) \rangle + \frac{p-1}{2} |\sigma(x,\hat{Y}^{x,y_1}_t)-\sigma(x,\hat{Y}^{x,y_2}_t)|_F^2 \Big) \Big] \\
            &\leq -p\beta \hat{\E}[|Z_t|^p].
        \end{align*}
        Hence, by Grönwall's differential inequality we have for all $t \ge 0$
        \begin{align*}
            \hat{\E}[|Z_t|^p] \leq e^{-\beta p t} |y_1-y_2|^p.
        \end{align*}
            
    (ii): Fix $x_1,x_2 \in \R^{d_1}$, $y \in \R^{d_2}$ and, for $t\ge 0$, set
            \begin{align*}
            Z_t 
            &:= \hat{Y}^{x_1,y}_t - \hat{Y}^{x_2,y}_t \\
            &= \int_0^t ( g(x_1,\hat{Y}^{x_1,y}_s)-g(x_2,\hat{Y}^{x_2,y}_s) ) ds + \int_0^t ( \sigma(x_1,\hat{Y}^{x_1,y}_s)-\sigma(x_2,\hat{Y}^{x_2,y}_s) ) d\hat W_s.
            \end{align*}
            Analogously to before, we get by Itô's formula and taking expectation that
            \begin{align*}
                \hat{\E}[|Z_t|^p] =  \int_0^t \hat{\E}\Big[ & p|Z_s|^{p-2}  \langle Z_s, g(x_1,\hat{Y}^{x_1,y}_s)-g(x_2,\hat{Y}^{x_2,y}_s) \rangle \\
                &+ \frac{p}{2} |Z_s|^{p-2} \, |\sigma(x_1,\hat{Y}^{x_1,y}_s)-\sigma(x_2,\hat{Y}^{x_2,y}_s)|_F^2 \\
                &+ \frac{p(p-2)}{2} |Z_s|^{p-4} \, |(\sigma(x_1,\hat{Y}^{x_1,y}_s)-\sigma(x_2,\hat{Y}^{x_2,y}_s))^\top Z_s| ^2 \Big] ds,
            \end{align*}
            so that after differentiating we get
            \begin{align*}
            \begin{split}
                %\label{eq:230723fuh3i}
                &\frac{d}{dt} \hat{\E}[|Z_t|^p] \\
                &\leq  p\hat{\E}\Big[|Z_t|^{p-2} \Big( \langle Z_t, g(x_1,\hat{Y}^{x_1,y}_t)-g(x_2,\hat{Y}^{x_2,y}_t) \rangle + \frac{p-1}{2} |\sigma(x_1,\hat{Y}^{x_1,y}_t)-\sigma(x_2,\hat{Y}^{x_2,y}_t)|_F^2 \Big) \Big] \\
                &= p\hat{\E}\Big[|Z_t|^{p-2} \Big( \langle Z_t, g(x_1,\hat{Y}^{x_1,y}_t)-g(x_1,\hat{Y}^{x_2,y}_t) \rangle + \frac{p-1}{2} |\sigma(x_1,\hat{Y}^{x_1,y}_t)-\sigma(x_1,\hat{Y}^{x_2,y}_t)|_F^2 \Big) \Big] \\
                    &\quad + p\hat{\E}\Big[|Z_t|^{p-2} \Big( \langle Z_t, g(x_1,\hat{Y}^{x_2,y}_t)-g(x_2,\hat{Y}^{x_2,y}_t) \rangle + \frac{p-1}{2} |\sigma(x_1,\hat{Y}^{x_2,y}_t)-\sigma(x_2,\hat{Y}^{x_2,y}_t)|_F^2 \Big) \Big] \\
                    &\quad + p(p-1)\hat{\E}\Big[|Z_t|^{p-2} \mathrm{tr} \Big( (\sigma(x_1,\hat{Y}^{x_1,y}_t)-\sigma(x_1,\hat{Y}^{x_2,y}_t))^\top(\sigma(x_1,\hat{Y}^{x_2,y}_t)-\sigma(x_2,\hat{Y}^{x_2,y}_t)) \Big) \Big] \\
                &\leq - p\beta \hat{\E}[|Z_t|^p] + (p L_g + p(p-1)L_\sigma^2)\,  |x_1-x_2| \, \hat{\E}[|Z_t|^{p-1}] \\
                &\quad + \frac{p(p-1)L_\sigma^2}{2} |x_1-x_2|^2 \,  \hat{\E}[|Z_t|^{p-2}] 
            \end{split}
            \end{align*}
            for all $t \geq 0$. Then, by Young's inequality,
            \begin{align} \label{___Young_ineq}
                u_1^{p-k} u_2^k = (\delta^{(p-k)/p} u_1^{p-k}) (\delta^{-(p-k)/p}u_2^k) 
                \leq \frac{p-k}{p} \delta u_1^p + \frac{k}{p} \delta^{-(p-k)/k} u_2^p\leq \delta u_1^p + \delta^{-(p-k)/k} u_2^p
            \end{align}
            for $k=1,2$ and arbitrary $\delta >0$, $u_1,u_2 \geq 0$, there exists $C_{\beta,p} >0$ with
            \begin{align*}
                \frac{d}{dt} \hat{\E}[|Z_t|^p] \leq -\frac{p\beta}{2}\hat{\E}[|Z_t|^p] + C_{\beta,p} |x_1-x_2|^p, \quad t \geq 0,
            \end{align*}
            such that, by Grönwall's differential inequality,
            \begin{align*}
                \hat{\E}[|Z_t|^p] \leq \frac{2C_{\beta,p}}{p\beta} |x_1-x_2|^p (1-e^{-p(\beta /2)t}) \leq \frac{2C_{\beta,p}}{p\beta} |x_1-x_2|^p, \quad t \geq 0.
            \end{align*}
            
        (iii): Fix $x \in \R^{d_1}$, $y \in \R^{d_2}$. By Itô's formula
            \begin{align*}
                \hat{\E}[|\hat{Y}^{x,y}_t|^p] 
                = |y|^p &+ p\int_0^t \hat{\E}\Big[
                |\hat{Y}^{x,y}_s|^{p-2}\langle \hat{Y}^{x,y}_s , g(x,\hat{Y}^{x,y}_s)\rangle \\
                &+\frac{1}{2}|\hat{Y}^{x,y}_s|^{p-2}|\sigma(x,\hat{Y}^{x,y}_s)|_F^2 
                + \frac{p-2}{2} |\hat{Y}^{x,y}_s|^{p-4} 
                |\sigma(x,\hat{Y}^{x,y}_s)^\top \hat{Y}^{x,y}_s|^2
                \Big]ds.
            \end{align*}
            Using Assumptions~\ref{assu:frozenLipschitz} and $p$-dissipativity, we get for $t \ge 0$
            \begin{align*}
                \frac{d}{dt} \hat{\E}[|\hat{Y}^{x,y}_t|^p] 
                \leq& 
                p \hat{\E}\Big[
                |\hat{Y}^{x,y}_t|^{p-2} \Big(
                \langle \hat{Y}^{x,y}_t , g(x,\hat{Y}^{x,y}_t)-g(x,0)\rangle 
                +\langle \hat{Y}^{x,y}_t , g(x,0)\rangle
                \\
                &
                \qquad + \frac{p-1}{2} |\sigma(x,\hat{Y}^{x,y}_t)-\sigma(x,0)|_F^2 
                + \frac{p-1}{2} |\sigma(x,0)|_F^2 
                \\
                &
                \qquad + (p-1) |\sigma(x,\hat{Y}^{x,y}_t)-\sigma(x,0)|_F |\sigma(x,0)|_F 
                \Big)  \Big] 
                \\
                \leq& -p\beta \hat{\E}[|\hat{Y}^{x,y}_t|^p] 
                + p (C_g+(p-1)L_\sigma C_\sigma) 
                \hat{\E}[|\hat{Y}^{x,y}_t|^{p-1}] 
                \\
                &+ \frac{p(p-1)}{2} 
                C_\sigma^2
                \hat{\E}[|\hat{Y}^{x,y}_t|^{p-2}].
            \end{align*}
            Using
            $\hat{\E}[|\hat{Y}^{x,y}_t|^{p-k}]
            \leq\hat{\E}[|\hat{Y}^{x,y}_t|^{p}]^{(p-k)/p}$, $k=1,2$, 
            one can find a constant $ C_{\beta,p}>0$  independent of $x,y$ with
            \begin{align*}
                \frac{d}{dt} \hat{\E}[|\hat{Y}^{x,y}_t|^p] 
                \leq -\frac{p\beta}{2}\hat{\E}[|\hat{Y}^{x,y}_t|^p] 
                + C_{\beta,p}.
            \end{align*}
            By Grönwall's differential inequality, the claim follows by
            \begin{align*}
                 \hat{\E}[|\hat{Y}^{x,y}_t|^p] 
                 \leq |y|^p e^{-p (\beta /2)t} 
                 + \frac{2C_{\beta,p} }{\beta p} 
                 (1-e^{-p (\beta /2)t}) 
                 \leq |y|^p e^{-p (\beta /2)t} 
                 + \frac{2C_{\beta,p} }{\beta p}.
            \end{align*}
\end{proof}

For $x \in \R^{d_1}$, we denote by
    $(P^x_t)_{t\geq 0}$ the Markov semigroup associated with
    $(\hat{Y}^{x,y}_t)_{t\geq 0}$, that is,
    \[
        P^x_t\varphi(y)
        :=
        \hat{\E}[\varphi(\hat{Y}^{x,y}_t)],
        \qquad \varphi\in B_b(\R^{d_2}).
    \]
The same notation will be used for Borel-measurable functions $\varphi$ of at most linear growth. In this case, $P_t^x\varphi(y)$ is well defined due to Lemma~\ref{lem:frozenmoment}.    

Lemma~\ref{lem:frozenmoment} (i) shows that, in expectation, solutions to the frozen SDE get contracted and the initial condition is forgotten exponentially fast. It is straightforward to prove the following consequence.

\begin{lemma} \label{lem:invariantmoment}
    Let $p\geq 2$ and $\beta >0$. Suppose that Assumptions~\ref{assu:frozenLipschitz} is satisfied and the frozen variable is $p$-dissipative with constant $\beta$.
    Then, for every $x\in\R^{d_1}$, the frozen process $(\hat{Y}^{x,y}_t)_{t\geq 0}$ admits a unique invariant distribution $\mu_x\in\mathcal{P}_p(\R^{d_2})$ such that
$$
\sup_{x\in\R^{d_1}}
\int |z|^p \, \mu_x(dz)
< \infty
$$
and 
$$
W_p(\mathcal{L} (\hat{Y}^{x,y}_t),\mu_x) \leq e^{-\beta t} \Big(\int|y-z|^p \,\mu_x(dz)\Big)^{1/p}
$$
for all $x\in\R^{d_1}$, $y\in\R^{d_2}$, and $t\geq 0$.
\end{lemma}

\begin{proof}
For fixed $x\in\R^{d_1}$, let 
$$
\nu_T^x
:=
\frac{1}{T}
\int_0^T
\mathcal{L}(\hat{Y}^{x,0}_t)\ ,dt.
$$
By Lemma~\ref{lem:frozenmoment}~(iii), we have
$$
\int |z|^p \, \nu_T^x(dz)= \frac{1}{T}\int_0^T\hat{\E}[|\hat{Y}^{x,0}_t|^p
]dt\leq C_{\beta,p}
$$
for a constant $C_{\beta,p}>0$.
Hence, $(\nu_T^x)_{T>0}$ is tight and, by Prokhorov's theorem, there exists a sequence $(T_n)_{n \in \N}$ with $T_n \overset{n \to \infty}{\longrightarrow}\infty$ and a probability measure $\mu_x$ such that $\nu_{T_n}^x \Rightarrow \mu_x$.
The Portmanteau theorem and Lemma~\ref{lem:frozenmoment}~(iii) yield existence of a constant $C_{\beta, p} >0 $ such that
\begin{align*}
    \int |z|^p\,\mu_x(dz)
    \leq
    \liminf_{n\to\infty}
    \int |z|^p \, d\nu_{T_n}^x(z) \le C_{\beta,p}.
\end{align*}

Moreover, $\mu_x$ is invariant, since for all $\varphi \in C_b(\R^{d_2})$ and $r \ge 0$ one has $P_r^x \varphi \in C_b(\R^{d_2})$ and, thus, for all $s \ge 0$
\begin{align*}
    \bigg| \int P_s^x \varphi \, d\nu_T^x -\int \varphi \, d\nu_T^x \bigg| \le \frac{1}{T} \bigg| \int_T^{T+s} P_r^x \varphi(0) \, dr - \int_0^s P_r^x \varphi(0) \, dr \bigg| \overset{T \to \infty}{\longrightarrow}0.
\end{align*}

Using the invariance of $\mu_x$ and Lemma~\ref{lem:frozenmoment}~(i), we obtain
$$
W_p(\mathcal{L} (\hat{Y}^{x,y}_t),\mu_x)^p\leq\int \hat{\E}[|\hat{Y}^{x,y}_t-\hat{Y}^{x,z}_t|^p]\, \mu_x(dz)
\leq e^{-p\beta t} \int |y-z|^p\,\mu_x(dz).
$$

Finally, let $\tilde{\mu}_x$ be another invariant probability measure. Then, for every $f\in C_b(\R^{d_2})$, dominated convergence gives
$$
\int f(z)\,\tilde{\mu}_x(dz)
=
\int P_t^xf(y)\,\tilde{\mu}_x(dy)
\overset{t\to\infty}{\longrightarrow}
\int f(z)\,\mu_x(dz),
$$
so that $\tilde{\mu}_x=\mu_x$.
\end{proof}

\subsection{Properties of the fast process}
In this section, we provide moment bounds for the fast process $(Y_t)_{t \ge 0}$. Compared to the estimates in the previous section, one has to take into account the movement of the slow variable and the resulting change in the drift and diffusion matrix for the fast variable.

Let us introduce the fast and slow time-scales
\begin{align*}
    \chi &\colon[0 , \infty )\to[0,\infty),\ \chi(t) = \int_{0}^t \frac{1}{\varepsilon_s}\, ds \\
    \psi &\colon [0,\infty) \to [0,\infty),\ \psi(t) = \chi^{-1}(t).
\end{align*}
with derivatives 
\begin{align*}
    \chi'(t) = \frac{1}{\varepsilon_t} \quad \text{and} \quad \psi'(t) = \varepsilon_{\psi(t)}.
\end{align*}

\begin{lemma} \label{lem:momentfast} 
    Let $\beta >0$ and $p \geq 2$. Suppose that Assumptions~\ref{assu:regularity} and~\ref{assu:epsilon} are satisfied and that the fast variable $(Y_t)_{t \ge 0}$ with initial condition $y \in \R^{d_2}$ is $p$-dissipative with constant $\beta$. Then, there exists $C_{\beta,p}^{(1)}, C_{\beta,p}^{(2)} > 0$ such that for all $t\geq 0$
    \begin{align*}
    \| Y_t\|_{L^p} \leq ( e^{-p(\beta  /2) \chi(t)}|y|^p + C_{\beta,p}^{(1)} )^{1/p} \leq e^{-(\beta  /2) \chi(t)}|y| + C_{\beta,p}^{(2)}.
    \end{align*}
\end{lemma}

\begin{proof}
    By a change of time, the fast variable satisfies
    \begin{align*}
        dY_{\psi(t)} = g(X_{\psi(t)},Y_{\psi(t)})dt +\sigma(X_{\psi(t)},Y_{\psi(t)})d\widetilde{W}_t
    \end{align*}
    where $\widetilde{W} := \int_0^{\psi(\cdot)} \frac{1}{\sqrt{\eps_s}}dW_s$ is an $\{\mathcal{F}_{\psi(t)}\}_{t\geq 0}$-Brownian motion. Since $g$ and $\sigma$ are uniformly bounded in $x$ and the fast variable is $p$-dissipative, one can follow the same arguments as in the proof of Lemma~\ref{lem:frozenmoment}, to show existence of a constant $C_{\beta,p}^{(1)} >0$ satisfying for all $t\ge 0$ that
    \begin{align*}
        \E [|Y_{\psi(t)}|^p] \leq& e^{-p(\beta/2)t} |y|^p  + C_{\beta,p}^{(1)}.
    \end{align*}
    The claim follows by resubstituting $t=\chi(u)$ so that $\psi(\chi(u))=u$.
\end{proof}

\subsection{The Poisson-equation}

We consider
\begin{align} \label{eq:Phiproofs}
    \Phi \colon \mathbb R^{d_1+d_2} \to \mathbb R^{d_1},\qquad
    \Phi(x,y) 
    = \int_0^\infty 
    \bar{f}(x)-P_t^x f(x,y)  dt.
\end{align}

\begin{lemma} \label{lem:barfLipschitz} 
    Suppose that Assumption~\ref{assu:frozenLipschitz} is satisfied, that $(\hat Y^x_t)_{t \geq 0}$ is $2$-dissipative and that $f$ is $L_f$-Lipschitz. Then,
    $\Phi: \R^{d_1+d_2} \to \R^{d_1}$ defined in \eqref{eq:Phiproofs} is well-defined and there exists a constant $C > 0$ such that  
    \begin{align} \label{eq:23r8u334gh94h9}
    |\Phi(x,y)| \leq  C(1+|y|) \quad \text{ for all } x\in \R^{d_1}, y \in \R^{d_2}.
    \end{align}
    Moreover, the averaged function $\bar{f}$ is Lipschitz-continuous.     
\end{lemma}

\begin{proof}
    By Lemma~\ref{lem:invariantmoment} there exists $C^{(1)}_\beta > 0$ with
\begin{align*}
    \sup_{x \in \R^{d_1}}\int |y'| \, \mu_x(dy') \leq \sup_{x \in \R^{d_1}} \Big(\int |y'|^2 \, \mu_x(dy')\Big)^{1/2}  \le C^{(1)}_\beta.
\end{align*}
Thus, 
\begin{align}
    \begin{split} \label{eq:weg9h34t79h8gd}
     |\bar{f}(x) -P_t^{x}f(x,y)| 
     &= \Big| \int f(x,y') \mu_x(dy') -P_t^xf(x,y) \Big|\\
     &= \Big| \int\left( P_t^x f(x,y')-P_t^xf(x,y) \right)\mu_x(dy') \Big| \\
     &\leq \int \hat{\E} [ |f(x, \hat{Y}^{x,y'}_t)-f(x,\hat{Y}^{x,y}_t)|] \, \mu_x(dy')\\
     &\leq L_f \int e^{-\beta t}|y-y'| \mu_x(dy') \\
     &\leq L_f e^{-\beta t} (|y|+C^{(1)}_\beta )
     \end{split}
\end{align}
for $x \in \R^{d_1}$ and $y \in \R^{d_2}$. This proves the well-definedness of $\Phi$ and \eqref{eq:23r8u334gh94h9}.
    
    Let us now prove Lipschitz-continuity of $\bar f$.
    Let $x_1,x_2 \in \R^{d_1}$, $y \in \R^{d_2}$ and $t \geq 0$. Then,
\begin{align} \begin{split}
    \label{eq:23047239hfi3}
    |\bar{f}(x_1)-\bar{f}(x_2)| \leq &| \bar{f}(x_1) -P_t^{x_1}f(x_1,y)| \\
    &+ |P_t^{x_1}f(x_1,y)-P_t^{x_2}f(x_2,y)| + |P_t^{x_2}f(x_2,y)-\bar{f}(x_2)|.
    \end{split}
\end{align}
The first and third term on the right-hand side of the latter inequality can be bounded by \eqref{eq:weg9h34t79h8gd}.
For the second term on the right-hand side of \eqref{eq:23047239hfi3}, by Lemma~\ref{lem:frozenmoment}, there exists a $C^{(2)}_\beta >0$ with
\begin{align*}
     |\hat{Y}^{x_1,y}_t-\hat{Y}^{x_2,y}_t|_{L^1(\hat{\P})} \leq |\hat{Y}^{x_1,y}_t-\hat{Y}^{x_2,y}_t|_{L^2(\hat{\P})} \leq C^{(2)}_\beta |x_1-x_2|,
\end{align*}
so that
\begin{align*}
    |P_t^{x_1}f(x_1,y)-P_t^{x_2}f(x_2,y)| &\leq \hat{\E} [|f(x_1,\hat{Y}^{x_1,y}_t)-f(x_2,\hat{Y}^{x_2,y}_t)|] \\
    &\leq L_f |x_1-x_2|+ L_f \hat{\E} [ |\hat{Y}^{x_1,y}_t-\hat{Y}^{x_2,y}_t| ] \\
    &\leq L_f(1+  C^{(2)}_\beta) |x_1-x_2|.
\end{align*}
Thus, letting $t \to \infty$ implies
\begin{align*}
    |\bar{f}(x_1)-\bar{f}(x_2)| \leq  L_f(1+  C^{(2)}_\beta)|x_1-x_2|.
\end{align*}
\end{proof}

Next, we show that $\Phi$ solves the Poisson equation. Note that in Lemma~\ref{___lem_D2yPhi} and Lemma~\ref{lem:DxPhi} we will show that, under Assumption~\ref{assu:regularity} and $2$-dissipativity of the fast variable, one has $\Phi \in C^{1,2}$ so that 
$$
    \mathcal{L}^1 \Phi(x,y) = \lim_{t\to 0} \frac{P_t^x\Phi(x,y)-\Phi(x,y)}{t}.
$$

\begin{lemma} \label{lem:Poisson} 
Suppose that Assumption~\ref{assu:frozenLipschitz} is satisfied, $f$ is $L_f$-Lipschitz, and the frozen variable is $2$-dissipative. Then, \(\Phi: \R^{d_1+d_2} \to \R^{d_1}\) satisfies
\[
\lim_{t\to 0} \frac{P_t^x\Phi(x,y)-\Phi(x,y)}{t} = f(x,y)-\bar{f}(x),
\]
for $x \in \R^{d_1}$ and $y \in \R^{d_2}$.
\end{lemma}

\begin{proof}
    Let $x \in \R^{d_1}$ and $y \in \R^{d_2}$. By the semigroup property $P_t^xP_s^x =P_{t+s}^x$, we get
\begin{align*}
    \lim_{t\to 0} \frac{P_t^x\Phi(x,y)-\Phi(x,y)}{t}
    &= \lim_{t\to 0} \frac{1}{t}  \left( -\hat{\E} \left[\int_0^\infty P_s^x f(x,\hat{Y}^{x,y}_t)-\bar f(x) ds\right]+ \int_0^\infty P_s^x f(x,y)-\bar f(x) ds\right) \\
    &= \lim_{t\to 0} \frac{1}{t}  \left( - \int_0^\infty P_t^x(P_s^x f(x,\cdot))(y)-\bar f(x)ds+ \int_0^\infty P_s^x f(x,y)-\bar f(x) ds\right) \\
    &= \lim_{t\to 0} \frac{1}{t}  \left( -\int_t^\infty P_s^x f(x,y)-\bar f(x) ds+ \int_0^\infty P_s^x f(x,y)-\bar f(x) ds\right) \\
    &= \frac{d}{dt} \left. \int_0^t P_s^x f(x,y)-\bar f(x) ds \right\vert_{t=0} = f(x,y)-\bar f(x).
\end{align*}
\end{proof}

\subsection{Regularity of the solution to the Poisson equation}
Lastly, we provide regularity results for the corresponding Poisson equation depending on the regularity of the coefficients of the fast SDE. In the proofs, we will make use of the following alternative formulation of the p-dissipativity condition from Definition~\ref{def:frozendissi}, namely that for all $x \in \R^{d_1}$ and $y,v \in \R^{d_2}$ we have
    \begin{align} \label{___dissipativity_equiv} % Will be used quite often
    \begin{split}
        &\langle v,D_y g(x,y) v \rangle + \frac{p-1}{2} \sum_{k=1}^{d_2} |D_y \sigma_k(x,y) v |^2 \\
        &= \lim\limits_{h \to 0} \frac{\langle v, g(x,y+hv)-g(x,y) \rangle}{h} + \lim\limits_{h \to 0} \frac{p-1}{2} \sum_{k=1}^{d_2} \frac{| \sigma_k(x,y+hv) - \sigma_k(x,y) |^2}{h^2} \\
        &= \lim\limits_{h \to 0} \frac{1 }{h^2} \bigg( \langle hv, g(x,y+hv)-g(x,y) \rangle +  \frac{p-1}{2} |\sigma(x,y+hv) - \sigma(x,y)|_F^2 \bigg) \\
        &\leq \limsup\limits_{h \to 0} \frac{1}{h^2} (-\beta) |hv|^2= -\beta |v|^2.
    \end{split}
    \end{align}

\begin{lemma} \label{___lem_DyFlow}
Suppose that Assumption~\ref{assu:frozenLipschitz} holds and that $g$ and $\sigma_k$, $k=1,\dots,d_2$, are twice continuously differentiable in $y \in \R^{d_2}$. Assume, moreover, that $D_y g$, $D_y \sigma_k$, $D^2_y g$, and $D^2_y \sigma_k$ are jointly continuous and uniformly bounded on $\R^{d_1+d_2}$.
    Then the mapping $(x,y) \mapsto \hat{Y}^{x,y}$ is twice differentiable in $y$ in the sense that there are families of processes $D_y \hat{Y}^{x,y} v$ and $D^2_y \hat{Y}^{x,y} (v_1,v_2)$ indexed with $(x,y )\in \R^{d_1+d_2}$ and $v \in\R^{d_2}$ and $v_1,v_2 \in\R^{d_2}$, respectively, satisfying
    \begin{align*}
        \lim\limits_{h \to 0} &\hat{\E} \bigg[ \sup_{t \in [0,T]} \bigg| \frac{\hat{Y}^{x,y+hv}_t - \hat{Y}^{x,y}_t }{h} -D_y \hat{Y}^{x,y}_t v\bigg|^2 \bigg] = 0, \\
        \lim\limits_{h \to 0} &\hat{\E} \bigg[ \sup_{t \in [0,T]} \bigg| \frac{D_y \hat{Y}^{x,y+hv_2}_t v_1 - D_y \hat{Y}^{x,y}_t v_1}{h} -D^2_y \hat{Y}^{x,y}_t (v_1,v_2)\bigg|^2 \bigg] = 0
    \end{align*}
    for all $x \in \R^{d_1}$, $y,v,v_1,v_2 \in \R^{d_2}$, and any finite $T >0$ where the first-order derivative flow $D_y \hat{Y}^{x,y} v$ is the unique strong solution to
    \begin{align*}
    \begin{cases}
        d\zeta^1_t = D_yg(x,\hat{Y}^{x,y}_t)\zeta^1_t  dt + \sum_{k=1}^{d_2} D_y\sigma_k(x,\hat{Y}^{x,y}_t) \zeta^1_t d\hat{W}^k_t, \\
        \zeta^1_0 = v
    \end{cases}
    \end{align*}
    and the second-order derivative flow $D^2_y \hat{Y}^{x,y} (v_1,v_2) = D^2_y \hat{Y}^{x,y} (v_2,v_1)$ is the unique strong solution to
    \begin{align*}
    \begin{cases}
        d\zeta^2_t = \Big(D_y g(x,\hat{Y}^{x,y}_t) \zeta^2_t + D_y^2 g(x,\hat{Y}^{x,y}_t) (D_y\hat{Y}^{x,y}_t v_1,D_y\hat{Y}^{x,y}_t v_2) \Big) dt
        \\
        \qquad \qquad +\sum_{k=1}^{d_2} \Big( D_y\sigma_k(x,\hat{Y}^{x,y}_t)\zeta^2_t + D_y^2 \sigma_k(x,\hat{Y}^{x,y}_t) (D_y\hat{Y}^{x,y}_t v_1,D_y\hat{Y}^{x,y}_t v_2)\Big) d\hat{W}^k_t,
        \\
        \zeta^2_0 =0.
    \end{cases}
    \end{align*}
    Moreover, for each $p\geq 2$ and finite time horizon $T > 0$, both processes are uniformly $L^p$-bounded in $x,y$, i.e. there exists $C_{p,T} > 0$ with
    \begin{align*}
        \hat\E \Big[ \sup_{0\leq t \leq T}|D_y \hat{Y}^{x,y}_t v|^p \Big] &\leq C_{p,T}|v|^p
        \\
        \hat\E \Big[ \sup_{0\leq t \leq T}|D^2_y \hat{Y}^{x,y}_t (v_1,v_2)|^p\Big] &\leq C_{p,T} |v_1|^p|v_2|^p
    \end{align*}
    for $ x \in \R^{d_1}$, $y,v,v_1,v_2 \in \R^{d_2}$, and they are continuous in $x,y$ with respect to the $L^p$-norm, i.e.
    \begin{align*}
        \lim\limits_{(x',y') \to (x,y)} &\hat{\E} \Big[ \sup_{0 \leq t \leq T} | D_y \hat{Y}^{x',y'}_t v - D_y \hat{Y}^{x,y}_t v|^p \Big] = 0, \\
        \lim\limits_{(x',y') \to (x,y)} &\hat{\E} \Big[ \sup_{0 \leq t \leq T} | D^2_y \hat{Y}^{x',y'}_t (v_1,v_2) - D^2_y \hat{Y}^{x,y}_t (v_1,v_2)|^p \Big] = 0
    \end{align*}
    for $ x\in \R^{d_1}$, $y,v,v_1,v_2 \in \R^{d_2}$. Additionally, if $\hat Y^x$ is $p$-dissipative for some $p \geq 2$ with constant $\beta> 0$, then
    \begin{align*}
        \|D_y \hat{Y}^{x,y}_t v\|_{L^p} \leq e^{-\beta  t}|v|.
    \end{align*}
    for all $x \in \R^{d_1}$, $y,v \in \R^{d_2}$, and $t \geq 0$. 
\end{lemma}

\begin{proof} The differentiability assertions follow from a standard differentiability theorem for solution maps of SDEs; see e.g. Theorem~6.5 in~\cite{marinelli2020frechet}. It remains to prove the moment bounds, continuous dependence on $(x,y)$, and the dissipativity estimates.

    The proof of the uniform $L^p$-boundedness and the continuity in $(x,y) \in \R^{d_1+d_2}$ follows the same line of argument as in Lemma~\ref{___lem_Y} for both processes. We define the constants $C^{(1)}_p , C^{(1)}_2>0$ by the BDG inequality as in \eqref{___Y_BDG_p} and let $C^{(2)} > 0$ be the bound for the derivatives. 
    First, let $\zeta := \zeta^{x,y} := \zeta^{x,y,v} := D_y \hat Y^{x,y} v$ for $x \in \R^{d_1}$, $y,v \in \R^{d_2}$ and consider the localizing sequence
    \begin{align*}
        \tau_N := \tau_N^{x,y,v} := \inf\{ t \geq 0: |\zeta_t | \geq N\}, \quad N\in\N.
    \end{align*}
    Fix $T \geq 0$ and $p \geq 2$. By the BDG and the Hölder inequalities, there exists a constant $C^{(3)}_{p,T} > 0$ such that
    \begin{align*}
        \hat \E \Big[\sup_{0 \leq s \leq t \wedge \tau_N} |\zeta_s|^p \Big] 
        &\leq 
        (d_2+2)^{p-1} \bigg(|v|^p + \hat \E \bigg[ \bigg( \int_0^{t \wedge \tau_N} |D_y g(x, \hat Y^{x,y}_s) \zeta_s | ds \bigg)^p \bigg] \\ &\qquad + C^{(1)}_p \sum_{k=1}^{d_2} \hat \E \bigg[ \bigg(\int_0^{t \wedge \tau_N} |D_y \sigma_k (x, \hat Y^{x,y}_s) \zeta_s|^2 ds \bigg)^{p/2} \bigg]\bigg)
        \\
        &\leq (d_2+2)^{p-1} \bigg( |v|^p + T^{p-1} (C^{(2)})^p \int_0^{t} \hat \E \Big[ \sup_{0 \leq r \leq s \wedge \tau_N} |\zeta_r|^p \Big] ds \\
        &\qquad + d_2 T^{p/2-1} C^{(1)}_p (C^{(2)})^p \int_0^t \hat \E \Big[ \sup_{0 \leq r \leq s \wedge \tau_N} |\zeta_r|^p \Big] ds \bigg)
        \\
        &\leq C^{(3)}_{p,T} |v|^p + C^{(3)}_{p,T} \int_0^t \hat \E \Big[ \sup_{0 \leq r \leq s \wedge \tau_N} |\zeta_r|^p \Big] ds
    \end{align*} 
    holds for all $0 \leq t \leq T$ and $x\in\R^{d_1}$, $y,v \in \R^{d_2}$ and a constant $C^{(3)}_{p,T} >0$. By Fatou's and Grönwall's lemma
    \begin{align*}
        \hat \E \Big[\sup_{0 \leq t \leq T} |\zeta_t|^p \Big] \leq \liminf_{N \to \infty} \hat \E \Big[\sup_{0 \leq t \leq T \wedge \tau_N} |\zeta_t|^p \Big] \leq  C^{(3)}_{p,T} e^{C^{(3)}_{p,T} T} |v|^p
    \end{align*}
    for all $x \in \R^{d_1}$, $y,v\in\R^{d_2}$. 

    Again, we only show the continuity with respect to $L^2$-convergence. The general case follows by the argument as in Lemma~\ref{___lem_Y}. First, one can find a constant $C^{(4)}_T >0$ with
    \begin{align*}
        \hat \E \Big[ \sup_{0 \leq s \leq t} |\zeta^{x',y'}_s - \zeta^{x,y}_s|^2 \Big] 
        &\leq 
        (d_2+1) \bigg( \hat \E \bigg[ \bigg( \int_0^t |D_y g(x', \hat Y^{x',y'}_s) \zeta^{x',y'}_s - D_y g(x, \hat Y^{x,y}_s) \zeta^{x,y}_s| ds \bigg)^2 \bigg]
        \\
        &\quad + C^{(1)}_2 \sum_{k=1}^{d_2} \hat \E \bigg[ \int_0^t |D_y \sigma_k (x', \hat Y^{x',y'}_s) \zeta^{x',y'}_s - D_y \sigma_k (x, \hat Y^{x,y}_s) \zeta^{x,y}_s |^2 ds \bigg] \bigg)
        \\
        &\leq
        2 (C^{(2)})^2 (d_2+1) T \int_0^t \hat \E \Big[\sup_{0 \leq r \leq s}|\zeta^{x',y'}_r - \zeta^{x,y}_r|^2 \Big] ds \\
        &\quad + 2(d_2+1) T \int_0^T \hat \E [|(D_y g(x', \hat Y^{x',y'}_s) - D_y g(x, \hat Y^{x,y}_s)) \zeta^{x,y}_s|^2] ds \\
        &\quad + 2(C^{(2)})^2 C^{(1)}_2 (d_2+1)d_2  \int_0^t \hat \E \Big[\sup_{0 \leq r \leq s}|\zeta^{x',y'}_r - \zeta^{x,y}_r|^2 \Big] ds \\
        &\quad + 2C^{(1)}_2(d_2+1) \sum_{k=1}^{d_2} \int_0^T \hat \E [|(D_y \sigma_k(x', \hat Y^{x',y'}_s) - D_y \sigma_k(x, \hat Y^{x,y}_s)) \zeta^{x,y}_s|^2] ds
        \\
        &\leq C^{(4)}_T \int_0^{T} \hat \E[Z^{x',y',x,y}_s] ds + C^{(4)}_T  \int_0^t \hat \E \Big[\sup_{0 \leq r \leq s}|\zeta^{x',y'}_r - \zeta^{x,y}_r|^2 \Big] ds
    \end{align*}
    for all $0\leq t \leq T$ and $x',x \in \R^{d_1}$, $y',y,v\in\R^{d_2}$ such that Grönwall's inequality gives
    \begin{align*}
        \hat \E \Big[ \sup_{0 \leq t \leq T} |\zeta^{x',y'}_t - \zeta^{x,y}_t|^2 \Big] \leq C^{(4)}_T e^{C^{(4)}_T T} \int_0^T \hat \E[Z^{x',y',x,y}_s] ds 
    \end{align*}
    where we set
    \begin{align*}
        Z^{x',y',x,y}_t 
        = |(D_y g(x', \hat Y^{x',y'}_t) - D_y g(x, \hat Y^{x,y}_t)) \zeta^{x,y}_t|^2 +  \sum_{k=1}^{d_2} |(D_y \sigma_k(x', \hat Y^{x',y'}_t) - D_y \sigma_k(x, \hat Y^{x,y}_t)) \zeta^{x,y}_t|^2
    \end{align*}
    for all $0\leq t \leq T$.
    From the uniform $L^4$-boundedness of $(\zeta^{x,y}_t)_{t \in [0,T]}$ in $(x,y)\in\R^{d_1+d_2}$, the uniform $L^2$-boundedness of $(Z^{x',y',x,y}_t)_{t \in [0,T]}$ in $(x',y'),(x,y)\in\R^{d_1+d_2}$ for each $v\in\R^{d_2}$ follows. Hence, by Fatou's lemma and Vitali's convergence theorem,
    \begin{align*}
        \limsup_{(x',y')\to (x,y)} \hat \E \Big[ \sup_{0 \leq t \leq T} |\zeta^{x',y'}_t - \zeta^{x,y}_t|^2 \Big] \leq C^{(4)}_T e^{C^{(4)}_T T} \int_0^T \limsup_{(x',y')\to (x,y)} \hat \E[ Z^{x',y',x,y}_s] ds = 0
    \end{align*}
    for all $x\in \R^{d_1}$, $y,v\in\R^{d_2}$.
    
    The argument for the second-order derivative $\zeta : = \zeta^{x,y} : =\zeta^{x,y,v_1,v_2}  := D^2_y \hat Y^{x,y} (v_1,v_2)$, $x\in \R^{d_1}$, $y,v_1,v_2\in\R^{d_2}$ is analogous. Abbreviate
    \begin{align*}
        \widetilde \zeta := \widetilde \zeta^{x,y} := (\widetilde \zeta^{x,y,v_1}, \widetilde \zeta^{x,y,v_2}) := (D_y \hat Y^{x,y} v_1,D_y \hat Y^{x,y} v_2)
    \end{align*}
    and define the localizing sequence
    \begin{align*}
        \tau_N := \tau_N^{x,y,v_1,v_2} := \inf\{ t \geq 0 : |\zeta_t| \geq N\}, \quad N\in\N.
    \end{align*}
    for $x \in \R^{d_1}$ and $y,v_1,v_2 \in \R^{d_2}$. Then as before, a constant $C^{(5)}_{p,T} > 0$ exists with
    \begin{align*}
        \hat \E \Big[\sup_{0 \leq s \leq t \wedge \tau_N} |\zeta_s|^p \Big] 
        &\leq (2d_2 + 2)^{p-1} \bigg(\hat \E \bigg[  \bigg(\int_0^{t \wedge \tau_N} | D_y g(x, \hat Y^{x,y}_s) \zeta_s | ds \bigg)^p \bigg]  
        \\&\qquad + \hat \E \bigg[ \bigg( \int_0^{t \wedge \tau_N} |D^2_y g(x, \hat Y^{x,y}_s) (\widetilde \zeta_s)| ds \bigg)^p \bigg] \\ 
        &\qquad + C^{(1)}_p \sum_{k=1}^{d_2} \hat \E \bigg[ \bigg(\int_0^{t \wedge \tau_N} |D_y \sigma_k(x, \hat Y^{x,y}_s) \zeta_s|^2 ds \bigg)^{p/2} \bigg] 
        \\ &\qquad + C^{(1)}_p \sum_{k=1}^{d_2} \hat \E \bigg[ \bigg(\int_0^{t \wedge \tau_N} |D^2_y \sigma_k(x, \hat Y^{x,y}_s) (\widetilde \zeta_s)|^2 ds \bigg)^{p/2} \bigg] \bigg) 
        \\
        &\leq C^{(5)}_{p,T} \int_0^{t} \hat \E \Big[ \sup_{0 \leq r \leq s \wedge \tau_N} |\zeta_r|^p \Big] ds + C^{(5)}_{p,T} \sup_{0 \leq s \leq T} \hat \E [  |\widetilde \zeta^{x,y,v_1}_s|^p |\widetilde \zeta^{x,y,v_2}_s|^p ].
    \end{align*}
    for all $0 \leq t \leq T$. By the previously shown $L^p$-boundedness of $D_y \hat Y^{x,y}v_i$, $i=1,2$ and Hölder's inequality, we have
    \begin{align*}
        \sup_{0 \leq s \leq T} \hat \E [  |\widetilde \zeta^{x,y,v_1}_s|^p |\widetilde \zeta^{x,y,v_2}_s|^p ]
        \leq \sup_{0 \leq s \leq T} (\hat \E [  |\widetilde \zeta^{x,y,v_1}_s|^{2p}] \hat \E [ |\widetilde \zeta^{x,y,v_2}_s|^{2p} ])^{1/2} \leq C^{(6)}_{p,T}|v_1|^p|v_2|^p
    \end{align*}
    for some $C^{(6)}_{p,T}> 0$ and by Fatou's lemma and Grönwall's inequality
    \begin{align*}
        \hat \E \Big[\sup_{0 \leq t \leq T } |\zeta_t|^p \Big] \leq \liminf_{N \to \infty} \hat \E \Big[\sup_{0 \leq t \leq T \wedge \tau_N} |\zeta_t|^p \Big] \leq C^{(5)}_{p,T}C^{(6)}_{p,T} e^{C^{(5)}_{p,T} T}|v_1|^p|v_2|^p.
    \end{align*}
    for all $x \in \R^{d_1}$, $y,v_1,v_2 \in \R^{d_2}$.

    For the continuity of $D_y^2 \hat Y^{x,y}_t (v_1,v_2)$ with respect to $(x,y)\in\R^{d_1+d_2}$ note that there is a constant $C^{(7)}_{T} > 0$ such that for each $x',x \in \R^{d_1}$, $y',y,v_1,v_2 \in \R^{d_2}$, we have
    \begin{align*}
        \hat \E \Big[\sup_{0\leq s \leq t} |\zeta^{x',y'}_s - \zeta^{x,y}_s|^2 \Big]
        &\leq (2d_2+2) \bigg( \hat \E \bigg[ \bigg(\int_0^t |D_y g(x',\hat Y^{x',y'}_s)\zeta^{x',y'}_s - D_yg(x,\hat Y^{x,y}_s) \zeta^{x,y}_s | ds \bigg)^2 \bigg] \\
        &\quad +  \hat\E \bigg[ \bigg(\int_0^t |D^2_y g(x',\hat Y^{x',y'}_s)(\widetilde\zeta^{x',y'}_s) - D^2_y g(x,\hat Y^{x,y}_s)(\widetilde \zeta^{x,y}_s)| ds \bigg)^2 \bigg] \\
        &\quad + C^{(1)}_2 \sum_{k=1}^{d_2} \hat \E \bigg[ \int_0^t | D_y \sigma_k (x',\hat Y^{x',y'}_s)\zeta^{x',y'}_s - D_y \sigma_k(x,\hat Y^{x,y}_s) \zeta^{x,y}_s |^2 ds \bigg] \\
        &\quad +C^{(1)}_2 \sum_{k=1}^{d_2} \hat \E \bigg[  \int_0^t | D^2_y \sigma_k (x',\hat Y^{x',y'}_s)(\widetilde \zeta^{x',y'}_s) - D^2_y \sigma_k(x,\hat Y^{x,y}_s) (\widetilde \zeta^{x,y}_s) |^2 ds \bigg] \bigg) \\
        &\leq C^{(7)}_T  \int_0^t \hat \E \Big[\sup_{0\leq r \leq s} |\zeta^{x',y'}_r - \zeta^{x,y}_r|^2 \Big] ds + C^{(7)}_T \int_0^T \hat \E [ Z^{x',y',x,y}_s ] ds
    \end{align*}
    for all $0 \leq t \leq T$ and by Grönwall's inequality
    \begin{align*}
        \hat \E \Big[\sup_{0\leq t \leq T} |\zeta^{x',y'}_t - \zeta^{x,y}_t|^2  \Big] \leq C^{(7)}_T e^{C^{(7)}_T T} \int_0^T \hat\E [ Z^{x',y',x,y}_s] ds
    \end{align*}
    where
    \begin{align*}
        Z^{x',y',x,y}_t 
        &= 
        |( D_y g(x',\hat Y^{x',y'}_t) - D_yg(x,\hat Y^{x,y}_t) ) \zeta^{x,y}_t |^2 \\
        &\qquad+ |D^2_y g(x',\hat Y^{x',y'}_t)(\widetilde\zeta^{x',y'}_t) - D^2_y g(x,\hat Y^{x,y}_t)(\widetilde \zeta^{x,y}_t)|^2
        \\
        &\qquad + \sum_{k=1}^{d_2} |( D_y \sigma_k(x',\hat Y^{x',y'}_t) - D_y\sigma_k(x,\hat Y^{x,y}_t) ) \zeta^{x,y}_t |^2\\
        &\qquad + \sum_{k=1}^{d_2} |D^2_y \sigma_k (x',\hat Y^{x',y'}_t)(\widetilde \zeta^{x',y'}_t) - D^2_y \sigma_k(x,\hat Y^{x,y}_t) (\widetilde \zeta^{x,y}_t) |^2
    \end{align*}
    for all $0 \leq t \leq T$.
    Now, since we have already shown that for each $v_1,v_2 \in \R^{d_2}$, the processes $(\hat Y^{x,y}_t)_{t \in [0,T]}$, $(\widetilde \zeta_t^{x,y,v_i})_{t \in [0,T]}$, $i = 1,2$, and $(\zeta^{x,y}_t)_{t \in [0,T]}$ are uniformly (or at least locally) $L^8$-bounded and continuous with respect to convergence in probability in $(x,y) \in \R^{d_1+d_2}$, also $Z^{x',y',x,y}$ is uniformly $L^2$-bounded and thus converges in $L^1$ by the convergence theorem of Vitali. Finally, by Fatou's lemma
    \begin{align*}
        \limsup_{(x',y')\to(x,y)}\hat \E \Big[\sup_{0\leq t \leq T} |\zeta^{x',y'}_t - \zeta^{x,y}_t|^2 \Big] \leq C^{(7)}_{T} e^{C^{(7)}_{T} T} \int_0^T \limsup_{(x',y')\to(x,y)} \hat\E [Z^{x',y',x,y}_s] ds =0
    \end{align*}
    for all $x\in\R^{d_1}$, $y,v_1,v_2 \in \R^{d_2}$.

    For the exponential convergence, assume $p$-dissipativity for $p \geq 2$ and fix $x\in\R^{d_1}$ and $y,v\in\R^{d_2}$. Then applying the Itô formula to $\zeta = D_y \hat Y^{x,y} v$ under expectation gives
    \begin{align*}
        \hat \E [|\zeta_t |^p] 
        &=|v|^p + p \int_0^t \hat{\E} [|\zeta_s |^{p-2} \langle \zeta_s,D_yg(x,\hat{Y}^{x,y}_s)\zeta_s  \rangle ] ds\\
        &\qquad + \frac{p}{2} \sum_{k=1}^{d_2} \int_0^t \hat{\E} [ |\zeta_s |^{p-2} |D_y\sigma_k(x,\hat{Y}^{x,y}_s)\zeta_s|^2 ] ds \\
        &\qquad + \frac{p(p-2)}{2} \sum_{k=1}^{d_2} \int_0^t \hat{\E} [ |\zeta_s |^{p-4}  \langle \zeta_s,D_y\sigma_k(x,\hat{Y}^{x,y}_s)\zeta_s\rangle^2] ds \\
        &\qquad + p\sum_{k=1}^{d_2} \hat\E \bigg[ \int_0^t |\zeta_s |^{p-2} \langle \zeta_s,D_y \sigma_k(x,\hat{Y}^{x,y}_s)\zeta_s  \rangle d \hat W^k_s \bigg]
    \end{align*}
    for $t \geq 0$ where the Brownian integral inside the last term is a true martingale by the previously established $L^{2p}$-boundedness and thus has expectation $0$.
    By \eqref{___dissipativity_equiv}, this expression turns into
    \begin{align*}
        \frac{d}{dt}\hat{\E}[|\zeta_t |^p] &\leq p\hat{\E}\Big[|\zeta_t|^{p-2} \Big( \langle \zeta_t,D_yg(x,\hat{Y}^{x,y}_t)\zeta_t  \rangle + \frac{p-1}{2}\sum_{k=1}^{d_2} |D_y\sigma_k(x,\hat{Y}^{x,y}_t)\zeta_t|^2 \Big)\Big]
        \leq -p\beta \hat{\E}[|\zeta_t |^p]
    \end{align*}
    for all $t \geq 0$. Then Grönwall's differential inequality yields the desired inequality
    \begin{align*}
        \hat{\E}[|\zeta_t |^p] \leq e^{-p\beta t} |v|^p 
    \end{align*}
    for all $t \geq 0$.
\end{proof}

\begin{lemma}\label{___lem_DxFlow}
    Suppose that Assumption~\ref{assu:frozenLipschitz} holds, that $g$ and $\sigma_k$, $k=1,\dots,d_2$, are continuously differentiable in $x \in \R^{d_1}$ as well as in $y\in \R^{d_2}$, and that the mixed second-order derivatives $D^2_{y,x} g$ and $D^2_{y,x} \sigma_k$ exist. Assume moreover that $D_x g$, $D_x \sigma_k$, $D_y g$, $D_y \sigma_k$, $D^2_{y,x} g$, and $D^2_{y,x} \sigma_k$ are jointly continuous and uniformly bounded on $\R^{d_1+d_2}$.
    Then the mapping $(x,y) \mapsto \hat{Y}^{x,y}$ is differentiable in $x$ in the sense that there is a family of processes $D_x \hat Y^{x,y} v$ indexed with $x,v\in \R^{d_1}$, $y\in \R^{d_2}$ satisfying
    \begin{align} \label{eq:Dxconv}
        \lim\limits_{h \to 0} &\hat{\E} \bigg[ \sup_{t \in [0,T]} \bigg| \frac{\hat{Y}^{x + hv,y}_t - \hat{Y}^{x,y}_t }{h} -D_x \hat{Y}^{x,y}_t v\bigg|^2 \bigg] = 0
    \end{align}
    for any finite $T >0$ where the first-order derivative flow $D_x \hat{Y}^{x,y} v$ is an instance of the general dynamics
    \begin{align} \label{eq:DxSDE}
        \begin{cases}
        d \zeta^{x,y,z,v}_t = D_x g(x,\hat{Y}^{x,y}_t) v dt + D_y g(x,\hat{Y}^{x,y}_t) \zeta^{x,y,z,v}_t dt \\ \qquad \qquad \qquad + \sum_{k=1}^{d_2} \Big( D_x \sigma_k(x,\hat{Y}^{x,y}_t) v + D_y \sigma_k(x,\hat{Y}^{x,y}_t) \zeta^{x,y,z,v}_t \Big) d \hat{W}^k_t, \\
        \zeta^{x,y,z,v}_0 = z \in \R^{d_2}
        \end{cases}
    \end{align}
    for $z=0$, i.e. $D_x \hat{Y}^{x,y} v = \zeta^{x,y,0,v}$. Moreover for each $p \geq 2$ and finite time horizon $T > 0$, the process is uniformly bounded in $x,y$, i.e. there exists $C_{p,T}> 0$ with
    \begin{align*}
        \hat\E \Big[ \sup_{0\leq t \leq T}|\zeta^{x,y,z,v}_t|^p \Big] &\leq C_{p,T}( |v|^p + |z|^p )
    \end{align*}
    for $x,v \in \R^{d_1}$, $y,z \in \R^{d_2}$ and continuous in $x,y$ with respect to convergence in $L^p$, i.e.
    \begin{align*}
        \lim\limits_{(x',y') \to (x,y)} &\hat{\E} \Big[ \sup_{t \in [0,T]} | \zeta^{x',y',z,v}_t - \zeta^{x,y,z,v}_t|^p \Big] = 0.
    \end{align*}
    for $x,v \in \R^{d_1}$, $y,z \in \R^{d_2}$. Additionally, if $\hat Y^{x}$ is $p$-dissipative for $p \geq 2$ with constant $\beta > 0$, then there exists $ C_{\beta,p}^{(1)},C_{\beta,p}^{(2)} > 0$ with
    \begin{align*}
        \|\zeta^{x,y,z,v}_t\|_{L^p} \leq \{ e^{-(p\beta /2)t}|z|^p + C_{\beta,p}^{(1)}|v|^p \}^{1/p} \leq e^{-(\beta  /2)t} |z| + C_{\beta,p}^{(2)}|v|
    \end{align*}
    for all $x,v \in \R^{d_1}$, $y,z \in \R^{d_2}$, and $t \geq 0$.
\end{lemma}

\begin{proof}
Set $V_t^{x,y}=(x,\hat Y_t^{x,y})$. A standard differentiability theorem for SDEs (see, e.g., Theorem 6.5 in~\cite{marinelli2020frechet}) shows that for all $T\ge 0$, $(x,y) \mapsto (V_t^{x,y})_{t \in [0,T]} \in L_2(C([0,T],\R^{d_1+d_2}))$ is differentiable and $DV_t^{x,y}(v,z)=(v, \zeta_t^{x,y,z,v})$.
Therefore, \eqref{eq:Dxconv} and \eqref{eq:DxSDE} hold.

The argument is analogous to Lemma~\ref{___lem_DyFlow}.
Recall the constants $C^{(1)}_p,C^{(1)}_2 > 0$ from \eqref{___Y_BDG_p}, and denote by $C^{(2)}>0$ a bound for the derivatives of the coefficients.
Let $\zeta := \zeta^{x,y} :=\zeta^{x,y,z,v}$ for $x,v \in \R^{d_1}$, $y,z \in \R^{d_2}$ and consider the localizing sequence $(\tau_N)_{N\in\N}$ defined by 
    \begin{align*}
        \tau_N := \inf\{ t \geq 0: |\zeta_t | \geq N\}.
    \end{align*}
    Then there exists a constant $C^{(3)}_{p,T} > 0$ such that
    \begin{align*}
        \hat \E \Big[\sup_{0 \leq s \leq t \wedge \tau_N} |\zeta_s|^p \Big] 
        &\leq 
        (2d_2+3)^{p-1} \bigg(|z|^p + \hat \E \bigg[\bigg(\int_0^{t\wedge \tau_N} |D_x g(x,\hat Y^{x,y}_s)v|ds \bigg)^p \bigg]
        \\ 
        &\qquad 
        + \hat \E \bigg[ \bigg(\int_0^{t \wedge \tau_N} |D_y g(x, \hat Y^{x,y}_s) \zeta_s |ds \bigg)^p \bigg] 
        \\
        &\qquad 
        + C^{(1)}_p \sum_{k=1}^{d_2} \hat \E \bigg[ \bigg(\int_0^{t \wedge \tau_N} |D_x \sigma_k (x, \hat Y^{x,y}_s) v|^2 ds\bigg)^{p/2} \bigg]
        \\ 
        &\qquad 
        +C^{(1)}_p \sum_{k=1}^{d_2} \hat \E \bigg[ \bigg( \int_0^{t \wedge \tau_N} |D_y \sigma_k (x, \hat Y^{x,y}_s) \zeta_s|^2 ds \bigg)^{p /2} \bigg]
        \bigg)
        \\
        &\leq C^{(3)}_{p,T} (|z|^p + |v|^p) + C^{(3)}_{p,T} \int_0^t \hat \E \Big[ \sup_{0 \leq r \leq s \wedge \tau_N} |\zeta_r|^p \Big] ds.
    \end{align*} 
    holds for all $0 \leq t \leq T$ and as a consequence of Fatou's and Grönwall's lemmas,
    \begin{align*}
        \hat \E \Big[\sup_{0 \leq s \leq T} |\zeta_s|^p \Big] \leq \liminf_{N \to \infty} \hat \E \Big[\sup_{0 \leq s \leq T \wedge \tau_N} |\zeta_s|^p \Big] \leq  C^{(3)}_{p,T} (|z|^p + |v|^p) e^{C^{(3)}_{p,T}T}
    \end{align*}
    for any $x,v \in \R^{d_1}$, $y,z \in \R^{d_2}$.

To show continuity
\begin{align*}
    \hat \E \Big[ \sup_{0 \leq t \leq T} |\zeta^{x',y'}_t - \zeta^{x,y}_t|^p \Big] \longrightarrow 0, \quad (x',y')\to (x,y)
\end{align*}
for $p \geq 2$ it suffices to prove the convergence for the case $p =2$ as the $L^p$-boundedness for arbitrary $p \geq 2$ implies the general case by the Vitali's convergence theorem as in Lemma~\ref{___lem_Y} and Lemma~\ref{___lem_DyFlow}. By the BDG inequality
\begin{align*}
    \hat \E \Big[ \sup_{0 \leq s \leq t} |\zeta^{x',y'}_s - \zeta^{x,y}_s|^2 \Big]
    &\leq 
    (2d_2+2) \bigg( \hat \E \bigg[ \bigg( \int_0^t|D_x g(x',\hat Y^{x',y'}_s) v - D_x g(x,\hat Y^{x,y}_s) v| ds \bigg)^2 \bigg] \\
    &\quad  
    + \hat \E \bigg[ \bigg( \int_0^t|D_y g(x',\hat Y^{x',y'}_s) \zeta^{x',y'}_s - D_y g(x,\hat Y^{x,y}_s) \zeta^{x,y}_s| ds \bigg)^2 \bigg]
    \\
    &\quad  + C^{(1)}_2 \sum_{k=1}^{d_2} \hat \E \bigg[  \int_0^t |D_x \sigma_k (x',\hat Y^{x',y'}_s)v - D_x \sigma_k (x,\hat Y^{x,y}_s)v|^2 ds \bigg]
    \\&\quad  + C^{(1)}_2\sum_{k=1}^{d_2} \hat \E \bigg[ \int_0^t |D_y \sigma_k (x',\hat Y^{x',y'}_s)\zeta^{x',y'}_s - D_y \sigma_k (x,\hat Y^{x,y}_s)\zeta^{x,y}_s|^2 ds \bigg]
    \bigg) 
    \\ &
    \leq C^{(4)}_T \int_0^T \hat \E [Z^{x',y',x,y}_s] ds + C^{(4)}_T \int_0^t \hat \E \Big[ \sup_{0 \leq r \leq s} |\zeta^{x',y'}_r - \zeta^{x,y}_r|^2 \Big] ds
\end{align*}
for $0\leq t \leq T$ and thus by Grönwall's lemma
\begin{align*}
    \hat \E \Big[ \sup_{0 \leq t \leq T} |\zeta^{x',y'}_t - \zeta^{x,y}_t|^2 \Big] \leq C^{(4)}_T e^{C^{(4)}_T T}\int_0^T \hat \E [Z^{x',y',x,y}_t]dt
\end{align*}
where
\begin{align*}
    Z^{x',y',x,y}_t 
    &
    = |D_x g(x',\hat Y^{x',y'}_t)v-D_xg(x, \hat Y^{x,y}_t)v|^2 
    \\&\qquad
    + |( D_y g(x',\hat Y^{x',y'}_t) - D_y g(x,\hat Y^{x,y}_t) )\zeta^{x,y}_t|^2 
    \\&\qquad
    + \sum_{k=1}^{d_2}  |D_x \sigma_k(x',\hat Y^{x',y'}_t)v - D_x \sigma_k(x,\hat Y^{x,y}_t)v|^2
    \\&\qquad
    +  \sum_{k=1}^{d_2} |(D_y \sigma_k (x',\hat Y^{x',y'}_t) - D_y \sigma_k (x,\hat Y^{x,y}_t)) \zeta^{x,y}_t|^2 
\end{align*}
for $0\leq t \leq T$, $x',x,v \in \R^{d_1}$, and $y',y,z \in \R^{d_2}$. Note that by the uniform $L^4$-boundedness of $\zeta^{x,y,z,v}$ in $x,y$, Fatou's lemma and Vitali's convergence theorem give
\begin{align*}
    \limsup_{(x',y')\to (x,y)} \hat \E \Big[ \sup_{0 \leq t \leq T} |\zeta^{x',y'}_t - \zeta^{x,y}_t|^2 \Big] \leq C^{(4)}_T e^{C^{(4)}_TT}\int_0^T \limsup_{(x',y')\to (x,y)}\hat \E [Z^{x',y',x,y}_t]dt =0
\end{align*}
for any $x,v \in \R^{d_1}$ and $y,z \in \R^{d_2}$.

Now, assume that $\hat Y^x$ is $p$-dissipative and fix $x,v\in\R^{d_1}$, $y,z \in \R^{d_2}$.
Then an application of the Itô formula under expectation yields
\begin{align*}
&\hat{\E}[|\zeta^{x,y,z,v}_t|^p]
\\
&=
|z|^p  + p\int_0^t \hat{\E}\Big[ |\zeta^{x,y,z,v}_s|^{p-2} \Big\langle \zeta^{x,y,z,v}_s, D_x g(x,\hat{Y}^{x,y}_s)v + D_y g(x,\hat{Y}^{x,y}_s) \zeta^{x,y,z,v}_s \Big\rangle \Big] ds \\
&\qquad + \frac{p}{2} \sum_{k=1}^{d_2} \int_0^t \hat{\E}\Big[ |\zeta^{x,y,z,v}_s|^{p-2} | D_x \sigma_k(x,\hat{Y}^{x,y}_s)v + D_y \sigma_k(x,\hat{Y}^{x,y}_s) \zeta^{x,y,z,v}_s |^2 \Big] ds \\
&\qquad
+
\frac{p(p-2)}{2} \sum_{k=1}^{d_2} \int_0^t \hat{\E}\Big[ |\zeta^{x,y,z,v}_s|^{p-4} \Big\langle \zeta^{x,y,z,v}_s, D_x \sigma_k(x,\hat{Y}^{x,y}_s)v +
D_y \sigma_k(x,\hat{Y}^{x,y}_s) \zeta^{x,y,z,v}_s \Big\rangle^2 \Big] ds \\
&\qquad + p\sum_{k=1}^{d_2} \hat \E \bigg[ \int_0^t  |\zeta^{x,y,z,v}_s|^{p-2} \Big\langle \zeta^{x,y,z,v}_s, D_x \sigma_k(x,\hat{Y}^{x,y}_s)v + D_y \sigma_k(x,\hat{Y}^{x,y}_s) \zeta^{x,y,z,v}_s \Big\rangle d \hat W^k_s \bigg]
\end{align*}
for $t \geq 0$ where the stochastic integral inside the expectation of the last line is a true martingale.
By taking the derivative and applying \eqref{___dissipativity_equiv} as well as Young's inequality \eqref{___Young_ineq}, one finds $C^{(5)}_{\beta,p} >0$ such that 
\begin{align*}
    \frac{d}{dt} \hat{\E}[|\zeta^{x,y,z,v}_t|^p] 
    &\leq
    p \hat{\E}\Big[ |\zeta^{x,y,z,v}_t|^{p-2} \Big\langle \zeta^{x,y,z,v}_t, D_x g(x,\hat{Y}^{x,y}_t)v + D_y g(x,\hat{Y}^{x,y}_t) \zeta^{x,y,z,v}_t \Big\rangle \Big] \\
    &\qquad
    + \frac{p(p-1)}{2} \sum_{k=1}^{d_2} \hat{\E}\Big[ |\zeta^{x,y,z,v}_t|^{p-2} | D_x \sigma_k(x,\hat{Y}^{x,y}_t)v + D_y \sigma_k(x,\hat{Y}^{x,y}_t) \zeta^{x,y,z,v}_t |^2 \Big]
    \\
    &\leq
    -p\beta \hat{\E}[|\zeta^{x,y,z,v}_t|^{p}] 
    + (pC^{(2)}+d_2 p (p-1) (C^{(2)})^2) \hat{\E}[|\zeta^{x,y,z,v}_t|^{p-1}]|v| 
    \\
    &\qquad + \frac{d_2 p(p-1)}{2} (C^{(2)})^2\hat{\E}[|\zeta^{x,y,z,v}_t|^{p-2}]|v|^2 
    \\
    &\leq -\frac{p\beta}{2} \hat{\E}[|\zeta^{x,y,z,v}_t|^{p}] + C^{(5)}_{\beta,p} |v|^p.
\end{align*}
Hence, by Grönwall's differential inequality
\begin{align*}
    \hat{\E}[|\zeta^{x,y,z,v}_t|^p] \leq |z|^p e^{-(p\beta /2)t} + \frac{2C^{(5)}_{\beta,p}}{p\beta}|v|^p (1-e^{-(p\beta /2)t}) \leq |z|^p e^{-(p\beta /2)t} + \frac{2C^{(5)}_{\beta,p}}{p\beta}|v|^p
\end{align*}    
for all $t \geq 0$.
\end{proof}

\begin{lemma}\label{___lem_DyPhi}
Let $\beta>0$ and assume that $(\hat Y^x_t)_{t \ge 0}$ is $2$-dissipative with constant $\beta$. Suppose that $f$ is $L_f$-Lipschitz and differentiable in $y$ with $D_y f$ bounded and jointly continuous on $\R^{d_1+d_2}$. Under the hypotheses of Lemma~\ref{___lem_DyFlow}, $\Phi$ is differentiable in $y$, its derivative is jointly continuous on $\R^{d_1+d_2}$, and
    \begin{align*}
        |D_y \Phi (x,y)| \leq \frac{L_f}{\beta} 
    \end{align*}
    for all $x \in \R^{d_1}$ and $y \in \R^{d_2}$.
\end{lemma}

\begin{proof}
    We set
    \begin{align}\label{___DyPhi_F_def}
        F \colon [0,\infty)\times\R^{d_1+d_2} \to \R^{d_1}\quad ,\quad F(t,x,y) = P_t^x [f(x,\cdot)](y) = \hat \E[f(x, \hat Y^{x,y}_t)]
    \end{align} 
    and recall the definition
    \begin{align*}
        \Phi(x,y) = \int_0^\infty \bar{f}(x)-F(t,x,y) \,  dt = \int_0^\infty \bar{f}(x) - \hat{\E} [f(x,\hat{Y}^{x,y}_t)] dt, \qquad (x,y)\in\R^{d_1+d_2}.
    \end{align*}
    The proof of the differentiability is carried out by consecutively interchanging the order of differentiation and integration with respect to the Lebesgue measure over $[0,\infty)$ and with respect to $\hat{\P}$. For $t \geq 0$, $h \in (-1,1) \setminus \{0\}$, $x \in \R^{d_1}$, and $y,v \in \R^{d_2}$, we consider
    \begin{align*}
        \frac{F(t,x,y+hv) - F(t,x,y)}{h} = \hat \E \bigg[ \frac{f(x,\hat{Y}^{x,y+hv}_t) - f(x,\hat{Y}^{x,y}_t)}{h} \bigg].
    \end{align*}
    Using the mean value theorem, we get for $y_1,y_2 \in \R^{d_2}$
    \begin{align} \label{eq:239r348z9}
        f(x,y_2) - f(x,y_1) = \int_0^1 D_y f(x,y_1+r (y_2-y_1)) (y_2-y_1) dr,
    \end{align}
    so that
    \begin{align}\label{___DyPhi_0}
    \begin{split}
        &\bigg\| \frac{f(x,\hat{Y}^{x,y+hv}_t) - f(x,\hat{Y}^{x,y}_t)}{h} - D_yf(x,\hat Y^{x,y}_t) D_y \hat Y^{x,y}_t v\bigg\|_{L^1} 
        \\
        &
        \leq \int_0^1 \bigg\|\Big( D_y f(x,\hat Y^{x,y}_t+r (\hat{Y}^{x,y+hv}_t-\hat{Y}^{x,y}_t)) - D_yf(x,\hat Y^{x,y}_t ) \Big) \frac{\hat{Y}^{x,y+hv}_t-\hat{Y}^{x,y}_t}{h}\bigg\|_{L^1} dr 
        \\
        &\qquad +\bigg\|  D_yf(x,\hat Y^{x,y}_t ) \Big( \frac{\hat{Y}^{x,y+hv}_t-\hat{Y}^{x,y}_t}{h} - D_y \hat Y^{x,y}_t v \Big) \bigg\|_{L^1}.
    \end{split}
    \end{align}
    By Lemma \ref{___lem_DyFlow} and $|D_y f | \leq L_f$, the latter term vanishes for $h\to 0$. Now, Lemma~\ref{lem:frozenmoment} implies
    \begin{align*}
        &\bigg\| \Big|D_y f(x,\hat Y^{x,y}_t+r (\hat{Y}^{x,y+hv}_t-\hat{Y}^{x,y}_t)) - D_yf(x,\hat Y^{x,y}_t )\Big| \bigg|\frac{\hat{Y}^{x,y+hv}_t-\hat{Y}^{x,y}_t}{h} \bigg| \bigg\|_{L^2}
        \leq 2L_f e^{-\beta t}|v|
    \end{align*}
    for $h \in (-1,1)\setminus\{0\}$, and $r \in [0,1]$. Thus, the term inside the $L^2$-norm on the left-hand side is uniformly integrable with respect to $h$. The same term also converges in probability for $h \to 0$ by the continuity of $D_y f$ and, thus, it converges in $L^1$ by Vitali's convergence theorem. As the right-hand side in the latter inequality does not depend on $r$, we can use the dominated convergence theorem to get convergence of the first term on the right-hand side in \eqref{___DyPhi_0}. Hence, 
    \begin{align*}
        \bigg\| \frac{f(x,\hat{Y}^{x,y+hv}_t) - f(x,\hat{Y}^{x,y}_t)}{h} - D_y f(x,\hat Y^{x,y}_t) D_y \hat Y^{x,y}_t v \bigg\|_{L^1} \overset{h \to 0}{\longrightarrow} 0,
    \end{align*}
    which implies existence of a directional derivative of $F$ at $(x,y)$ in direction $y$, and
    \begin{align} \label{___DyPhi_F}
        D_y F(t,x,y) v = \lim\limits_{h\to0} \hat \E \bigg[ \frac{f(x,\hat{Y}^{x,y+hv}_t) - f(x,\hat{Y}^{x,y}_t)}{h} \bigg] = \hat \E[ D_y f(x, \hat Y^{x,y}_t ) D_y \hat Y^{x,y}_t v].
    \end{align}
    For the remaining claims note that by \eqref{eq:239r348z9} and Lemma~\ref{lem:frozenmoment},
    \begin{align} \label{___DyPhi_1}
        \bigg| \frac{F(t,x,y+hv)-F(t,x,y)}{h} \bigg|
        \leq \bigg\| \frac{f(x,\hat{Y}^{x,y+hv}_t) - f(x,\hat{Y}^{x,y}_t)}{h} \bigg\|_{L^2} 
        \leq L_f  e^{-\beta t}|v|
    \end{align}
    and that \eqref{___DyPhi_F} as well as Lemma~\ref{___lem_DyFlow} give
    \begin{align} \label{___DyPhi_2}
        |D_y F(t,x,y)v|
        \leq \| D_y f(x,\hat Y^{x,y}_t) D_y \hat Y^{x,y}_t v \|_{L^2} 
        \leq L_f  e^{-\beta t}|v|
    \end{align}
    for all $t \geq 0$, $h \in (-1,1) \setminus\{0\}$, $x\in \R^{d_1}$, and $y , v \in \R^{d_2}$.
    Since for each $t \geq 0$ and $v \in \R^{d_2}$,
    \begin{align*}
        D_y f(x, \hat Y^{x,y}_t ) D_y \hat Y^{x,y}_t v
    \end{align*}
    is continuous in $(x,y) \in \R^{d_1+d_2}$ with respect to convergence in probability and uniformly integrable by \eqref{___DyPhi_2}, continuity of $D_y F(t,\cdot,\cdot)v$ holds by Vitali's convergence theorem. 

    By the dominated convergence theorem, the existence and the continuity of the directional derivative $D_y \Phi(x,y)v$ follows from \eqref{___DyPhi_1} and \eqref{___DyPhi_2}, respectively. Moreover, \eqref{___DyPhi_2} yields 
    \begin{align*}
        |D_y \Phi(x,y) v|  = \bigg| \int_0^\infty D_y F(t,x,y) v \,  dt \bigg| \leq \frac{L_f}{\beta} |v|
    \end{align*}
    for all $x \in \R^{d_1}$ and $y,v \in \R^{d_2}$.
\end{proof}

\begin{lemma}\label{___lem_D2yPhi}
    Let $\beta>0$ and assume that $(\hat Y^x_t)_{t \ge 0}$ is $2$-dissipative with constant $\beta$. Suppose that $f$ is $L_f$-Lipschitz and twice differentiable in $y$ with $D_y f$ as well as $D^2_y f$ bounded and jointly continuous on $\R^{d_1+d_2}$. Under the hypotheses of Lemma~\ref{___lem_DyFlow}, $\Phi$ is twice differentiable in $y$, its second-order derivative is jointly continuous on $\R^{d_1+d_2}$, and there exists a constant $C_\beta >0$ with
    \begin{align*}
        |D^2_y \Phi (x,y)| \leq C_{\beta}
    \end{align*}
    for all $x \in \R^{d_1}$ and $y \in \R^{d_2}$
\end{lemma}

\begin{proof}
    Define $F$ as in \eqref{___DyPhi_F_def}. As before we consider difference quotients
    \begin{align}
    \label{___D2yPhi_1}
    \begin{split}
        &\frac{D_y F(t,x,y+hv_2)v_1 - D_y F(t,x,y)v_1}{h} \\
        &= \hat{\E} \bigg[ \frac{D_y f(x,\hat{Y}^{x,y+hv_2}_t) D_y \hat{Y}^{x,y+hv_2}_t v_1 - D_y f(x,\hat{Y}^{x,y}_t) D_y \hat{Y}^{x,y}_t v_1}{h} \bigg] \\
        &= \hat\E \bigg[D_y f(x,\hat{Y}^{x,y+hv_2}_t) \frac{D_y \hat{Y}^{x,y+hv_2}_t v_1 -D_y \hat{Y}^{x,y}_t v_1 }{h} \bigg] \\
        &\qquad + \hat \E \bigg[ \frac{D_y f(x,\hat{Y}^{x,y+hv_2}_t) - D_y f(x,\hat{Y}^{x,y}_t)}{h}D_y \hat{Y}^{x,y}_t v_1\bigg]
    \end{split}
    \end{align}
    for $t \geq 0$, $h \in (-1,1) \setminus \{0\}$, $x \in \R^{d_1}$, and $y,v_1,v_2 \in \R^{d_2}$. 
    Fix $t \geq 0$, $x \in \R^{d_1}$, $y,v_1,v_2\in \R^{d_2}$ and let $C^{(1)} >0$ be the bound $C^{(1)} \geq |D^2_y f|$.
    Then by passing $h \to 0$, a combination of the $L^2$-convergence
    \begin{align*}%\label{___D2yPhi4}
        \||D_y f(x,\hat{Y}^{x,y+hv_2}_t) - D_y f(x,\hat{Y}^{x,y}_t)| \|_{L^2} \leq C^{(1)} \|\hat{Y}^{x,y+hv_2}_t - \hat{Y}^{x,y}_t\|_{L^2} \longrightarrow 0
    \end{align*}
    due to Lemma~\ref{lem:frozenmoment} and the $L^2$-convergence 
    \begin{align*}
         \bigg\| \frac{D_y \hat{Y}^{x,y+hv_2}_t v_1 -D_y \hat{Y}^{x,y}_t v_1 }{h} - D_y^2\hat{Y}^{x,y}_t(v_1,v_2)  \bigg\|_{L^2} \longrightarrow 0
    \end{align*}
    due to Lemma~\ref{___lem_DyFlow}, yields the convergence of the first term of \eqref{___D2yPhi_1} by
    \begin{align*}
        \bigg\| D_y f(x,\hat{Y}^{x,y+hv_2}_t)\frac{D_y \hat{Y}^{x,y+hv_2}_t v_1 -D_y \hat{Y}^{x,y}_t v_1 }{h} - D_y f(x,\hat{Y}^{x,y}_t)D_y^2\hat{Y}^{x,y}_t(v_1,v_2)  \bigg\|_{L^1} \longrightarrow 0.
    \end{align*}
    For the second term, the second order mean value theorem
    \begin{align*}
        D_y f(x,y_2) v - D_y f(x,y_1) v = \int_0^1 D_y^2 f(x,y_1+r (y_2-y_1)) (v,y_2-y_1) dr, \qquad y_1,y_2,v \in \R^{d_2}
    \end{align*}
    gives
    \begin{align*}
        &\bigg\| \frac{D_y f(x,\hat{Y}^{x,y+hv_2}_t) - D_y f(x,\hat{Y}^{x,y}_t)}{h}D_y \hat{Y}^{x,y}_t v_1 - D_y^2 f(x,\hat{Y}^{x,y}_t) (D_y \hat{Y}^{x,y}_t v_1 , D_y \hat{Y}^{x,y}_t v_2) \bigg\|_{L^1} \\
        &\leq \int_0^1\bigg\|  \Big(D_y^2 f\Big(x,\hat{Y}^{x,y}_t + r (\hat{Y}^{x,y+hv_2}_t-\hat{Y}^{x,y}_t)\Big) - D_y^2 f(x, \hat{Y}^{x,y}_t) \Big) \bigg(D_y \hat{Y}^{x,y}_t v_1, \frac{\hat{Y}^{x,y+hv_2}_t-\hat{Y}^{x,y}_t}{h} \bigg) \bigg\|_{L^1}  \\&\qquad dr \\
        &\qquad + \bigg\| D_y^2 f(x, \hat{Y}^{x,y}_t) \bigg( D_y \hat{Y}^{x,y}_t v_1,  D_y \hat{Y}^{x,y}_t v_2 - \frac{\hat{Y}^{x,y+hv_2}_t-\hat{Y}^{x,y}_t}{h} \bigg) \bigg\|_{L^1}
    \end{align*}
    for $h \in (-1,1) \setminus \{0\}$.
    Convergence to $0$ for $h\to 0$ follows with the same arguments used for \eqref{___DyPhi_0}. Note that
    \begin{align*}
        \Big(D_y^2 f\Big(x,\hat{Y}^{x,y}_t + r (\hat{Y}^{x,y+hv_2}_t-\hat{Y}^{x,y}_t)\Big) - D_y^2 f(x, \hat{Y}^{x,y}_t) \Big) \bigg(D_y \hat{Y}^{x,y}_t v_1, \frac{\hat{Y}^{x,y+hv_2}_t-\hat{Y}^{x,y}_t}{h} \bigg),
    \end{align*}
    $h \in (-1,1)\setminus\{0\}$, is uniformly integrable since for sufficiently small $\alpha >0$ its $L^{1+\alpha}$-norm has a bound
    \begin{align}\label{___D2yPhi_2}
        2C^{(1)}  \|D_y \hat Y^{x,y}_t v_1  \|_{L^{2+2\alpha}} \bigg\| \frac{\hat{Y}^{x,y+hv_2}_t-\hat{Y}^{x,y}_t}{h}  \bigg\|_{L^{2+2\alpha}} \leq 2 C^{(1)} e^{ - (\beta /2 + \beta /2)t}|v_1||v_2|
    \end{align}
    that is uniform in $h \in (-1,1) \setminus\{0\}$ by Remark~\ref{rem:dissibigger}, though one can show that a uniform bound still holds if $\alpha =1$ at the cost of an additional coefficient that diverges for $t \to \infty$. Thus, we conclude that the second order directional derivative of $F$ exists, namely
    \begin{align}\label{___D2yPhi_3}
    \begin{split}
        &D_y [D_y F(t,x,\cdot)v_1](y)v_2 \\
        &\qquad = \lim\limits_{h \to 0} \frac{D_y F(t,x,y+hv_2)v_1 - D_y F(t,x,y)v_1}{h} \\
        &\qquad = \hat{\E}[ D_y f(x,\hat{Y}^{x,y}_t) D_y^2\hat{Y}^{x,y}_t(v_1,v_2)] + \hat{\E}[D_y^2 f(x,\hat{Y}^{x,y}_t) (D_y \hat{Y}^{x,y}_t v_1 , D_y \hat{Y}^{x,y}_t v_2) ].
    \end{split}
    \end{align}

    For the continuity fix $v_1,v_2 \in\R^{d_2}$ and $t \geq 0$. Since $\R^{d_1+d_2} \ni (x,y) \mapsto D_y f(x, \hat Y^{x,y}_t)$ is continuous with respect to convergence in probability and a.s. uniformly bounded, Vitali's convergence theorem yields continuity in the sense of $L^2$-convergence. Together with Lemma~\ref{___lem_DyFlow}, the continuity of the first expectation 
    \begin{align*}
        \| D_y f(x',\hat{Y}^{x',y'}_t) D_y^2\hat{Y}^{x',y'}_t(v_1,v_2) - D_y f(x,\hat{Y}^{x,y}_t) D_y^2\hat{Y}^{x,y}_t(v_1,v_2)\|_{L^1} \longrightarrow 0
    \end{align*}
    for $(x',y') \to(x,y) \in \R^{d_1+d_2}$ follows. Again, one can bound the $L^{1+\alpha}$-norm of the term inside the second expectation of \eqref{___D2yPhi_3} by \eqref{___D2yPhi_2} uniformly for sufficiently small $\alpha >0$. Thus, Vitali's convergence theorem also gives continuity of the remaining term
    \begin{align*}
        \|D_y^2 f(x',\hat{Y}^{x',y'}_t) (D_y \hat{Y}^{x',y'}_t v_1 , D_y \hat{Y}^{x',y'}_t v_2) - D_y^2 f(x,\hat{Y}^{x,y}_t) (D_y \hat{Y}^{x,y}_t v_1 , D_y \hat{Y}^{x,y}_t v_2) \|_{L^1} \longrightarrow 0
    \end{align*}
    for any $(x',y') \to (x,y) \in \R^{d_1+d_2}$. Combining these two limits, we obtain the continuity of
    \begin{align*}
        D^2_y F(t,x,y)(v_1,v_2)  = D_y [D_y F(t,x,\cdot)v_1](y)v_2
    \end{align*}
    in $(x,y) \in \R^{d_1+d_2}$.

    For the continuous differentiability of $\Phi$, let $C^{(2)} > 0$ be a bound for the derivatives of the coefficients. The aim is to find a bound for \eqref{___D2yPhi_1} that is uniform in $h \in (-1,1) \setminus\{0\}$ and exponentially decaying in $t \geq 0$ so that differentiability follows by the theorem of dominated convergence. In particular, we are concerned with controlling
    \begin{align}
        \bigg|\hat\E \bigg[D_y f(x,\hat{Y}^{x,y+hv_2}_t) \frac{D_y \hat{Y}^{x,y+hv_2}_t v_1 -D_y \hat{Y}^{x,y}_t v_1 }{h} \bigg] \bigg| \leq C^{(2)}\hat \E \bigg[ \bigg| \frac{D_y \hat{Y}^{x,y+hv_2}_t v_1 -D_y \hat{Y}^{x,y}_t v_1 }{h} \bigg| \bigg]
    \end{align}
    for $x \in \R^{d_1}$, $y,v_1,v_2 \in \R^{d_2}$, $h \in (-1,1)\setminus \{0\}$, and $t \geq 0$. Fix $x\in \R^{d_1}$, $y,v_1,v_2 \in \R^{d_2}$, $h \in (-1,1) \setminus \{0\}$ and set
    \begin{align*}
        Z_t := D_y \hat Y^{x,y + h v_2}_t v_1 - D_y \hat Y^{x,y}_t v_1, \quad \widetilde Z_t := \hat Y^{x,y + h v_2}_t - \hat Y^{x,y}_t, \quad \zeta_t := D_y \hat Y^{x,y}_t v_1.
    \end{align*}
    for $t \geq 0$. We consider an auxiliary function
    \begin{align*}
        \rho \colon \R^{3d_2} \to [0,\infty),\ \rho(z_1,z_2,z_3) = |z_1|^2 + |z_2|^2|z_3|^2
    \end{align*}
    and
    \begin{align*}
        \rho_t := \rho(Z_t,\widetilde Z_t, \zeta_t), \qquad t \geq 0.
    \end{align*}
    Given a fixed arbitrary smoothing parameter $\gamma>0$, we have
    \begin{align*}
        \hat \E [|Z_t|] \leq \hat \E [(|Z_t|^2 +|\widetilde Z_t|^2|\zeta_t|^2 + \gamma^2)^{1/2}] =  \hat \E [(\rho_t + \gamma^2)^{1/2}]
    \end{align*}
    for $t \geq 0$. Hence, it suffices to show the exponential convergence of the right-hand-side up to an additional constant that vanishes with $\gamma \downarrow 0$.
    By Itô's formula, we have
    \begin{align*}
        d\rho_t= 2 \langle Z_t ,dZ_t \rangle +d[Z]_t + |\widetilde Z_t|^2 (2\langle \zeta_t , d \zeta_t \rangle +d[\zeta]_t) + |\zeta_t|^2 (2\langle \widetilde Z_t , d \widetilde Z_t \rangle +d[\widetilde Z]_t) + d [|\widetilde Z|^2,|\zeta|^2]_t 
    \end{align*}
    where $[\cdot]$ is defined in \eqref{___Y_trace_covariation}, and
    \begin{align*}
        d(\rho_t + \gamma^2)^{1/2} = \frac{1}{2 (\rho_t + \gamma^2)^{1/2}} d \rho_t - \frac{1}{8(\rho_t + \gamma^2)^{3/2}} d[\rho]_t
    \end{align*}
    for $t \geq 0$. Thus, combining both dynamics under expectation yields
    \begin{align}\label{___D2yPhi_4}
    \begin{split}
        \hat{\E}[(\rho_t + \gamma^2)^{1/2}] 
        &= 
        (h^2|v_1|^2|v_2|^2 + \gamma^2)^{1/2} + \hat \E \bigg[\int_0^t \frac{2\langle Z_s,dZ_s \rangle +d[Z]_s}{2(\rho_s+\gamma^2)^{1/2}}  \bigg] 
        \\
        &\qquad + \hat \E \bigg[ \int_0^t |\widetilde Z_s|^2 \frac{2\langle \zeta_s,d\zeta_s \rangle +d[\zeta]_s}{2(\rho_s+\gamma^2)^{1/2}} \bigg] + \hat \E \bigg[ \int_0^t |\zeta_s|^2 \frac{2\langle \widetilde Z_s,d\widetilde Z_s \rangle + d[\widetilde Z]_s}{2(\rho_s+\gamma^2)^{1/2}} \bigg]
        \\
        &\qquad + \hat \E \bigg[ \int_0^t \frac{d[|\widetilde Z|^2,|\zeta|^2]_s }{2(\rho_s+\gamma^2)^{1/2}}\bigg] - \hat \E \bigg[ \int_0^t \frac{d[\rho]_s}{8(\rho_s + \gamma^2)^{3/2}} \bigg]
    \end{split}
    \end{align}
    Note that due to the continuity of the drift and diffusion coefficients of $\rho$ the last term
    \begin{align*}
        R \colon [0,\infty) \to [0,\infty),\ R_t = \hat \E \bigg[ \int_0^t \frac{d[\rho]_s}{8(\rho_s + \gamma^2)^{3/2}} \bigg]
    \end{align*}
    is a continuously differentiable non-decreasing function, i.e.
    \begin{align}\label{___D2yPhi_4.5}
        \frac{dR_t}{dt} \geq 0
    \end{align}
    for $t\geq 0$. For the second-to-last term, recall the dynamics
    \begin{align*}
        d \widetilde Z_t = (g(x,\hat Y^{x,y + h v_2}_t) - g(x,\hat Y^{x,y}_t))dt +(\sigma(x,\hat Y^{x,y + h v_2}_t) - \sigma(x,\hat Y^{x,y}_t)) d\hat W_t
    \end{align*}
    and
    \begin{align*}
        d \zeta_t = D_y g(x,\hat Y^{x,y}_t) \zeta_t dt + \sum_{k=1}^{d_2} D_y \sigma_k(x,\hat Y^{x,y}_t) \zeta_t d \hat W^k_t.
    \end{align*}
    for $t \geq 0$. Then 
    \begin{align*}
         \hat \E \bigg[ \int_0^t \frac{d[|\widetilde Z|^2,|\zeta|^2]_s }{2(\rho_s+\gamma^2)^{1/2}}\bigg] = 2\int_0^t \sum_{k=1}^{d_2} \hat \E \bigg[ \frac{\langle \widetilde Z_s , \sigma_k(x,\hat Y^{x,y+hv_2}_s) - \sigma_k (x, \hat Y^{x,y}_s) \rangle \langle \zeta_s , D_y \sigma_k(x,\hat Y^{x,y}_s) \zeta_s \rangle}{(|Z_s|^2 + |\widetilde Z_s|^2|\zeta_s|^2 +\gamma^2)^{1/2}} \bigg] ds 
    \end{align*}
    allows differentiation with
    \begin{align}\label{___D2yPhi_5}
    \begin{split}
        \frac{d}{dt} \hat \E \bigg[ \int_0^t \frac{d[|\widetilde Z|^2,|\zeta|^2]_s }{2(\rho_s+\gamma^2)^{1/2}}\bigg]
        &\leq  
        2d_2 \hat \E \bigg[ \frac{ (C^{(2)})^2 |\widetilde Z_t|^2 |\zeta_t|^2 }{(|Z_t|^2 + |\widetilde Z_t|^2|\zeta_t|^2 +\gamma^2)^{1/2}} \bigg] \\
        &\leq C^{(3)} \hat{\E} [|\widetilde Z_t||\zeta_t|] \\
        &\leq C^{(3)} \| \widetilde Z_t \|_{L^2} \| \zeta_t\|_{L^2}
    \end{split}
    \end{align}
    for all $t\geq 0$ where $C^{(3)} := 2 d_2 (C^{(2)})^2$. The next term can be written as
    \begin{align*}
        &\hat \E \bigg[ \int_0^t |\zeta_s|^2 \frac{2\langle \widetilde Z_s,d\widetilde Z_s \rangle + d[\widetilde Z]_s}{2(\rho_s+\gamma^2)^{1/2}} \bigg]
        \\
        & = \int_0^t \hat \E \bigg[ \frac{|\zeta_s|^2}{(\rho_s + \gamma^2)^{1/2}} \Big(\langle  \widetilde{Z}_s ,g(x,\hat Y^{x,y + h v_2}_s) - g(x,\hat Y^{x,y}_s)\rangle + \frac{1}{2} |\sigma(x,\hat Y^{x,y + hv_2}_s) -\sigma(x,\hat Y^{x,y}_s)|_F^2 \Big) \bigg]ds
        \\
        & \qquad + \sum_{k=1}^{d_2} \hat \E \bigg[ \int_0^t \frac{|\zeta_s|^2\langle \widetilde Z_s , \sigma_k(x,\hat Y^{x,y + h v_2}_s) - \sigma_k(x,\hat Y^{x,y}_s) \rangle}{(\rho_s + \gamma^2)^{1/2}}  d\hat W^k_s \bigg]
    \end{align*}
    for $t \geq 0$ where the local martingale term is a true martingale as the expression inside the Brownian integral is an $L^2$-integrand. Hence, one can differentiate
    \begin{align}\label{___D2yPhi_6}
    \begin{split}
        \frac{d}{dt} \hat \E \bigg[ \int_0^t |\zeta_s|^2 \frac{2\langle \widetilde Z_s,d\widetilde Z_s \rangle + d[\widetilde Z]_s}{2(\rho_s+\gamma^2)^{1/2}} \bigg] 
        &\leq -\beta \hat{\E} \bigg[ \frac{|\zeta_t|^2 |\widetilde Z_t|^2}{(\rho_t+\gamma^2)^{1/2}} \bigg]
    \end{split}
    \end{align}
    for $t \geq 0$. In the same manner,
    \begin{align*}
        &\hat \E \bigg[ \int_0^t |\widetilde Z_s|^2 \frac{2\langle \zeta_s,d\zeta_s \rangle + d[\zeta]_s}{2(\rho_s+\gamma^2)^{1/2}} \bigg]
        \\
        & = \int_0^t \hat \E \bigg[ \frac{|\widetilde Z_s|^2}{(\rho_s + \gamma^2)^{1/2}} \Big(\langle  \zeta_s ,D_y g(x,\hat Y^{x,y}_s) \zeta_s \rangle + \frac{1}{2} \sum_{k=1}^{d_2} |D_y \sigma_k(x,\hat Y^{x,y}_s) \zeta_s|^2 \Big)\bigg]ds
        \\
        & \qquad + \sum_{k=1}^{d_2} \hat \E \bigg[ \int_0^t \frac{|\widetilde Z_s|^2\langle \zeta_s , D_y \sigma_k(x,\hat Y^{x,y}_s) \zeta_s \rangle}{(\rho_s + \gamma^2)^{1/2}}  d\hat W^k_s \bigg]
    \end{align*}
    and thus,
    \begin{align}\label{___D2yPhi_7}
    \begin{split}
        \frac{d}{dt} \hat \E \bigg[ \int_0^t |\widetilde Z_s|^2 \frac{2\langle \zeta_s,d\zeta_s \rangle + d[\zeta]_s}{2(\rho_s+\gamma^2)^{1/2}} \bigg] 
        &\leq -\beta \hat{\E} \bigg[ \frac{|\zeta_t|^2 |\widetilde Z_t|^2}{(\rho_t+\gamma^2)^{1/2}} \bigg]
    \end{split}
    \end{align}
    for $t \geq 0$. Now, $Z$ follows the dynamics
    \begin{align*}
        dZ_t &= \Big(D_y g(x,\hat{Y}^{x,y+hv_2}_t)D_y\hat{Y}^{x,y+hv_2}_t v_1 - D_y g(x,\hat{Y}^{x,y}_t)D_y\hat{Y}^{x,y}_t v_1 \Big) dt \\
        &\qquad +\sum_{k=1}^{d_2} \Big(D_y \sigma_k(x,\hat{Y}^{x,y+hv_2}_t)D_y\hat{Y}^{x,y+hv_2}_t v_1 - D_y \sigma_k(x,\hat{Y}^{x,y}_t)D_y\hat{Y}^{x,y}_t v_1 \Big)d\hat{W}^k_t
        \\
        &= \Big(D_y g(x,\hat{Y}^{x,y+hv_2}_t) Z_t + (D_y g(x,\hat{Y}^{x,y+hv_2}_t) - D_y g(x,\hat{Y}^{x,y}_t)) \zeta_t \Big) dt
        \\
        &\qquad + \sum_{k=1}^{d_2}\Big(D_y \sigma_k(x,\hat{Y}^{x,y+hv_2}_t)Z_t + (D_y \sigma_k(x,\hat{Y}^{x,y+hv_2}_t)-D_y \sigma_k(x,\hat{Y}^{x,y}_t))\zeta_t\Big) d\hat W^k_t
    \end{align*}
    for $t \geq 0$. Hence, the first term decomposes into
    \begin{align*}
        &\hat \E \bigg[\int_0^t \frac{2\langle Z_s,dZ_s \rangle +d[Z]_s}{2(\rho_s+\gamma^2)^{1/2}}  \bigg] 
        \\
        & = \int_0^t \hat \E \bigg[ \frac{1}{(\rho_s + \gamma^2)^{1/2}} \Big( \langle  Z_s ,  D_y g(x,\hat{Y}^{x,y+hv_2}_s) Z_s \rangle + \frac{1}{2} \sum_{k=1}^{d_2} |D_y \sigma_k(x,\hat Y^{x,y+hv_2}_s) Z_s|^2 \Big) \bigg] ds
        \\
        &\qquad + \int_0^t \hat \E \bigg[ \frac{1}{(\rho_s + \gamma^2)^{1/2}} \langle Z_s, (D_y g(x,\hat{Y}^{x,y+hv_2}_s) - D_y g(x,\hat{Y}^{x,y}_s)) \zeta_s \rangle  \bigg] ds
        \\
        &\qquad  + \int_0^t \sum_{k=1}^{d_2} \hat \E \bigg[ \frac{1}{(\rho_s + \gamma^2)^{1/2}} \langle D_y \sigma_k(x,\hat Y^{x,y+hv_2}_s) Z_s, (D_y \sigma_k(x,\hat{Y}^{x,y+hv_2}_s)-D_y \sigma_k(x,\hat{Y}^{x,y}_s))\zeta_s \rangle  \bigg] ds 
        \\
        &\qquad  + \frac{1}{2} \int_0^t \sum_{k=1}^{d_2} \hat \E \bigg[ \frac{1}{(\rho_s + \gamma^2)^{1/2}} |(D_y \sigma_k(x,\hat{Y}^{x,y+hv_2}_s)-D_y \sigma_k(x,\hat{Y}^{x,y}_s))\zeta_s|^2  \bigg] ds \\
        &\qquad + \sum_{k=1}^{d_2} \hat \E\bigg[ \int_0^t \frac{ \langle Z_s , D_y \sigma_k(x,\hat{Y}^{x,y+hv_2}_s)Z_s + (D_y \sigma_k(x,\hat{Y}^{x,y+hv_2}_s)-D_y \sigma_k(x,\hat{Y}^{x,y}_s))\zeta_s \rangle}{(\rho_s + \gamma^2)^{1/2}}  d \hat W^k_s\bigg]
    \end{align*}
    for $t \geq 0$. Again, the term inside the Brownian integral is an $L^2$-integrand and the corresponding expectation vanishes. Thus, by Young's inequality there exists $C^{(4)}_\beta > 0$ satisfying
    \begin{align}
    \label{___D2yPhi_8}
    \begin{split}
        &\frac{d}{dt} \hat \E \bigg[\int_0^t \frac{2\langle Z_s,dZ_s \rangle +d[Z]_s}{2(\rho_s+\gamma^2)^{1/2}}  \bigg] 
        \\
        &\leq
        -\beta \hat \E \bigg[ \frac{|Z_t|^2}{(\rho_t + \gamma^2)^{1/2}} \bigg] 
        + (C^{(2)} + d_2  (C^{(2)})^2) \hat \E \bigg[ \frac{|Z_t||\widetilde Z_t||\zeta_t|}{(\rho_t + \gamma^2)^{1/2}} \bigg] 
        + d_2 (C^{(2)})^2 \hat \E \bigg[ \frac{|\widetilde Z_t|^2|\zeta_t|^2}{(\rho_t + \gamma^2)^{1/2}} \bigg] 
        \\
        &\leq -\frac{\beta}{2} \hat \E \bigg[ \frac{|Z_t|^2}{(\rho_t + \gamma^2)^{1/2}} \bigg] + C^{(4)}_\beta \hat \E \bigg[ \frac{|\widetilde Z_t|^2|\zeta_t|^2}{(\rho_t + \gamma^2)^{1/2}} \bigg] \\
        &\leq -\frac{\beta}{2} \hat \E \bigg[ \frac{|Z_t|^2}{(\rho_t + \gamma^2)^{1/2}} \bigg] + C^{(4)}_\beta \| \widetilde Z_t \|_{L^2} \| \zeta_t\|_{L^2}.
    \end{split}
    \end{align}
    As the right-hand-side of \eqref{___D2yPhi_4} is continuously differentiable, $t \mapsto \hat \E [(\rho_t+\gamma^2)^{1/2}]$ is too. Thus, setting $C^{(5)}_\beta := C^{(3)} + C^{(4)}_\beta$ and substituting all terms by \eqref{___D2yPhi_4.5}, \eqref{___D2yPhi_5}, \eqref{___D2yPhi_6}, \eqref{___D2yPhi_7}, and \eqref{___D2yPhi_8} yields
    \begin{align*}
        \frac{d}{dt} \hat \E [(\rho_t+\gamma^2)^{1/2}] &\leq - \frac{\beta}{2} \hat \E \bigg[ \frac{|Z_t|^2}{(\rho_t+\gamma^2)^{1/2}} \bigg] -2\beta \hat \E \bigg[ \frac{|\zeta_t|^2|\widetilde Z_t|^2}{(\rho_t+\gamma^2)^{1/2}} \bigg] +C^{(5)}_\beta \|\widetilde Z_t\|_{L^2} \|\zeta_t\|_{L^2} - \frac{dR_t}{dt} \\
        &\leq - \frac{\beta}{2} \hat \E \bigg[ \frac{\rho_t+\gamma^2}{(\rho_t+\gamma^2)^{1/2}}\bigg] + \frac{\beta }{2}  \hat \E \bigg[ \frac{\gamma^2}{(\rho_t+\gamma^2)^{1/2}}\bigg] + C^{(5)}_\beta \|\widetilde Z_t\|_{L^2} \|\zeta_t\|_{L^2}\\
        &\leq -\frac{\beta}{2}\hat \E [(\rho_t+\gamma^2)^{1/2}] + \frac{\beta \gamma} {2} + C^{(5)}_\beta \|\widetilde Z_t\|_{L^2} \|\zeta_t\|_{L^2}
    \end{align*}
    for $t \geq 0$. Applying Grönwall's differential lemma and Lemmas~\ref{lem:frozenmoment}, \ref{___lem_DyFlow}, we have  for a constant $C^{(6)}_\beta >0$ that is independent of $\gamma > 0$,
    \begin{align*}
        \hat \E [(\rho_t+\gamma^2)^{1/2}] 
        &\leq e^{-(\beta / 2)t}(h^2 |v_1|^2|v_2|^2 + \gamma^2)^{1/2}  + \int_0^t e^{-(\beta/2) (t-s)}\Big(\frac{\beta \gamma}{2} + C^{(5)}_{\beta} e^{-2\beta s}|h||v_1||v_2| \Big) ds \\
        &\leq C^{(6)}_{\beta}e^{-(\beta / 2)t} |h||v_1||v_2| + \gamma 
    \end{align*}
    for $t \geq 0$. As the choice of $\gamma >0$ was arbitrary, we finally obtain
    \begin{align*}
        \bigg\| \frac{D_y \hat{Y}^{x,y+hv_2}_t v_1 -D_y \hat{Y}^{x,y}_t v_1 }{h} \bigg\|_{L^1} \leq \liminf_{\gamma \to 0} \frac{\hat \E [(\rho_t + \gamma^2)^{1/2}]}{|h|} \leq C^{(6)}_{\beta}e^{-(\beta / 2)t} |v_1||v_2| 
    \end{align*} 
    and letting $h \to 0$, we have by Lemma \ref{___lem_DyFlow}
    \begin{align*}
        \| D^2_y \hat{Y}^{x,y}_t (v_1,v_2) \|_{L^1} \leq C^{(6)}_{\beta}e^{-(\beta / 2)t} |v_1||v_2| 
    \end{align*} 
    for each $t \geq 0$. 
    Set $C^{(7)}_\beta := C^{(2)} C^{(6)}_\beta +C^{(2)}$. Applying the preceding bounds to \eqref{___D2yPhi_1} and  \eqref{___D2yPhi_3}, respectively, yields
    \begin{align}\label{___D2yPhi_9}
    \begin{split}
        &\bigg| \frac{ D_y F(t,x,y+hv_2)v_1 - D_y F(t,x,y)v_1}{h} \bigg| 
        \\
        & \leq C^{(2)} \bigg\| \frac{D_y \hat{Y}^{x,y+hv_2}_t v_1 -D_y \hat{Y}^{x,y}_t v_1 }{h} \bigg\|_{L^1} + C^{(2)} \bigg\| \frac{\hat{Y}^{x,y+hv_2}_t - \hat{Y}^{x,y}_t}{h} \bigg\|_{L^2} \|D_y \hat Y^{x,y}_t v_1\|_{L^2}\\
        & \leq C^{(7)}_\beta e^{-(\beta /2)t} |v_1| |v_2|
    \end{split}
    \end{align}
    as well as 
    \begin{align}\label{___D2yPhi_9.5}
    \begin{split}
        |D^2_y F(t,x,y)(v_1,v_2)| &\leq C^{(2)} \|D_y^2\hat{Y}^{x,y}_t(v_1,v_2)\|_{L^1} + C^{(2)} \|D_y \hat{Y}^{x,y}_t v_1\|_{L^2}\|D_y \hat{Y}^{x,y}_t v_2\|_{L^2}
        \\
        &\leq C^{(7)}_\beta e^{-(\beta /2)t}|v_1||v_2|
    \end{split}
    \end{align}
    for $x \in \R^{d_1}$, $y ,v_1,v_2 \in \R^{d_2}$, $t \geq 0$, and $h \in (-1,1) \setminus \{0\}$.
    By the theorem of dominated convergence, \eqref{___D2yPhi_9} implies the existence of 
    \begin{align*}
        D_y[D_y \Phi(x,\cdot)v_1](y)v_2 = \lim\limits_{h \to 0} \frac{D_y \Phi(x,y +hv_2)v_1 - D_y \Phi(x,y)v_1}{h} = -\int_0^\infty D^2_y F(t,x,y)(v_1,v_2) dt
    \end{align*}
    while \eqref{___D2yPhi_9.5} implies the continuity of the expression on the left in $x,y$ such that the second order total derivative
    \begin{align*}
        D_y^2 \Phi(x,y) (v_1,v_2) = D_y^2 \Phi(x,y) (v_2,v_1)=D_y[D_y \Phi(x,\cdot)v_1](y)v_2
    \end{align*}
    exists for all $x \in \R^{d_1}$, $y,v_1,v_2 \in \R^{d_2}$. Finally, \eqref{___D2yPhi_9.5} also gives the bound
    \begin{align*}
        |D_y^2 \Phi(x,y) (v_1,v_2)| \leq \frac{2 C^{(7)}_{\beta}}{\beta} |v_1||v_2|
    \end{align*}
    for all $x\in\R^{d_1}$ and $y,v_1,v_2 \in \R^{d_2}$.
\end{proof}

\begin{lemma} \label{lem:DxPhi}
Let $\beta>0$ and assume that $(\hat Y^x_t)_{t \ge 0}$ is $2$-dissipative with constant $\beta$ and $f$ is $L_f$-Lipschitz. Under Assumption~\ref{assu:regularity}, $\Phi$ is differentiable in $x$, its derivative is jointly continuous on $\R^{d_1+d_2}$, and for all $\alpha>0$, there exists a constant $C_{\alpha,\beta} >0$ satisfying
    \begin{align*}
        |D_x \Phi (x,y)| \leq C_{\alpha,\beta} (1+|y|^\alpha)
    \end{align*}
    for all $x \in \R^{d_1}$ and $y \in \R^{d_2}$.
\end{lemma}

\begin{proof}  
Clearly, it suffices to prove the statement for sufficiently small $\alpha$.
Define $F$ as in \eqref{___DyPhi_F_def}. We will show differentiability w.r.t. $x$ for $F$, $\bar{f}$, and $\Phi$ in that order.

Given an arbitrary but fixed $y \in \R^{d_2}$, it has already been established in \eqref{eq:weg9h34t79h8gd} that $F(t,\cdot,y)$ converges to $\bar{f}$ for $t \to \infty$ under the supremum norm. To prove differentiability of $F(t,\cdot,y)-\bar{f}$, we will show that $F$ is continuously differentiable in $x \in \R^{d_1}$ and 
\begin{align} \label{eq:24fh39fh9222}
    \sup_{s \ge t} \|F(t,\cdot,y)-F(s,\cdot,y)\|_{C^1_b} \overset{t \to \infty}{\longrightarrow} 0,
\end{align}
where for a bounded function $h \colon \R^{d_1} \to \R^{d_1}$ with bounded continuous first derivatives we define
\begin{align*}
    \|\varphi\|_{C^1_b} := |\varphi|_{\sup} + \sup_{x\in \R^{d_1}}|D\varphi(x)|.
\end{align*}
Then by completeness of the space $C_b^1$ and by $\bar{f}$ being the only possible cluster point of the sequence, we obtain that $\bar f$ is continuously differentiable. Finally, if the convergence in \eqref{eq:24fh39fh9222} is exponential, the derivative of $\Phi$ in $x$ also exists by the dominated convergence theorem. 

Regarding the differentiability of $F$, one has
\begin{align}\label{___DxPhi_1}
\begin{split}
    &\frac{F(t,x+hv,y) - F(t,x,y)}{h} 
    \\
    &= \hat \E \bigg[ \frac{f(x+hv,\hat{Y}^{x+hv,y}_t)-f(x,\hat{Y}^{x,y}_t)}{h} \bigg] 
    \\
    &= \hat \E \bigg[ \frac{f(x+hv,\hat{Y}^{x+hv,y}_t) - f(x+hv,\hat{Y}^{x,y}_t)}{h} \bigg] + \hat \E \bigg[ \frac{f(x+hv,\hat{Y}^{x,y}_t) - f(x,\hat{Y}^{x,y}_t)}{h} \bigg]
\end{split}
\end{align}
for $t \geq 0$, $h\in(-1,1) \setminus\{0\}$, $x,v \in \R^{d_1}$, and $y \in \R^{d_2}$.
Now, fix $t \geq 0$, $x,v \in \R^{d_1}$, and $y \in \R^{d_2}$. Then the term inside the latter expectation has an almost sure bound $L_f |v|$ that is uniform in $h$ so that the dominated convergence theorem gives
\begin{align}\label{___DxPhi_2}
    \bigg\| \frac{f(x+hv,\hat{Y}^{x,y}_t) - f(x,\hat{Y}^{x,y}_t)}{h} - D_x f(x,\hat{Y}^{x,y}_t)v \bigg\|_{L^1} \overset{h \to 0}{\longrightarrow} 0.
\end{align}
The convergence of the former expectation is shown with the same argument used to prove Lemmas~\ref{___lem_DyPhi} and \ref{___lem_D2yPhi}. By the mean value theorem
\begin{align*}
    f(x',y_2) - f(x',y_1)=
    \int_0^1 D_y f(x', y_1 +r(y_2-y_1)) (y_2-y_1) dr, \quad x'\in\R^{d_1},\ y_1,y_2 \in \R^{d_2},
\end{align*}
we obtain
\begin{align*}
    &\bigg\| \frac{f(x+hv,\hat{Y}^{x+hv,y}_t) - f(x+hv,\hat{Y}^{x,y}_t)}{h} - D_y f(x,\hat Y^{x,y}_t) D_x \hat Y^{x,y}_t v \bigg\|_{L^1} 
    \\
    &\le \int_0^1 \bigg\| \Big(D_y f(x+hv, \hat{Y}^{x,y}_t +r(\hat{Y}^{x+hv,y}_t-\hat{Y}^{x,y}_t))-D_y f(x,\hat Y^{x,y}_t)\Big)\frac{\hat{Y}^{x+hv,y}_t-\hat{Y}^{x,y}_t}{h} \bigg\|_{L^1} dr 
    \\
    &\qquad + \bigg\| D_y f(x,\hat Y^{x,y}_t) \bigg( \frac{\hat{Y}^{x+hv,y}_t-\hat{Y}^{x,y}_t}{h} - D_x \hat Y^{x,y}_t v \bigg) \bigg\|_{L^1}
\end{align*}
for $h \in (-1,1)\setminus\{0\}$. 
Using a bound analogous to the one derived in \eqref{___DyPhi_0}, Lemmas~\ref{___lem_DxFlow} and \ref{lem:frozenmoment} give
\begin{align}\label{___DxPhi_3}
    \bigg\| \frac{f(x+hv,\hat{Y}^{x+hv,y}_t) - f(x+hv,\hat{Y}^{x,y}_t)}{h} - D_y f(x,\hat Y^{x,y}_t) D_x \hat Y^{x,y}_t v \bigg\|_{L^1} \overset{h \to 0}{\longrightarrow} 0.
\end{align}
Hence, combining \eqref{___DxPhi_1}, \eqref{___DxPhi_2}, and \eqref{___DxPhi_3} yields the existence of
\begin{align}\label{___DxPhi_4}
\begin{split}
     D_x F(t,x,y) v &= \lim\limits_{h\to 0} \frac{F(t,x+hv,y) - F(t,x,y)}{h} 
     \\
     &=\hat \E [D_x f(x,\hat Y^{x,y}_t)v ] + \hat \E[ D_y f(x, \hat Y^{x,y}_t ) \, D_x \hat Y^{x,y}_t v]
\end{split}
\end{align}
for all $t \geq 0$, $x,v\in\R^{d_1}$, and $y \in \R^{d_2}$.

By the boundedness of the derivatives, Assumption~\ref{assu:regularity}, and Lemma~\ref{___lem_DxFlow}, the terms in the expectations in \eqref{___DxPhi_4} admit a uniform $L^2$-bound that is uniform in $x,y$. As they are also continuous in $(x,y) \in \R^{d_1+d_2}$ with respect to convergence in probability, Vitali's convergence theorem yields for each $t \geq 0$ and $v \in \R^{d_1}$,
\begin{align*}
    &\|D_x f(x',\hat Y^{x',y'}_t)v - D_x f(x,\hat Y^{x,y}_t)v\|_{L^1} \\
    &+ \|D_y f(x', \hat Y^{x',y'}_t ) D_x \hat Y^{x',y'}_t v-D_y f(x, \hat Y^{x,y}_t ) D_x \hat Y^{x,y}_t v \|_{L^1} \longrightarrow 0
\end{align*}
for all $(x,y) \in \R^{d_1+d_2}$ and thus continuity of $D_x F(t,\cdot,\cdot)v$.

Next, we show exponential convergence in \eqref{eq:24fh39fh9222} using the expression in \eqref{___DxPhi_4}.
Let $C^{(1)} >0$ be the constant bounding all required derivatives of the coefficients, and set $C^{(2)} :=2C^{(1)}$. Then the first order derivatives are Hölder-continuous for arbitrary $\alpha \in [0,1]$ by
\begin{align*}
    |\phi(x,y_1) - \phi(x,y_2)| 
    &\leq (|\phi(x,y_1)| + |\phi(x,y_2)|)^{1-\alpha} |\phi(x,y_1) - \phi(x,y_2)|^{\alpha} \\
    &\leq ( 2C^{(1)} )^{1-\alpha} (C^{(1)})^\alpha |y_1-y_2|^{\alpha} \\
    &\leq C^{(2)} |y_1-y_2|^{\alpha}, \qquad x \in \R^{d_1},\ y_1,y_2 \in \R^{d_2}
\end{align*}
for any $\phi = D_xf,\ D_y f,\ D_x g,\ D_y g,\ D_x \sigma_k,\ D_y \sigma_k$, $k = 1,\dots,d_2$. 
Now, for the first term on the right-hand side in \eqref{___DxPhi_4} we get for $ t\geq 0$, $x,v\in \R^{d_1}$, $y \in \R^{d_2}$, and $\alpha \in (0,1]$
\begin{align}\label{___DxPhi_5}
\begin{split}
    &\bigg|\hat{\E}[D_xf(x,\hat{Y}^{x,y}_t)v] - \int D_xf(x,y')v \, \mu_x(dy') \bigg|\\
    &\leq \int \hat{\E} \Big[| D_xf(x,\hat{Y}^{x,y}_t)v-D_xf(x,\hat{Y}^{x,y'}_t)v |\Big] \, \mu_x(dy') \\
    & \leq C^{(2)}|v|\int\hat{\E}[|\hat{Y}^{x,y}_t-\hat{Y}^{x,y'}_t|^{\alpha}] \, \mu_x(dy') \\
    & \leq C^{(2)} |v| e^{-\beta \alpha t} \bigg(|y|+\int |y'| \,  \mu_x(dy')\bigg)^{\alpha} \\
    & \leq C^{(3)}_\beta e^{-\beta \alpha t} |v| (1+|y|^{\alpha}),
\end{split}
\end{align}
 where we have used Lemma~\ref{lem:frozenmoment} and the constant $C^{(3)}_\beta >0$ is due to Lemma~\ref{lem:invariantmoment}. 

In order to show the convergence of the second component, we introduce for each $x,v\in\R^{d_1}$ the Markov semigroup $(Q^{x,v}_t)_{t \geq 0}$ associated to the joint dynamics of $(\hat{Y}^{x,y}, \zeta^{x,y,z,v})$ satisfying
\begin{align*}
    \begin{cases}
        d \hat{Y}^{x,y}_t = g(x,\hat{Y}^{x,y}_t)dt + \sigma(x,\hat{Y}^{x,y}_t) d\hat{W}_t
        \\
        d \zeta^{x,y,z,v}_t = D_x g(x,\hat{Y}^{x,y}_t) v dt + D_y g(x,\hat{Y}^{x,y}_t) \zeta^{x,y,z,v}_t dt \\ \qquad \qquad \qquad  + \sum_{k=1}^{d_2} \Big( D_x \sigma_k(x,\hat{Y}^{x,y}_t) v + D_y \sigma_k(x,\hat{Y}^{x,y}_t) \zeta^{x,y,z,v}_t \Big) d \hat{W}^k_t
    \end{cases}
\end{align*}
and starting in $(\hat{Y}^{x,y}_0, \zeta^{x,y,z,v}_0) = (y,z) \in \R^{d_2+d_2}$. We will prove its exponential stability. Let $x,v\in \R^{d_1}$, $y_1,y_2,z_1,z_2 \in \R^{d_2}$ and define
\begin{align*}
     Z_t := \zeta^{x,y_1,z_1,v}_t-\zeta^{x,y_2,z_2,v}_t, \quad \widetilde Z_t := \hat Y^{x,y_1}_t - \hat Y^{x,y_2}_t, \quad \zeta_t := \zeta^{x,y_2,z_2,v}_t
\end{align*} 
for $t \geq 0$.
By application of Itô's formula
{\allowdisplaybreaks[4]
\begin{align}
\hat{\E} [|Z_t|^2]
&=
|z_1-z_2|^2 + 2\int_0^t \hat{\E} [\langle Z_s, D_x g(x,\hat Y^{x,y_1}_s)v-D_x g(x,\hat Y^{x,y_2}_s)v \rangle] ds \notag\\
&\qquad+
2\int_0^t \hat{\E} [\langle Z_s, D_y g(x,\hat Y^{x,y_1}_s)\zeta^{x,y_1,z_1,v}_s -D_y g(x,\hat Y^{x,y_2}_s) \zeta^{x,y_2,z_2,v}_s\rangle] ds \notag\\
&\qquad+
\sum_{k=1}^{d_2} \int_0^t \hat{\E} [| D_x \sigma_k(x,\hat Y^{x,y_1}_s)v-D_x \sigma_k(x,\hat Y^{x,y_2}_s)v \notag\\
&\qquad\qquad
+
D_y \sigma_k(x,\hat Y^{x,y_1}_s)\zeta^{x,y_1,z_1,v}_s -D_y \sigma_k(x,\hat Y^{x,y_2}_s)\zeta^{x,y_2,z_2,v}_s |^2] ds \notag\\
&\qquad+2\sum_{k=1}^{d_2}\hat{\E} \bigg[ \int_0^t \langle Z_s,D_x \sigma_k(x,\hat Y^{x,y_1}_s)v-D_x \sigma_k(x,\hat Y^{x,y_2}_s)v \notag\\
&\qquad \qquad+
    D_y \sigma_k(x,\hat Y^{x,y_1}_s)\zeta^{x,y_1,z_1,v}_s 
   -D_y \sigma_k(x,\hat Y^{x,y_2}_s)\zeta^{x,y_2,z_2,v}_s \rangle d \hat W^k_s \bigg]  \notag\\
&=
|z_1-z_2|^2 + 
\int_0^t \hat{\E} \bigg[2\langle Z_s, D_y g(x,\hat Y^{x,y_1}_s) Z_s \rangle + \sum_{k=1}^{d_2}  |D_y \sigma_k(x,\hat Y^{x,y_1}_s) Z_s|^2 \bigg]ds \label{eq:34f934hf39gh9333} \\
&\qquad+
2\int_0^t \hat{\E} [\langle Z_s, D_x g(x,\hat Y^{x,y_1}_s)v-D_x g(x,\hat Y^{x,y_2}_s)v \rangle] ds \notag\\
&\qquad+
2\int_0^t \hat{\E} [\langle Z_s, ( D_y g(x,\hat Y^{x,y_1}_s)-D_y g(x,\hat Y^{x,y_2}_s) ) \zeta_s\rangle] ds \notag\\
&\qquad + 2\sum_{k=1}^{d_2} \int_0^t \hat{\E}[\langle D_y \sigma_k(x,\hat Y^{x,y_1}_s)Z_s , D_x \sigma_k(x,\hat Y^{x,y_1}_s)v-D_x \sigma_k(x,\hat Y^{x,y_2}_s)v \notag\\
&\qquad \qquad \qquad + (D_y \sigma_k(x,\hat Y^{x,y_1}_s)-D_y \sigma_k(x,\hat Y^{x,y_2}_s))\zeta_s \rangle ]ds \notag\\
&\qquad + \sum_{k=1}^{d_2}\int_0^t \hat{\E}[| D_x \sigma_k(x,\hat Y^{x,y_1}_s)v-D_x \sigma_k(x,\hat Y^{x,y_2}_s)v \notag\\
&\qquad \qquad \qquad + (D_y \sigma_k(x,\hat Y^{x,y_1}_s)-D_y \sigma_k(x,\hat Y^{x,y_2}_s))\zeta_s |^2]ds \notag
\end{align}
}
for $t \geq 0$ where the Brownian integrals are true martingales with expectation $0$ as they are applied to $L^2$-integrands. 
Hence, we can take the derivative and by Young's and Hölder's inequalities there exists $C^{(4)}_\beta > 0$ such that
\begin{align*}
    \frac{d}{dt} \hat{\E} [|Z_t|^2]
    &\leq -2\beta \hat{\E} [|Z_t|^2] + (2 C^{(2)} +2 d_2 C^{(1)}C^{(2)}) |v|\hat{\E} [|Z_t| \, |\widetilde Z_t|^{\alpha}] \\
    &\qquad + (2 C^{(2)} +2 d_2 C^{(1)}C^{(2)}) \hat{\E} [ |Z_t| \,  |\widetilde Z_t|^{\alpha} \, |\zeta_t| ] \\
    &\qquad + 2 d_2 (C^{(2)})^{2} |v|^2 \hat{\E}[|\widetilde Z_t|^{2\alpha}] + 2 d_2 (C^{(2)})^2 \hat{\E}[|\widetilde Z_t|^{2\alpha} \, |\zeta_t|^2] \\
    &\leq -\beta \hat{\E} [|Z_t|^2] + C^{(4)}_{\beta} \hat{\E}[|\widetilde Z_t|^{2\alpha}] \, |v|^2  + C^{(4)}_{\beta} \hat{\E}[|\widetilde Z_t|^{2\alpha} \, |\zeta_t |^2]  \\
    &\leq - \beta \hat{\E} [|Z_t|^2] + C^{(4)}_{\beta} \|\widetilde Z_t\|_{L^{2}}^{2\alpha} \,  |v|^2 + C^{(4)}_{\beta} \| \widetilde Z_t \|_{L^{2+2\alpha}}^{2\alpha} \,  \|\zeta_t \|_{L^{2+2\alpha}}^2 
\end{align*}
holds for $t \geq 0$ and $\alpha \in (0,1)$. 
By Remark~\ref{rem:dissibigger}, there exists $\alpha^* \in (0,1)$ such that $\hat Y^x$ remains $p$-dissipative for some $p =2+2\alpha >2$ with the coefficient $\beta /2$ for all $\alpha \in (0,\alpha^*]$ so that
\begin{align*}
    \|\widetilde Z_t\|_{L^{2}} \leq \| \widetilde Z_t \|_{L^{2+2\alpha}} \leq e^{-(\beta/2) t} |y_1-y_2|
\end{align*}
for all $t \geq 0$ by Lemma \ref{lem:frozenmoment} and so that there exists an $\alpha$-dependent constant $C^{(5)}_{\alpha ,\beta} \geq 1$
with
\begin{align*}
    \|\zeta_t \|_{L^{2+2\alpha}} 
    &\leq (e^{-(2+2\alpha)(\beta/4)t}|z_2|^{2+2\alpha} + C^{(5)}_{\alpha ,\beta}|v|^{2+2\alpha})^{1/(2+2\alpha)} \leq (|z_2|^2 + C^{(5)}_{\alpha ,\beta} |v|^2)^{1/2}
\end{align*}
for all $t \geq 0$ by Lemma \ref{___lem_DxFlow}.
Hence, by substituting these inequalities into \eqref{eq:34f934hf39gh9333}, one obtains by Grönwall's differential lemma a constant $C^{(6)}_{\alpha,\beta} \geq 1$ depending on $\alpha$ with
\begin{align*}
    &\hat{\E}[|\zeta^{x,y_1,z_1,v}_t-\zeta^{x,y_2,z_2,v}_t|^2]\\ 
    &\leq e^{-\beta t}|z_1-z_2|^2 +C^{(4)}_\beta|y_1-y_2|^{2 \alpha} ((1+C^{(5)}_{\alpha,\beta})|v|^2+|z_2|^2) \int_0^t e^{-\beta (t-s)} e^{-\beta \alpha s} ds \\
    &\leq C^{(6)}_{\alpha,\beta} e^{-\beta \alpha t}(|z_1-z_2|^2 +|y_1-y_2|^{2 \alpha} (|v|^2+|z_2|^2) )
\end{align*}
for all $ t \geq 0$ and for each $\alpha \in (0,\alpha^*]$.
Next, consider the family of functions
\begin{align*}
    \varphi^x \colon \R^{d_2 + d_2}\to \R^{d_1},\ \varphi^x(y,z) = D_y f(x,y)z, \qquad x\in\R^{d_1}.
\end{align*}
Then for each $\alpha \in (0,\alpha^*]$ there exists $C^{(7)}_{\alpha, \beta} >0$ with
\begin{align*}
    &|Q^{x,v}_t \varphi^x(y_1,z_1) - Q^{x,v}_t \varphi^x(y_2,z_2)| \\
    &\leq \hat{\E}[|D_yf(x,\hat{Y}^{x,y_1}_t)\zeta^{x,y_1,z_1,v}_t - D_yf(x,\hat{Y}^{x,y_2}_t)\zeta^{x,y_2,z_2,v}_t|] \\
    &\leq \hat{\E}[|D_yf(x,\hat{Y}^{x,y_1}_t)||\zeta^{x,y_1,z_1,v}_t -\zeta^{x,y_2,z_2,v}_t|]+\hat{\E}[|D_yf(x,\hat{Y}^{x,y_1}_t) - D_yf(x,\hat{Y}^{x,y_2}_t)||\zeta^{x,y_2,z_2,v}_t|]  \\
    &\leq C^{(1)}\|\zeta^{x,y_1,z_1,v}_t -\zeta^{x,y_2,z_2,v}_t\|_{L^1} + C^{(2)}\|\hat Y^{x,y_1}_t-\hat Y^{x,y_2}_t\|_{L^2}^{\alpha} \|\zeta^{x,y_2,z_2,v}_t\|_{L^2} \\
    &\leq C^{(1)} C^{(6)}_{\alpha,\beta} e^{-(\beta/2) \alpha t}(|z_1-z_2| +|y_1-y_2|^{\alpha} (|v|+|z_2|) ) + C^{(2)}  e^{-\beta \alpha t} |y_1-y_2|^{\alpha}(|z_2|+C^{(5)}_{\alpha, \beta} |v|) \\
    &\leq C^{(7)}_{\alpha,\beta} e^{-(\beta/2)\alpha t} (|z_1-z_2|+ |y_1-y_2|^{\alpha}(|v|+|z_2|))
\end{align*}
for all $t \geq 0$. It follows that for each $v \in \R^{d_1}$, $y\in \R^{d_2}$, $\alpha \in (0,\alpha^*]$, the maps \begin{align*}
\R^{d_1} \to \R^{d_1},\ x \mapsto \hat{\E}[D_y f(x,\hat{Y}^{x,y}_t) D_x \hat{Y}^{x,y}_t v] = Q^{x,v}_t \varphi^x(y,0)
\end{align*}
is a Cauchy sequence indexed with $t \geq 0$
since
\begin{align*}
\begin{split}
    &\sup_{x\in\R^{d_1}}|\hat{\E}[D_y f(x,\hat{Y}^{x,y}_{t+t_0}) D_x \hat{Y}^{x,y}_{t+t_0} v] - \hat{\E}[D_y f(x,\hat{Y}^{x,y}_t) D_x \hat{Y}^{x,y}_t v]|\\
    &\leq \sup_{x\in\R^{d_1}}\hat{\E}[|Q^{x,v}_t \varphi^x(\hat{Y}^{x,y}_{t_0},\zeta^{x,y,0,v}_{t_0}) -Q^{x,v}_t \varphi^x(y,0)|] \\
    &\leq C^{(7)}_{\alpha,\beta} e^{-(\beta/2)\alpha t} \sup_{x\in\R^{d_1}}\hat{\E}[|\zeta^{x,y,0,v}_{t_0}|+ |\hat{Y}^{x,y}_{t_0} -y|^{\alpha}|v|]\\
    &\leq C^{(8)}_{\alpha,\beta} e^{-(\beta/2)\alpha t} |v|(1+|y|^{\alpha})
\end{split}
\end{align*}
for all $t, t_0 \geq 0$. By completeness of the space of bounded continuous functions this sequence is also exponentially convergent with
\begin{align}\label{___DxPhi_6}
\begin{split}
    &\sup_{x\in\R^{d_1}}| \lim\limits_{t_0 \to \infty} \hat{\E}[D_y f(x,\hat{Y}^{x,y}_{t_0}) D_x \hat{Y}^{x,y}_{t_0} v] - \hat{\E}[D_y f(x,\hat{Y}^{x,y}_t) D_x \hat{Y}^{x,y}_t v]|\\
    &\leq \sup_{x\in\R^{d_1}} \Big( \lim_{t_0 \to \infty} | \hat{\E}[D_y f(x,\hat{Y}^{x,y}_{t+t_0}) D_x \hat{Y}^{x,y}_{t+t_0} v] -\hat{\E}[D_y f(x,\hat{Y}^{x,y}_t) D_x \hat{Y}^{x,y}_t v]| \Big) \\
    &\leq C^{(8)}_{\alpha,\beta} e^{-(\beta/2)\alpha t} |v|(1+|y|^{\alpha}) 
\end{split}
\end{align}
for $t \geq 0$.
Since $(F(t,\cdot,y))_{t \geq 0}$ is a Cauchy sequence with respect to $|\cdot|_{C^1_b}$ and $C^1_b$ is complete, a bounded limiting function $\bar{F}_y(\cdot)$ in $C^1_b$ exists for each $y \in \R^{d_2}$. However by \eqref{eq:weg9h34t79h8gd},
\begin{align*}
    |\bar{f}-F(t,\cdot,y)|_\infty \longrightarrow 0, \qquad t \to \infty,
\end{align*}
which implies $\bar{f}=\bar{F}_y(\cdot) \in C^1_b$ for any $y \in \R^{d_2}$ and in particular, the continuous differentiability of $\bar f$. 

By combining \eqref{___DxPhi_5} and \eqref{___DxPhi_6}, their derivatives vanish exponentially by
\begin{align}\label{___DxPhi_7}
\begin{split}
    \bigg| \frac{(\bar f (x+hv)-F(t,x+hv,y)) - (\bar f(x)-F(t,x,y) )}{h} \bigg|
    &\leq \sup_{x'\in\R^{d_1}}| D_x \bar{f}(x')v- D_x F(t,x',y)v | \\
    &\leq (C^{(3)}_\beta + C^{(8)}_{\alpha,\beta} )e^{-(\beta/2)\alpha t}|v|(1+|y|^{\alpha})
\end{split}
\end{align}
for $\alpha \in (0,\alpha^*]$, $t \geq 0$, $x,v \in \R^{d_1}$, and $y \in \R^{d_2}$. The uniform boundedness in $h$ implies differentiability of $\Phi$ in $x$ by the theorem of dominated convergence with the following bound
\begin{align*}
    |D_x \Phi(x,y)v|
    &\leq \int_0^\infty |D_x \bar{f}(x)v - D_x F(t,x,y)v| dt\\
    &\leq (C^{(3)}_\beta + C^{(8)}_{\alpha,\beta} ) |v|(1+|y|^{\alpha}) \int_0^\infty e^{-(\beta/2)\alpha t} dt\\
    &\leq C^{(9)}_{\alpha,\beta} |v|(1+|y|^{\alpha}).
\end{align*}
for each $\alpha \in (0,\alpha^*]$.
At last, \eqref{___DxPhi_7} yields an integrable upper bound for each neighborhood of any pair $(x,y) \in \R^{d_1+d_2}$ so that $D_x \Phi$ is continuous, again by the dominated convergence theorem.
\end{proof}

\section*{Acknowledgments} The authors acknowledge funding by the Deutsche Forschungsgemeinschaft (DFG, German Research Foundation) – CRC/TRR 388 ''Rough Analysis, Stochastic Dynamics and Related Fields'' – Project ID 516748464.

\bibliographystyle{alpha}
\bibliography{literature}
		
\end{document}